\documentclass[a4paper,11pt]{article}

 \usepackage[normalem]{ulem}
\usepackage[english]{babel}
\usepackage{amsmath}
  \usepackage{amsthm}
  \usepackage{bbm}
 \usepackage{amssymb}
  \usepackage{fullpage}
 \usepackage{graphicx}
 \usepackage{color}
 \usepackage{epsfig} 
 \usepackage[only,llbracket,rrbracket]{stmaryrd}
 \usepackage{enumerate}

\newtheorem{thm}{Theorem}[section]
\newtheorem{deft}{Definition}[section]

\newtheorem{prop}{Proposition}[section]
\newtheorem{lem}{Lemma}[section]
\newtheorem{coro}{Corollary}[section]

\numberwithin{equation}{section}

{\vskip 0.5cm}

\newtheorem{rque}{\textbf{Remark}}[section]{\vskip 0.5cm} 
{\vskip 0.5cm} 

\let\ds\displaystyle

\newcommand{\dds}{{\frac d {ds}}}
\newcommand{\T}{{\mathbb T}}
\newcommand{\C}{{\mathbb C}}
\newcommand{\Z}{{\mathbb Z}}
\newcommand{\R}{{\mathbb R}}
\newcommand{\N}{{\mathbb N}}
\newcommand{\ep}{\varepsilon}
\newcommand{\eps}{\varepsilon}
\renewcommand{\epsilon}{\varepsilon}

\newcommand{\pa}{\partial}

\DeclareMathOperator{\re}{Re}
\DeclareMathOperator{\ima}{Im}

\newcommand{\QQ}{{\mathcal Q}}
\newcommand{\EE}{{\mathcal E}}
\newcommand{\OO}{{\mathcal O}}
\newcommand{\UU}{{\mathcal U}}
\newcommand{\LL}{{\mathcal L}}

\newcommand{\KK}{{\mathcal K}}
\newcommand{\II}{{\mathcal I}}

\newcommand{\Black}{\color{black}}

\title{Stability issues in the quasineutral limit \\of the one-dimensional Vlasov-Poisson equation}
\author{Daniel Han-Kwan\footnote{CNRS and \'Ecole Polytechnique, Centre de
Math\'ematiques Laurent Schwartz UMR7640,  F91128 Palaiseau cedex, email : { \tt
daniel.han-kwan@math.polytechnique.fr}} \, and 
\, Maxime
Hauray\footnote{Universit\'e d'Aix-Marseille, CNRS 
et \'Ecole Centrale Marseille,  LATP,  F13453 Marseille Cedex, email : { \tt maxime.hauray@univ-amu.fr}}}
\date{}
\begin{document}
 \maketitle

\begin{abstract}
This work is concerned with the quasineutral limit of the one-dimensional Vlasov-Poisson equation, for initial data close to stationary homogeneous profiles. Our 
objective is threefold: first, we provide a proof of the fact that the formal 
limit does not hold for homogeneous profiles that satisfy the Penrose 
\emph{instability} criterion. Second, we prove on the other hand that the limit 
is true for homogeneous profiles that satisfy some monotonicity condition, together with a symmetry condition. We handle the case of well-prepared as well as ill-prepared data. Last, we study a stationary boundary-value problem for the formal limit,
the so-called quasineutral Vlasov equation. We show 
the existence of numerous stationary states, with a lot of freedom in the 
construction (compared to that of BGK waves for Vlasov-Poisson): this illustrates the degeneracy of the limit equation.
\end{abstract}

\tableofcontents

\section{Introduction}

\paragraph{The one-dimensional Vlasov-Poisson equation and its quasineutral limit.}

We study the 
dynamics of electrons in a plasma, in the presence of ions that are assumed to 
be immobile 
and uniformly distributed in space. We assume that this dynamics is described by the Vlasov-Poisson equation.
We introduce a positive parameter $\eps$, defined as the ratio of the 
so-called \emph{Debye length} of the plasma to the typical size of the domain.
Loosely speaking, the Debye length is the typical length of electrostatic 
interaction and in most physical situations, the ratio $\ep$ is  
small:  $\eps\ll1$. For a physically oriented discussion on the 
points mentioned above (and below), we refer to \cite{HK}.

Throughout this paper, we will focus on the one-dimensional and periodic (in space) case. In this framework, we define $f_\eps(t,x,v)$ the 
distribution function of the electrons, for $t \in \R^+$, $x \in \T:= \R/\Z$ and $v \in \R$, so that $f_\eps(t,x,v) \, dv \, dx$ can be interpreted as the probability of finding electrons at time $t$ with position close to $x$ and velocity close to $v$.  We also introduce the electric potential $V_\eps(t,x)$ and the associated electric field $-\pa_x V_\eps(t,x)$.
After nondimensionalization (see \cite{HK}), the rescaled 1D 
Vlasov-Poisson equation reads
\begin{equation}
\label{quasi}
\left\{
    \begin{array}{l}
\ds   \partial_t f_\epsilon + v \, \partial_x f_\epsilon   -\partial_x V_\epsilon \, \partial_v f_\epsilon =0,  \\
\ds   -\epsilon^2\, \partial_x^2 V_\epsilon = \int f_\epsilon dv -1, 
  \end{array}
  \right.
\end{equation}
with
 an initial condition $f_{0,\ep}  \in L^1(\T\times \R)$ such that $f_{0,\ep}\geq 
0$,  $\int 
f_{0,\ep} dvdx =1$.

The energy associated to the system~\eqref{quasi} is the following 
functional
\begin{align} \label{energy}
\EE_\epsilon[f_\ep] &:= \frac12 \int f_\epsilon \vert v \vert^2 dv dx  + \frac{\epsilon^2}{2} \int \vert \partial_x V_\ep[f_\eps] \vert^2 dx \\
& = \frac12 \int f_\epsilon \vert v \vert^2 dv dx + \frac12
\int  V_\ep[f_\eps]  \, \rho_\ep  \,dx \nonumber.
\end{align}

We have used here the notation $V_\eps[f_\eps]$ in order to emphasize that the 
potential $V_\eps$ depends on the distribution function $f_\eps$. We will often 
forget to mention explicitly this dependance in the sequel, in order to 
lighten the notations.

In the following, we consider (most of the time) global strong 
solutions to the system~\eqref{quasi},
with bounded initial energy $\EE_\ep[f_{0,\ep}]$.  
This entails that
the energy $\EE_\ep(t)= \EE_\ep[f_\ep(t)]$ 
is preserved, as so is $\int Q(f_\ep) \,dxdv$, for any continuous function 
$Q$ (and for any $\ep$).
Remark that in this 
one-dimensional case, the strongness assumption is not a huge restriction since 
there is a weak-strong stability principle for solutions whose density $\rho_\ep$ 
 remains bounded in $L^\infty$ for all times~\cite[Theorem 1.9]{HauX}. 
This last property can be ensured for instance if $f_{0,\ep}$ is bounded from above by a profile 
$g_0(|v|)$ that is decreasing, bounded, and integrable. 
\bigskip

In this paper, we shall study the behavior of solutions to 
\eqref{quasi} as $\eps\to0$, a limit that we shall refer to as the 
\emph{quasineutral limit}.

\paragraph{The formal limit : the quasineutral Vlasov equation.} We 
begin with a brief formal analysis of the limit $\epsilon \rightarrow 0$. Let 
us assume that in some sense, we have $f_\epsilon \rightarrow 
f$ and $V_\epsilon  \rightarrow V$. 
The formal limit is straightforward : only the Poisson equation is affected and degenerates into $\rho:= \int f\, dv =1$. The limit system then reads
\begin{equation}
\label{kineuler-intro}
\left\{
    \begin{array}{ll}
  \partial_t f + v \, \partial_x f - \partial_x V \,\partial_v f =0,  \\
\int f dv =1,\\
    \end{array}
  \right.
\end{equation}
which we shall refer to as the \emph{quasineutral Vlasov equation}.

We observe that the total energy associated to this system corresponds only to the kinetic 
part of \eqref{energy}
 \begin{equation}
\label{energy2}
\mathcal{E}[f]= \frac{1}{2} \int f \, \vert v \vert^2 dv dx.
\end{equation} 
 
 The unknown potential $V$  can be seen as a Lagrange multiplier, or a pressure, 
associated to  the ``incompressibility'' constraint $\rho=1$.  
But an explicit equation for $V$ is ``hidden'' in the equation.  
Indeed, if we integrate the transport equation~\eqref{kineuler-intro} with 
respect to $v$, we get the ``zero divergence'' constraint on the current $j$
\begin{equation}
 \partial_x  j =0, \qquad \text{where } \;j(t,x):= \int v f(t,x,v) dv,
\end{equation}
which implies that $j$ is only a function of time. 
Next, we use the equation on the local momentum, obtained after multiplication of 
equation~\eqref{kineuler-intro} by $v$ and integration in $v$
\begin{equation}
\partial_t j +  \partial_x \biggl(\int v^2 f dv  + V \biggr) =0.
\end{equation}
It implies that $\partial_t j$ which is only a function of time, is also a 
gradient in $x$. The only possibility is that $\partial_t j =  \partial_x 
\Bigl(\int v^2 f dv  + V \Bigr)=0$, so that $j$ should be constant in 
time and position (and \emph{a fortiori} at time $0$).

In particular, we get a kind of  ``pressure law'' : the potential $V$ is, up to a 
constant, the opposite of the local kinetic energy 
\begin{equation}
\label{pressure}
V =  - \int v^2 f dv.
\end{equation}

For this reason this  quasineutral  limit can somehow be seen as a 
kinetic version of the 
classical incompressible limit of fluid mechanics (we refer to  Gallagher 
\cite{Gal} for a review on this topic). To go even further into the analogy, 
System \eqref{kineuler-intro} and its higher dimensional dimension generalizations  can be interpreted as the kinetic version of the 
incompressible Euler system, as it has been pointed out by Brenier \cite{Br89}.  
One interesting feature of this system is that it still makes sense in one 
dimension, which is of course not the case for the incompressible Euler 
equation.  For a numerical analysis of the pressure law \eqref{pressure}, we refer 
to~\cite{Degond&all}, where it is used in an attempt to get an asymptotic 
preserving scheme in the quasineutral limit.

\bigskip

To the best of our knowledge, only little is known about the (local) 
well-posedness of the quasineutral Vlasov equation~\eqref{kineuler-intro}. 
The local in time existence of \emph{analytic} solutions is shown in~\cite[Theorem 1.1]{BFJJ}; the case under consideration here corresponds to 
$\beta = \sigma = \alpha =0$ in this reference.
Similar results~\cite{HK,JabNou,BarNou} have also been proved for a related system of equations, 
namely 
\begin{equation}
\label{VDB}
\left\{
    \begin{array}{ll}
  \partial_t f + v\, \partial_x f  - \partial_x V \partial_v f =0,  \\
V= \rho-1,\\
    \end{array}
  \right.
\end{equation}
which was called \emph{Vlasov-Dirac-Benney} by Bardos~\cite{Bardos}. 
The local in time existence in Sobolev spaces of monotonic solutions to~\eqref{VDB}
(precisely, solutions that are for any $x$, increasing and then decreasing in 
$v$) is shown in \cite{Besse,Bardos}. We also refer to the very recent work \cite{BB}. However, at least to our knowledge, such a result is not known for the quasineutral Vlasov equation~\eqref{kineuler-intro}. 
Remark also that~\cite{Bardos} contains 
an interesting argument, that suggests ill-posedness for non monotone (in the above sense) and non analytic initial conditions.

\bigskip
The problem of well-posedness of~\eqref{kineuler-intro} is not the main topic 
of this paper: we rather focus on the justification of the quasineutral limit 
(although the two questions are of course related). Does the asymptotics
\begin{equation}
\label{asymp}
f_\eps(t,x,v) \approx f(t,x,v),
\end{equation}
where $f_\eps$ satisfies~\eqref{quasi} and $f$ 
satisfies~\eqref{kineuler-intro}, holds when $\eps\to 0$ and $f_{0,\eps} \approx f_{|t=0}$?  In this paper, we will restrict ourselves to the particular case where  $f$ is an homogeneous 
equilibrium: $f(t,x,v) = \mu(v)$.  In other words,  we are in particular interested 
in the \emph{the question of stability (or instability) around homogeneous 
equilibria in the quasineutral limit}.
The general case seems much more difficult to handle, but the study of the quasineutral limit around such equilibria already gives an overview of the problems raised by this limit.

\paragraph{Unstable and stable homogeneous equilibria for Vlasov-Poisson.}

Before going on, we shall now recall some well-known facts about the stability or instability of  
homogeneous equilibria (homogeneous meaning that the profile only depends on $v$) for the classical Vlasov-Poisson equation
\begin{equation}
\label{eq:VP1D}
\left\{
    \begin{array}{l}
\ds   \partial_t f + v \, \partial_x f   -\partial_x V \, \partial_v f =0,  \\
\ds   -\, \partial_x^2 V = \int f dv -1.
  \end{array}
  \right.
\end{equation}
The question of the 
linear stability of such profiles is now quite well understood, and some 
important results about the non-linear stability were also proved in the last 
years, culminating in the proof of  nonlinear  Landau damping by 
Mouhot and Villani \cite{MV} (see also the very recent paper by Bedrossian, Masmoudi and Mouhot \cite{BMM}). In the sixties, O. Penrose gave in 
\cite{Pen} a famous criterion for the existence of unstable modes (or generalized eigenvalues) for the 
linearized Vlasov-Poisson equation around an homogeneous profile $\mu(v)$.  
It is related to the existence of instabilities of kinetic nature, often 
referred to as \emph{two stream instabilities} which appear for homogeneous 
distributions in $v$, with two or more maxima. 
In dimension one, it may be stated in the following way.
\begin{deft}
 \label{def:Pen-first}
We say that an homogeneous profile $\mu(v)$, such that $\int \mu \, dv=1$, satisfies the  Penrose 
instability criterion if $\mu$ has a local minimum point $\bar v$ such that 
the following inequality holds 
\begin{equation} \label{def:Pen}
\int_{\R} \frac{ \mu(v) - \mu(\bar v)} {(v-\bar v)^2} \, dv >   0.
\end{equation}
 If the local minimum is flat, i.e. is reached on an interval $[\bar v_1, \bar v_2]$, then~\eqref{def:Pen} has to be satisfied for all $\bar v \in [\bar v_1, \bar v_2]$.
\end{deft}
Taking the regularity and decrease at infinity aside, this criterion is a necessary and sufficient condition 
for the existence of unstable modes in the linearized equation around the homogeneous profile $\mu$ when the Vlasov-Poisson equation \eqref{eq:VP1D} is posed in the whole space $\R$ (for the position variable $x$).

When we are restricted to a torus $\T_M := \R/(M\Z)$ of size $M>0$, then it becomes 
more involved to give a necessary and sufficient condition. But there still exists a rather straightforward necessary criterion: the linearized VP equation~\eqref{eq:VP1D} around 
$\mu$ is unstable only if $\mu$ has a local minimum $\bar v$ such that
\begin{equation} \label{Peneven}
\int_{\R} \frac{ \mu(v) - \mu(\bar v)} {(v-\bar v)^2} \, dv >   \frac{4 
\pi^2}{M^2}.
\end{equation}
It is interesting to remark that the above necessary condition becomes actually 
sufficient if we restricts to profile that are symmetric around $\bar v$, i.e. 
$\mu(2\bar  v -v ) = \mu(v) $ for all $v \in \R$: see for instance \cite[Lemma 
2.1]{GS} for the case $\bar v =0$, $M=2\pi$. In fact, for 
these particular symmetric profiles (and under some smoothness assumptions), Guo 
and Strauss also gave a nonlinear instability result in \cite{GS}.

 We shall see later that the right criterion for our purposes (quasineutral limit in the torus $\T$) turns out to be \eqref{def:Pen} and we shall give more details on this fact in the section~\ref{sec:Penrose}. 

\bigskip

On the contrary, the Penrose criterion suggests that when $\mu$ has no local 
minimum, i.e. $\mu$ is increasing and then decreasing, then the profile may be 
stable  (see also \cite[Section 2.2]{MV}). In fact, it was proved by Marchioro and Pulvirenti 
in dimension one and 
two  \cite{MP86}, and by Batt and Rein in dimension three \cite{BR} that 
radially decreasing profiles are indeed non linearly stable. Their proof 
relies on rearrangement inequalities.

To our knowledge, the only result that proves nonlinear stability without any symmetry assumption
is the work of Mouhot and Villani on Landau damping: 
\cite[Theorem 2.2]{MV} implies some ``orbital'' stability,
but only for Gevrey  perturbations of a  Gevrey  homogeneous profile.

\paragraph{The mathematical difficulties of the quasineutral limit.}

The mathematical analysis is more subtle than what the formal analysis would suggest.
As already mentioned, we focus on the quasineutral limit around 
an unstable homogeneous profile $\mu(v)$, i.e. for initial data converging in some sense to $\mu(v)$.

\bigskip

The role of the stable or unstable nature of $\mu$ in the analysis of the limit was pointed out by Grenier in \cite{Gr99-1}.
Indeed, a major obstruction to the asymptotics arises when a profile $\mu$ satisfies the Penrose instability
criterion of Definition~\ref{def:Pen-first}.
One can remark that such a profile $\mu$ is a stationary solution of \eqref{quasi} and also 
\eqref{kineuler-intro},
with an electric field identically equal to zero, as it is the case for any distribution function depending 
only on $v$.
In the short proceeding note \cite{Gr99-1}, Grenier explains, without giving a proof, why for such profiles, the formal convergence to the expected 
system \eqref{kineuler-intro} is in general false. 
Basically, the idea is that thanks to some scaling invariance of 
the Vlasov-Poisson equation~\eqref{quasi}, the instabilities (whenever they exist)
develop on a very short timescale of order $\ep$. 
 
A consequence is that the linearized quasineutral or Vlasov-Dirac-Benney equations around an 
unstable equilibrium have a unbounded spectrum : it possesses eigenvalues with 
arbitrary high real part. See the analysis in \cite{Bardos,BarNou} for the 
Vlasov-Dirac-Benney case.

\bigskip

Even when we consider stable homogeneous profiles (with one and only one ``hump''), a 
second difficulty arises : it is due to the presence of time oscillations of frequency 
$\epsilon$ and amplitude $\OO(\epsilon^{-1})$ of the electric field, usually 
referred to as \emph{plasma oscillations} or \emph{Langmuir waves}
(see for instance~\cite{Gr96}). We will describe these more 
carefully in Section \ref{sec:heuristics}. Let us just emphasize here that these 
oscillations are not damped: they do carry 
a constant amount of energy, and the problem is very different from an initial 
boundary layer problem. Therefore, these oscillations have an impact on the formal limit.

\paragraph{State of the art.}

One of the first mathematical works on the quasineutral limit of the Vlasov-Poisson equation was performed by Brenier and Grenier in \cite{BG,Gr95}, using the defect measure approach, originally introduced in \cite{DM87} for the Euler equation. In these works, the limit of the two first moments in $v$ is studied. In the limit equations, in addition to the terms one could formally guess, two defect measures appear, which account for the lack of compactness,  together with time oscillations. 
Loosely speaking, it is explained that the defect measures are more or less related to the possible very fast instabilities, while the time oscillations are due to the fast Langmuir waves. 

\bigskip
In a subsequent work \cite{Gr96}, Grenier gives another description of what 
happens in the quasineutral limit, using a somewhat unusual point of view. He 
describes the plasma as a superposition of a (possibly uncountable) collection 
of fluids : the distribution function is written under the form
$$
f_\ep(t,x,v) = \int_M \rho^\ep_\theta(t,x) \delta_{v^\ep_\theta(t,x)}(v) \mu(d\theta),
$$
where $\theta$ is a parameter belonging to some probability space $M$ and $\mu$ is a probability measure on that space. This is quite general since any reasonable distribution $f$ may be written under this form (actually in a very large number of ways), and in particular it applies to the ``cold electrons'' case, i.e. when the sequence of initial distribution $f_{0,\ep}$ converges towards a monokinetic profile of the form
\begin{equation}
\label{mono}
f_0(x,v)=\rho_0(x) \delta_{v_0(x)}(v),
\end{equation}
where $\delta$ denotes as usual the Dirac measure. 
Under some strong uniform regularity assumption on the whole sequence of 
solutions, which ensures that the instability phenomena discussed before  
(such as two stream instabilities) are not present, he proved 
that if the fast and undamped plasma oscillations are ``filtered'',  the 
collection $(\rho_\theta^\ep,v_\theta^\ep)_{\theta \in M}$ converges (up to some extraction)  
to a solution of a multiphase incompressible Euler system. 
Moreover, he also shows that the strong regularity assumptions are fulfilled if 
 the sequence of initial conditions converges in some space of analytic 
functions.

\bigskip
The previous convergence result of Grenier was improved in two directions by Brenier 
\cite{Br00} and Masmoudi~\cite{Mas},  in the ``cold electrons'' case 
only. Brenier  introduced  
the ``modulated energy'' method (also called relative entropy method) and proved stability estimates which entail the 
convergence for ``well-prepared'' initial data towards a dissipative solution of the 
Euler equation, a very weak notion of solutions  introduced by Lions in \cite{Lio96}  which satisfy  a weak-strong uniqueness principle. In particular, Brenier's technique has the advantage to require weak regularity assumptions.
These well-prepared initial 
conditions correspond  exactly to those for which the Langmuir waves vanish in the 
limit. 

The stability estimates of Brenier are natural since monokinetic 
initial data of the form \eqref{mono} correspond to an ``extremal'' case 
 of symmetric and monotonic profiles. He also explained how the results obtained in 
\cite{BG,Gr95} on the existence of defect measures may lead to the same result. 
Later, Masmoudi extended in \cite{Mas} the convergence to non necessarily 
well-prepared initial data, but with stronger regularity assumptions on the 
limit equation. His work combines the filtration technique and the 
modulated energy method.

\section{Main results}
In this work, we provide three types of 
results, related to the quasineutral limit and stationary states.
\begin{itemize}
\item Let $\mu(v)$ be a profile satisfying the Penrose instability criterion of 
Definition~\ref{def:Pen-first}, and a technical condition (that essentially forbids  the presence of isolated zeros and fast oscillations on its tail, see Definitions~\ref{cond-pos}
 and~\ref{cond-pos'}). Our first result asserts that even if 
$f_{0,\ep} \rightarrow  \mu$ in $W^{s,1}_{x,v}$ for some $s \in \N$, as $\ep\to0$, the asymptotics
\begin{equation}
\label{asymptotics}
f_\eps(t,\cdot) \rightharpoonup \mu(\cdot), \quad \text{as }\; \ep \rightarrow 0
\end{equation}
where $f_\eps$ is the solution to \eqref{quasi} with initial datum $f_{0,\ep}$, 
is only true for $t=0$ and does not hold on an interval of time $[0,T]$, for any $T>0$,
 and any $W^{-r,1}_{x,v}$-norm, $r \in \N$. Actually, the results we prove are more accurate, see Theorem \ref{thmGrenier-revisited}.
In other words, we provide a complete proof of the result suggested by Grenier in his note \cite{Gr99-1} (actually for a larger class of homogeneous equilibria and in general topologies).
\item Conversely, we are able to justify the asymptotics \eqref{asymptotics}, on 
any interval of time $[0,T]$, as soon as $\mu(v)$ satisfy a monotonicity and a 
symmetry condition (along with some other minor technical conditions), see 
Theorems \ref{thm} and \ref{thm-IP}. 
 In some sense, it extends in dimension one the stability results of 
Brenier~\cite{Br00} and Masmoudi~\cite{Mas} to some cases of initial data which do not converge to a monokinetic profile. 
\item We show the existence of an {uncountably infinite} number of 
stationary solutions (or ``BGK waves'') to a boundary value problem associated 
to \eqref{kineuler-intro}, with a lot of freedom in the construction. This is due to the possible presence of trapped particles, whose density 
is shown to depend in a very simple way on the boundary conditions, see 
Theorem \ref{thm:1}.  We shall discuss the potential consequences on the stability properties of \eqref{kineuler-intro} below.
\end{itemize}

Let us now describe more precisely each of these results.

\subsection{Unstable case.}

Before stating or instability result, we need to introduce a technical condition:

\begin{deft}
\label{cond-pos}
 We say that a positive and  $C^1$ profile $\mu(v)$ satisfies the $\delta$-condition 
if
\begin{equation} \label{cond:alpha}
\sup_{v \in \R}  \,  \frac{|\mu'(v)|}{(1+ |v|)\mu(v)} < + \infty.
\end{equation}
\end{deft}
In the sequel, we shall also introduce a more general (but also more technical) condition that allows to handle some non-negative profile, see the $\delta'$-condition of Definition~\ref{cond-pos'} for details.

\bigskip

We are now in position to state  the ``instability'' result. As usual, the notation 
$W^{s,1}_{x,v}$ refers to the classical Sobolev spaces built on $L^1_{x,v}(\T 
\times \R)$ and of order $s$. In what follows, we say that a profile 
$\mu(v)$ is smooth if it belongs to $W^{s,1}_v$ for all $s\in \N$.
We shall use the following convention in the statement of the theorem: the notation $(f_\ep)$ (with the continuous parameter $\ep >0$) refers in fact to  a sequence $(\ep_k)_k$ going to $0$ and a sequence $(f_{\ep_k})_k$.
\begin{thm} 
\label{thmGrenier-revisited}
Let $\mu(v)$ be a smooth profile satisfying the Penrose instability criterion of Definition~\ref{def:Pen-first}.  Assume either that $\mu$ is positive and satifies the $\delta$-condition of Definition~\ref{cond-pos}, or that $\mu$ is non-negative and satisfies the $\delta'$-condition of Definition~\ref{cond-pos'}. For any $N>0$ and $s>0$, there exists a 
sequence of non-negative initial data $(f_{0,\eps})$ such that 
$$
\| f_{\eps,0}- \mu\|_{W^{s,1}_{x,v}} \leq \eps^N,
$$
and denoting by $(f_\eps)$ the sequence of solutions
to \eqref{quasi} with initial data $(f_{0,\eps})$, the following holds:

\begin{enumerate}[i)]
\item {\bf $L^1$ instability for the macroscopic observables:} the density $\rho_\ep := \int f_\eps \, dv$, and the electric field $E_\ep = - \pa_x V_\ep$. For all $\alpha\in [0,1)$, we have
\begin{equation}
\label{insta:macro}
\liminf_{\eps \rightarrow 0} \sup_{t \in [0,\eps^\alpha]} \left\| \rho_\eps(t) - 1 \right\|_{L^1_{x}} >0,
\qquad
\liminf_{\eps \rightarrow 0} \sup_{t \in [0,\eps^\alpha]} {\eps}\left\| E_\eps \right\|_{L^1_x} >0.
\end{equation}

\item {\bf Full instability for the distribution function:} for any $r \in \Z$, we have
\begin{equation}
\label{insta:full}
\liminf_{\eps \rightarrow 0} \sup_{t \in [0,\eps^\alpha]} \left\| f_\eps(t) - \mu \right\|_{W^{r,1}_{x,v}} >0.
\end{equation}
\end{enumerate}
\end{thm}

For $t=0$, by construction, we have
$ \lim_{\epsilon\rightarrow 0} \Black \| \rho_{0,\epsilon}- 1  \|_{L^1_x}  
=  
\lim_{\epsilon\rightarrow 0} \| E_{0,\epsilon} \|_{L^1_x}
= 0$. This theorem can thus be rephrased as follows: there exist small smooth perturbations of 
$f_0$ for which the corresponding solutions of \eqref{quasi} do not converge to 
the expected stationary solution $\mu$ in a weak $W^{-r,1}$-sense for any $r \in \N$. 
\begin{rque} 
It is possible to lower down the required regularity on the profile $\mu$, 
see Remark \ref{rem:NLinsta}.
\end{rque}
\begin{rque} 
Keeping the notations of the theorem, we can actually show that for any $r \in \N$,
\begin{equation}
\label{insta:macro3}
\liminf_{\eps \rightarrow 0} \sup_{t \in [0,\eps^\alpha]} \frac{1}{\eps^r}\left\| \rho_\eps(t) - 1 \right\|_{W^{-r,1}_{x}} >0,
\end{equation}
Note however that the instability is not directly seen in the
$W^{-1,1}$-norm of the density.
This is due to the preservation of the  total energy that provides stability in weak norms on macroscopic observables:
$$
\bigl\| \rho_\ep(t) - 1 \bigr\|_{W^{-1,1}_x} \le C \ep^2 \| E_\ep \|_1 \le C 
\ep \, \sqrt{\EE_\ep[f_{0,\ep}]}  . 
$$
\end{rque}

The $\delta$-condition of Definition~\ref{cond-pos} 
and the $\delta'$-condition of Definition~\ref{cond-pos'} are 
 precisely introduced in order to force that the sequence of initial data we build in Theorem~\ref{thmGrenier-revisited} is non-negative, and thus to ensure their physical relevancy (and this their only purpose). We refer to the paper of Guo and Strauss \cite{GS98} where the same problem is faced.

 The $\delta$-condition~\eqref{cond:alpha} is quite general: it is satisfied by positive profiles that, for large velocities, do not oscillate too
much and decrease slower than some Maxwellian distribution.
For instance profiles
\begin{itemize}
\item which coincide with a power law for large velocities, i.e.
$$
\mu(v) = \frac{\lambda_1}{|v|^{\lambda_2}}, \quad \text{for } v \text{ large enough, with } \lambda_1,\lambda_2>0;
$$
\item which coincide with Maxwellian profiles for large velocities, i.e.
$$
\mu(v) = {\lambda_1}e^{-\lambda_2 |v|^2}, \quad \text{for } v \text{ large enough, with  } \lambda_1,\lambda_2>0.
$$
\end{itemize}
Nevertheless, profiles that vanish at some point never satisfy the $\delta$-condition. However the $\delta'$-condition of Definition~\ref{cond-pos'} allows to handle profiles that vanish (at infinite order only) or decrease faster than exponentially.  As it is not very explicit, we postpone its precise statement to the subsection~\ref{subsec:delta'} and detail some sufficient conditions in the following Proposition.
 
\begin{prop} \label{suff-delta}
The $\delta'$-condition of Definition~\ref{cond-pos'} is satisfied if one of the following holds true:
\begin{itemize}
\item[i.] the profile $\mu$ is positive and satisfy for $v$ large enough with $\alpha >1$ and $C_\alpha >0$
\begin{equation} \label{cond:2.1i}
\frac1{C_\alpha} \, |v|^\alpha \le \frac{\mu'(v)}{\mu(v)} \le C_\alpha \,  |v|^\alpha;
\end{equation}
 \item[ii.] the profile $\mu$ is $C^\infty$ on the whole line $\R$, positive on the union of 
a finite number of disjoints open interval $(a_i,b_i)$ and equal to zero outside, and for all $i$, there exists $\eps_1$ such that on $(a_i,a_i + \eps_i)$ (resp. on $(b_i-\eps_i,b_i)$), $\mu$ satisfies for some $C_i>0$ and $\beta_i>1$ (resp. $C_i'>0$ and $\beta_i'>1$)
\begin{equation} \label{cond:2.1ii}
\frac1{C_i} \, |v-a_i|^{-\beta_i} \le \frac{\mu'(v)}{\mu(v)} \le C_i \,  |v-a_i|^{-\beta_i},
\end{equation}
$$
\text{resp.} \quad \frac1{C_i'} \, |v-b_i|^{-\beta_i'} \le \frac{\mu'(v)}{\mu(v)} \le C_i' \,  |v-b_i|^{-\beta_i'}.
$$
The case where some $b_i$ is equal to some $a_j$ is allowed, but only with the additional assumption that the profil $\mu$ should vanish at any order (all derivative should vanish) at this point.
\end{itemize}
A mix of points $i.$ and $ii.$ is also allowed. On the other hand, the $\delta'$-condition of Definition~\ref{cond-pos'} is not satisfied if $\mu$ has a zero with finite order.
\end{prop}
\Black

Proposition~\ref{suff-delta} will be proved after Theorem~\ref{thmGrenier-revisited}. It covers a large variety of profiles. Loosely speaking, we have:
\begin{itemize}
\item Point $i.$ is satisfied by positive profiles that do not oscillate for large velocities and decrease very fast: for instance, $\mu(v) \sim_{|v|\to +\infty} \exp(-|v|^\alpha)$, with $\alpha >2$. 
\item Point $ii.$ is satisfied by some smooth profiles with compact support: for instance \\
\mbox{$\mu(v)= \exp\bigl(-(b-v)^{-1}(v-a)^{-1} \bigr)$} on $(a,b)$ and $0$ outside.
\end{itemize}

However, while it is possible to handle zero of infinite order in $\mu$ 
, zeros of finite (and especially small) order may raise difficulties. As a matter of fact, our proof is not relevant if $\mu$ possesses a zero with small order.  We are not able to say if this is only a technical point, or a more important physical point. 

\subsection{Stable case.} We now restrict to particular homogeneous stable 
equilibria. 
The precise conditions we require are listed in the following definition.
\begin{deft}[$S$-stability] \label{def:S-stable}
We will say that a profile $\mu$  satisfying $\int_\R \mu(v) \,dv = 1$ 
 is $S$-stable if  the four following conditions are fulfilled:
\begin{itemize}
 \item[i)] \emph{Continuity:} $\mu$ is continuous on $\R$.
  \item[ii)] \emph{Finite energy:}  $ \int \mu(v) v^2 \,dv < + \infty$.
   \item[iii)] \emph{Monotonocity:} There is $\bar v \in \R$, such that  $v 
\mapsto \mu(v)$ is increasing for $v < \bar v$ and decreasing for $v > \bar v$, 
and so $\mu$ reaches its unique maximum at $\bar v$. 
  \item[iv)] \emph{Symmetry:} For all $v \in \R$, $\mu(2 \bar v - v) = \mu(v)$.
\end{itemize}
In this case, there is a unique continuous  and increasing  function 
$\varphi : (-\infty,0] \rightarrow \R^+$ such that
\begin{equation} \label{condphi}
\mu(v)= \varphi \Bigl(-\frac{\vert v -\bar v \vert^2} 2\Bigr), \quad  \text{and} \quad 
\int_{-\infty}^0 \varphi(u) \sqrt{-u} \,du <+\infty.
\end{equation}
\end{deft}

In Theorems \ref{thm} and \ref{thm-IP}, we will justify the quasineutral limit
 around such stable equilibria. 
The stability of the solution $f_\ep$
of~\eqref{quasi} around $\mu$ will be controlled thanks to the
so-called ``Casimir functionals'' defined in the following
\begin{deft} \label{def:Casimir}
For any $S$-stable profile, to which we associate the function~$\varphi$
defined in~\eqref{condphi}, 
we introduce a function $Q : \R^+ \rightarrow \R$ satisfying $Q(0)=0$ and $Q' = 
\varphi^{-1} $ on the range of $\varphi$. Outside this range, the only 
condition is 
that $Q'$ is increasing and continuous so that $Q$ is globally convex and $C^1$ 
on $\R^+$. For such a function $Q$, we can define the associated Casimir functional
\begin{equation} \label{HQintro} 
H_Q(f) := \int \left(Q(f) -  Q(\mu) - Q'(\mu)(f-\mu) \right) dvdx,
\end{equation}
which is well defined with value in $[0,+ \infty ]$.
\end{deft}
Remark that there is some freedom in the choice of $Q$, whose values are 
imposed only on the range of $\varphi$. But all the results will shall give in 
the sequel are valid independently of the particular choice made for $Q$.

This ``Casimir functional'' is a kind of relative entropy for the Vlasov-Poisson
equation; it is built in order to be minimized by $\mu$. Similar quantities were 
originally introduced by Arnold~\cite{Ar65,Ar66} for fluid models. Later, their use was
generalized to plasma models in~\cite{HMRW}.
The first fully rigorous application of these functionals to plasmas was
performed by Rein in~\cite{Rei}: he used them to prove the $L^2$ stability
around compactly supported equilibria that are decreasing function of the
energy. We also refer to that article for a very clear explanation of their
interest.
This result was later extended in \cite{CCD} to non compact equilibria. As we
shall see in Section~\ref{sec:conv}, the above quantity~\eqref{HQintro}, may
control $L^p$ norms
$$
\| f_\ep-\mu \|_p^2  \le \frac1C H_Q(f_\ep) , \quad \text{ for some } \;p, 
\quad  \text{ for instance } p=1,2.
$$ 
We refer to Proposition~\ref{prop:HQ} and more generally to 
section~\ref{subsec:HQ} for more details. 
Therefore, all the following results, showing that $H_Q(f_\ep)$ (or $H_Q$ 
applied to some filtered distribution function) remains small under some 
assumption on the initial
conditions, may be translated in results of stability in more usual norms.

\bigskip

 Our first result of stability is given in the following
\begin{thm}  \label{thm}
Let $\mu$ be a $S$-stable stationary solution to \eqref{kineuler-intro} of the 
form given in~\eqref{condphi}. Assume that there exists $\eta>0$, such that  $\mu$ satisfies 
\begin{equation}
\label{assumption-mu}
\int \mu(v) (1+ v^{2 + \eta})\,dv<+\infty.
\end{equation}

For all $\epsilon >0$, let $(f_\epsilon,V_\epsilon)$ be 
a  strong solution to~\eqref{quasi}, with initial datum $f_{0,\epsilon}$ and define the 
``modulated energy''
\begin{equation} \label{def:Lep}
\LL_\ep[f_\ep] := 
H_Q(f_\ep)  + \frac {\ep^2} 2 \int ( \partial_x V_\epsilon )^2 dx.
\end{equation}
Then, $\LL_\ep$ is a Lyapunov functional in the sense that
$$
\forall t \in \R^+, \qquad \LL_\ep[f_\ep(t)] = \LL_\ep[f_{0,\ep}].
$$
\end{thm}
The proof consists in a reformulation of the functional $\LL_\ep$ that shows 
that it is composed only of invariant quantities. Remark that the assumption 
that $\mu$ is $S$-stable is crucial since otherwise, we can not even define an 
associated Casimir functional.

\begin{rque} It is also possible in this theorem to consider a sequence of global weak solutions to~\eqref{quasi} in the sense of Arsenev \cite{Ar}, in which case the conclusion becomes
$$
\forall t \in \R^+, \qquad \LL_\ep[f_\ep(t)] \leq \LL_\ep[f_{0,\ep}].
$$
\end{rque}

The above theorem is particularly useful  when the initial potential energy 
vanishes in the limit, that is when $\| \partial_x V_{0,\ep} \|_2 = 
o\bigl(\ep^{-1}\bigr)$, or in other words when $\| \rho_{0,\ep} -1 \|_{H^{-1}} 
= o(\ep)$. 
It corresponds to what we can call well-prepared initial data, for which there are no plasma oscillations in the limit. In this situation, we can express the conclusion of the previous theorem as follows
\begin{coro}
If the sequence of initial data is well-prepared in the sense that
$\LL_\ep[f_{\eps,0}] \to 0$, then for all $t\geq 0$, $\LL_\ep[f_\ep(t)] \to 0$.
Moreover the rate of convergence to $0$ for any positive time is the same as the
one at initial time.
\end{coro}
 This corollary is thus orthogonal to Theorem~\ref{thmGrenier-revisited}, since it tells us that when the profile is $S$-stable, then it is not possible to find initial conditions satisfying the conclusions of that theorem.

On the other hand, this does not say much if the data are not well-prepared. For
instance, in the case where $ V_{0,\ep}$ displays some oscillations in space,
then we do not expect the plasma oscillations to vanish, and we have to filter
them in order to prove a convergence result. The precise result is the
following.

\begin{thm}
\label{thm-IP}
Let $\mu$ be a $S$-stable stationary solution to \eqref{kineuler-intro} of the 
form
given in~\eqref{condphi}. Assume that there exists $\eta>0$, such that  $\mu$
satisfies 
\begin{equation}
\label{assumption-mu-IP}
\int \mu(v) (1+ v^{2 + \eta})\,dv<+\infty.
\end{equation}

For all $\epsilon >0$, let $(f_\epsilon,V_\epsilon)$ be a global 
strong solution 
to~\eqref{quasi}, with initial datum $f_{0,\epsilon}$. For any smooth potential 
$V_0$ such that
$\partial_{xxx} V_0 \in L^\infty$, we define an associated ``modulated free
energy''
\begin{equation} \label{def:Lep2}
\begin{split}
\LL^O_\ep(t) & := 
H_Q\left[ f_\ep\left(t,x,v - \partial_x V_0(x-
\bar v t) \sin \frac t \ep \right) \right]
\\
& \hspace{20mm} + \frac 1 2 \int \Bigl[  \ep \partial_x V_\epsilon -
\partial_x V_0(x- \bar v t)
\cos \frac t \ep \Bigr]^2 dx.
\end{split}
\end{equation}
Then, we can control the growth of  $\LL^O_\ep$  in the
sense that there exists a constant $K>0$, depending on $\|
\partial_{xx} V_0 \|_\infty $ and $ \| \partial_{xxx} V_0 \|_\infty$, such
that for any $t>0$
\begin{equation}
 \forall t \geq 0, \quad \LL_\ep^{O}(t) \leq e^{2 \| \pa_{xx} V_0 
\|_{L^\infty} t}  \Bigl[ \LL_\ep^{O}(0)+  K  \ep \bigl( 1 + \EE_{\ep,0} 
+ \mathcal Q_{\ep,0} \bigr) \Bigr],
 \end{equation}
where
\begin{equation}
\label{eq-crucial}
\EE_{\ep,0} := \EE_\ep(f_{0,\ep}), 
\quad \text{and} \qquad 
\QQ_{\ep,0} := \int \left(|Q|(f_{\eps,0}) +
\frac{Q^2(f_{\eps,0})}{f_{\eps,0}}\right) \, dv dx.
\end{equation}
\end{thm}

Of course, this theorem implies the following stability result
\begin{coro}
Assume in addition to the hypotheses of Theorem~\ref{thm-IP} that $f_{\ep,0}$
satisfies the (very weak)  bound     
$$
\EE_{\ep,0} +\QQ_{\ep,0}=o(\ep^{-1}).
$$
If $\LL^O_\ep[f_{\eps,0}] \to 0$, then for all $t\geq 0$, $\LL^O_\ep[f_\ep(t)]
\to 0$. 
\end{coro}

Note that the best possible rate of convergence is given by $\eps$,
contrary to the well-prepared case, and that this best rate is reached when
$\EE_{\ep,0}$ and $\QQ_{\ep,0}$ are uniformly bounded in $\ep$.

\begin{rque}  
Observe that the plasma oscillations  which have to be filtered are mostly oscillations in velocity. In fact, 
the density $\rho_\ep$ will remain close to $1$. This may be quantified 
in some cases, for instance when $\mu$ is a Maxwellian equilibrium, in which case 
the Casimir functional is the ``usual'' relative entropy (see $iii)$ in Proposition~\ref{prop:HQ}).
\end{rque}

\begin{rque} In the ill-prepared case, to obtain stability, we have to add the
assumption that $Q_{\eps,0}$ is (not necessarly uniformly)   bounded.
This is not a huge requirement: for instance, when $\mu$ is a Maxwellian
distribution, it only requires $f_{\ep,0} (\ln f_{\ep,0})^2$ to be integrable
(see again $iii)$ in Proposition~\ref{prop:HQ}).

\end{rque}
In the case where $V_0$ does not have a bounded third derivative, 
an interpolation argument still allows to get the following stability result.
\begin{coro} \label{rem:lowreg}
With the same notations and hypotheses of Theorem \ref{thm-IP}, except that we only assume that $\pa_{xx} V_0 \in L^\infty$, we obtain
\begin{equation}
 \forall t \geq 0, \quad \LL_\ep^{O}(t) \leq 4 \, e^{2 \| \pa_{xx} V_0 
\|_{L^\infty} t}  \Bigl[ \LL_\ep^{O}(0)+  K'  \ep^{\frac23} \bigl( 1 +
\EE_{\ep,0} + \QQ_{\ep,0} \bigr)^{\frac 23} \Bigr],
\end{equation}
for some constant $K'$ depending only on $\| \pa_{xx} V_0 \|_\infty$. 
\end{coro}
The argument is  detailed at the end of Section~\ref{sec:IP}. In this case, the best rate of convergence is therefore given by $\ep^{2/3}$.
\begin{rque}
Alternatively to the Modulated energy (or Casimir functional) technique developed here, the 
technique of \cite{MP86,BR} based on rearrangement inequalities may be used to 
prove the stability in the well-prepared case. But this seems more difficult for the
ill-prepared case. Indeed, it relies on the fact 
that a $S$-stable profile (for $\bar v=0$), will be the minimizer of the total energy 
in some class of equi-measurable functions. It does not seem clear how to adapt this argument
in the ill-prepared case, mostly because the plasma oscillations do carry some 
energy in the limit.
\end{rque}

\subsection{Locally  symmetric solutions to~\eqref{kineuler-intro} 
are homogeneous.}

The two previous theorems of stability rely on
a natural monotonicity condition, but also on a symmetry condition on $f$. Indeed, we use in
a crucial way a Casimir functional, that can be constructed only for symmetric
profiles. 

It is maybe also natural to expect that such conditions could be helpful to 
prove the convergence for solutions to \eqref{kineuler-intro} which are not 
necessarily stationary, but which satisfy the monotonicity condition and the symmetry 
condition at any time $t$ and any position $x$
\begin{equation} \label{SymGene}
\forall\, t \in \R^+,\; \forall\, x \in \T, \; \exists\,  \bar v(t,x), \;
\varphi_{t,x} \; \text{s.t.} \quad 
f(t,x,v) := \varphi_{t,x} \Bigl( -\frac{\vert v -\bar v(t,x) \vert^2} 2 \Bigr)
\end{equation}
where $\varphi_{t,x}: \R^- \to \R^+$ is increasing. Then, for such
functions, it is easy to build (time and position dependent)  
 relevant Casimir functionals as in Definition \ref{def:Casimir}.
``Unfortunately'',  they cannot help us to get a stability result
around solutions of~\eqref{kineuler-intro} with a non trivial dynamics, 
since we shall prove that the only
solutions $(f,V)$ to \eqref{kineuler-intro} satisfying also~\eqref{SymGene}
(along with some weak regularity assumptions) are 
the stationary equilibria $\mu$ for which $\bar v$ and $\varphi$ are
independent of $t$ and $x$ 
\begin{equation}
\mu(v)= \varphi\left(-\frac{\vert v -\bar v\vert^2}{2}\right).
\end{equation}
See Proposition~\ref{prop-EQ} for the precise result. Loosely speaking, this means there is virtually no hope of relying on a modulated energy method to derive the quasineutral Vlasov equation~\eqref{kineuler-intro}, for non stationary data.

\subsection{Construction of BGK waves for  the quasineutral Vlasov
equation.}
 
In order to emphasize on the somewhat ``degeneracy'' of the limit
system~\eqref{kineuler-intro}, we will also show that the construction 
of the equivalent of what the so-called BGK waves in the Vlasov-Poisson case \cite{BGK} is much more degenerate in the quasineutral case.
Precisely, we will study the following boundary problem for
the associated stationary  kinetic equation
\begin{equation} \label{eq:vla1D}
\begin{cases}
v \,  \partial_x f - \partial_x V \partial_v f = 0, \\
\rho = \int f(x,v) \,dv =1, 
\end{cases} 
\end{equation}
on the space $\Omega = [0,1]\times \R$. The incoming boundary conditions are
given by
\begin{equation} \label{eq:BC} \begin{cases}
f(0,v) = f_0^+(v) \quad \text{if} \; v \geq 0, \\
f(1,v) = f_0^-(v) \quad \text{if} \; v \leq 0.
\end{cases} \end{equation} 

  This model is the stationary equation associated to the
quasineutral Vlasov equation~\eqref{kineuler-intro}, with the  boundary
conditions~\eqref{eq:BC}. We prove the following
\begin{thm} \label{thm:1}
Assume that $f_0^\pm: \R^\pm \rightarrow \R^+$  are  two nonnegative and measurable functions such that  
$\int_0^\infty \bigl( f_0^+(v) + f_0^-(-v) \bigr) \,dv =1$. Define the 
function $f_T$ on $(0,+\infty)$ as below
\begin{equation} \label{def:fTdouble}
f_T(u) := \frac1 \pi \int_0^\infty \bigl( f_0^+(v) + f^-_0(-v)\bigr) \frac{u\,v\,dv}{(u^2 + 
v^2)^{\frac32}}.
\end{equation}
Then for any continuous potential  
$ V : [0,1] \rightarrow \R^-$ satisfying $V(0)=V(1)=0$ 
the function
\begin{equation} \label{def:fstat}
 (x,v) \mapsto f(x,v) =
 \begin{cases} 
\ds   f_0^+ \bigl( \sqrt{v^2 + 2V(x) } \bigr) &
  \text{if} \quad v  \ge \sqrt{ -2 V(x) }, \\
  f_0^- \bigl(- \sqrt{v^2 + 2V(x) } \bigr) & \text{if} \quad v  \le - \sqrt{-2 V(x) }, \\
  f_T \bigl(\sqrt{-v^2 - 2V(x) } \bigr) &
\text{if} \quad |v|  < \sqrt{-2 V(x) },
 \end{cases}
\end{equation}
together with $V$ gives a solution of~\eqref{eq:vla1D} in the sense of
distributions.  Moreover, any solution with $V$ nonpositive and vanishing 
at the boundary is of the above form.
\end{thm}

The striking points in Theorem \ref{thm:1} are that:
\begin{itemize}
\item  there is a huge freedom in the choice of the potential $V$. In particular 
there is no a priori bound on its minimal value.
\item the density of trapped function depends only on the boundary conditions $f_0^\pm$, and not on the potential $V$. 
\end{itemize}
 
In some sense, this 
feature illustrates the degeneracy of the quasineutral Vlasov equation, compared 
to the classical Vlasov-Poisson.  For instance, the fact that $f_T$ is independent of $V$ is due to the fact we pass from \eqref{eq:vla1D} to \eqref{eq:vlapoi1D} by replacing a ``local'' equation by a Poisson equation. More precisely, the corresponding problem in the Vlasov-Poisson case  is the
construction of the so called ``BGK waves'', which was performed in 
the pioneering work of Bernstein, Greene and Kruskal \cite{BGK}. Consider the
problem
\begin{equation} \label{eq:vlapoi1D}
\begin{cases}
v  \, \partial_x f_\ep - \partial_x V_\ep \partial_v f_\ep = 0, \\
- \ep^2 \pa^2_x V_\ep = \int f_\ep \, dv - 1, 
\end{cases} 
\end{equation}
on the space $\Omega = [0,1]\times \R$, with the same boundary 
conditions~\eqref{eq:BC}.
Building on \cite{BGK}, we can construct numerous solutions $(f_\ep,V_\ep)$ 
of~\eqref{eq:vlapoi1D}-~\eqref{eq:BC},  but in this case, the density of trapped 
particles $f_{T,\ep}$ always depends on the potential $V_\ep$, and this one 
cannot be completely arbitrary: for instance it has to be bounded from below by 
some constant $C\ep^{-2}$.

\bigskip

BGK waves play an important role in the large time dynamics of the Vlasov-Poisson equation: we refer on this topic to the works of Lin and Zheng \cite{LZ1,LZ2}. The abundance of BGK waves for the quasineutral Vlasov equation \eqref{kineuler-intro} shown in Theorem \ref{thm:1} suggests that the dynamics for \eqref{kineuler-intro} is very rich.

\subsection{The case of the Vlasov-Poison equation for ions.}

The following Vlasov-Poisson equation
\begin{equation}
\left\{
    \begin{array}{l}
\ds   \partial_t f_\epsilon + v \, \partial_x f_\epsilon   -\partial_x V_\epsilon \, \partial_v f_\epsilon =0,  \\
\ds  \alpha V_\epsilon - \epsilon^2 \partial_{x}^2 V_\epsilon = \rho_\epsilon
-1,
  \end{array}
  \right.
\end{equation}
 for some $\alpha >0$, is also worth studying.  Such a model, which we shall call
\emph{Vlasov-Poisson equation for ions} is often encountered in plasma physics:
it aims at describing the dynamics of ions in a one dimensional plasma, 
when the electrons are assumed to be adiabatic (or
massless), so that they have reached a thermodynamic equilibrium. When $\ep\to
0$, one formally obtains the Vlasov-Dirac-Benney equation \eqref{VDB}.
This will also be studied later in this paper. { We will explain in the last 
section how to adapt most of our results to that case.}

\bigskip

As a final remark, let us mention that it seems possible to adapt all the results stated here to larger dimensions. 
It may be the goal of some future works.

\subsection{Organization of the following of the paper.}

The following is dedicated to the proofs of our main results.
In Section \ref{sec:nonconv}, we deal with the unstable case and  provide a proof of Theorem \ref{thmGrenier-revisited}.
In Section \ref{sec:conv} we study briefly the properties of the Casimir 
functional, we heuristically derive the equation for the correctors that are 
necessary in the ill-prepared case, and finally prove the stability
Theorems~\ref{thm} and~\ref{thm-IP}.
In Section \ref{sec-EQ}, we prove that the only solutions to 
\eqref{kineuler-intro} satisfying  locally everywhere the  monotonicity and 
symmetry condition~\eqref{SymGene}  are necessarily stationary. 
Then, in Section \ref{sec-sta1}, we turn to the construction of BGK waves for \eqref{kineuler-intro} and prove Theorem \ref{thm:1}. Finally, we explain how the results of this paper can be adapted to handle the Vlasov-Poisson equation for ions with adiabatic electrons in Section \ref{sec:mod}.

%
\section{Unstable case: proof of Theorem \ref{thmGrenier-revisited}}
\label{sec:nonconv}

In this section we will first give some elements about the linearized Vlasov-Poisson equation~\eqref{quasi} around homogeneous equilibria, and explain the relevancy of the Penrose criterion of Definition~\ref{def:Pen-first}.
Then, we shall provide a complete proof of Theorem~\ref{thmGrenier-revisited}.

That proof will rely on a instability theorem (Theorem~\ref{NLinsta}) for the original Vlasov-Poisson equation~\eqref{eq:VP1D} that is interesting by its own. It is proved in subsection~\ref{sec-insta}.
Later, we introduce the $\delta'$-condition, and explain the necessary modifications to perform in the proof of Theorem~\ref{NLinsta}. 

We will end this section by proving Proposition~\ref{suff-delta}.

\subsection{Linearized Vlasov-Poisson equation and Penrose criterion}
\label{sec:Penrose}

In this paragraph, we work on $ \T_M \times \R$, where we recall that $\T_M:= \R/(M\Z)$ and $M>0$.
Given some smooth homogeneous equilibrium $\mu(v)$, we study the linearized Vlasov-Poison equation around $\mu$:
\begin{equation}
\label{VPllin}
\left\{
    \begin{array}{l}
\ds   \partial_t f + Lf:=   \partial_t f + v \, \partial_x f   -\partial_x V 
\, \mu' =0,  \\
\ds   -\partial_x^2 V = \int f dv,  
  \end{array}
  \right.  \text{for }\; t\geq 0, x\in \mathbb{T}_M,  \, v\in \mathbb{R},
\end{equation}
with $D(L)= \{f \in L^1_{x,v}(\T_M \times \R), \, \int f \, dvdx =0, \, Lf \in L^1_{x,v}(\T_M \times \R)\}$.

We shall rely on the description of the spectrum of the linearized 
Vlasov-Poisson equation, which was performed by Degond in~\cite{Deg}.
We gather some useful information from~\cite[Theorem 1.1]{Deg} in the following 
proposition.

\begin{prop}\label{prop:Deg}
 Assume that $\mu \in W^{3,1}$. 
Consider the following dispersion relation, defined for $\lambda \in \mathbb{C} \setminus i \R$ and $n \in \Z^*$:
\begin{equation}
\label{dispersion}
D(n,\lambda) = 1 - \frac{M^2}{(2\pi n)^2} \int_\R \frac{\mu'(v)}{v - i 
\frac{M \lambda }{2\pi n }} \, dv.
\end{equation}
$i)$ The spectrum of $L$ is given by
$$
\sigma(L)  = i \R \cup \{ \lambda \in \mathbb{C}, \, \exists n \in \Z^*, D(n,\lambda) =0\}.
$$
It is symmetric with respect to the real and imaginary axis.
Moreover, $\lambda$ is an eigenvalue of $T$ if and only if $\lambda=0$ or if 
there exists $n \in \Z^*$ such that $D(n,\lambda)=0$.  

Moreover, there exists $\omega_0>0$ such that $\sigma(L) \subset \{\lambda \in \C, |\Re \lambda |\leq \omega_0\}$.

\bigskip \noindent
$ii)$ If $\Re 
\lambda >0$, the set of solutions $n$ of the equation $D(n,\lambda)=0$ is finite 
and denoted by $\{n_1, \cdots, n_p\}$. In addition, a basis of the eigenspace associated 
to $\lambda$ is given by
\begin{equation} \label{eigen-form}
\left\{ e^{i \frac{2\pi n_j}{M} x} \frac{\mu'(v)}{v - i 
\frac{M \lambda }{2\pi n }}
, \, j = 1, \cdots, p\right\}.
\end{equation}
In particular, these eigenfunctions associated to 
$\lambda$ belong to $W^{s-1,1}$ if $\mu \in W^{s,1}$,
and their associated spatial densities are equal to $\frac{(2\pi 
n)^2}{M^2} e^{i \frac{2\pi n_j}{M} x}$.
 
 \bigskip \noindent
  $iii)$ Assume that $\mu$ is smooth. For all $\Gamma >  
\max \{ \Re \lambda, \, \, \lambda \in \sigma(L) \}$ and all
all $s \in \N$, there exists $C_\Gamma^s \geq 1$ such that the 
following holds, for all 
$h_0 \in W^{1,s}(\T_M \times \R)$,
\begin{equation} \label{bound-lin}
\forall t \geq 0, \quad \| e^{-t L}h_0 \|_{W^{s,1}} \leq C_{\Gamma}^s \,
e^{\Gamma t} \, \| 
h_0 \|_{W^{s,1}}.
\end{equation}

\end{prop}

The proof of point $i)$ and $ii)$ in that Proposition is mostly done 
in~\cite[Theorem 1.1]{Deg}. The remarks about the regularity and the spatial 
density in point $ii)$ are plain consequences of the particular form of the 
eigenfunctions, and of the dispersion relation~\eqref{dispersion}.
The estimate~\eqref{bound-lin} is a consequence of 
the fact that $L$ generates a strongly continuous semi-group on $L^1$ and of a 
standard bootstrap argument.
The regularity assumptions on $\mu$ are not optimal, see~\cite{Deg} for details.

Remark that $D(n,\lambda)=0$ may be rewritten as
\begin{equation}\label{def:G}
G \Bigl( i \frac{M \lambda }{2\pi n } \Bigr) 
= \frac{(2\pi n)^2}{M^2}, \qquad
\text{with}
\quad
G(\zeta):= \int_\R \frac{\mu'(v)}{v - \zeta} \, dv.
\end{equation}
And in particular, there exists a eigenvalue $\lambda$ with positive real 
part if and only if
\begin{equation} \label{cond:ens}
\Bigl\{ \; \Bigl(\frac{2\pi n}{M}\Bigl)^2,  \; n \in \N \Bigr\} \bigcap 
G(\Im^+) \neq \emptyset, \quad \text{where} \quad
\Im^+ := \{ z \in \C, \;\text{s.t.} \; \Im z >0\}.
\end{equation}
As Penrose remarks in~\cite{Pen}, the linear instability is therefore linked to the values 
taken by $G$ on $\Im^+$, and since $G$ is holomorphic, we can use some powerful 
theorems of complex analysis to simplify the condition~\eqref{cond:ens}. Before 
going on, we shall gather some properties of $G$ in the following lemma.
\begin{lem} \label{lem:propG}
Assume that $\mu \in W^{3,1}$. Then the function $G$ is holomorphic on $\Im^+$.
It can be extended to the real line by
\begin{equation} \label{extendG}
\forall \xi \in \R, \quad 
G(\xi) := P.V. \int_\R  \frac{\mu'(v)}{v - \xi} \, dv + i \mu'(\xi),
\end{equation}
where $P.V.$ means that the integral has to be understood as a principal value.
Moreover, the extended function $G$ is uniformly continuous and $G(\xi)$ goes 
to zero when $|\xi|$ goes to infinity. 
\end{lem}

We shall not prove this lemma. We refer to~\cite{Pen} for some explanations. 
The uniform continuity is not stated in the later reference,  but it 
is a consequence on the continuity of $G$ on $\R$, which 
can be obtained after a careful estimate of the quantities involved.

\bigskip

In view of standard results of complex analysis, a complex $z$ belongs to  
$G(\Im^+)$ if and only if it is  encircled clockwise  by the curve $G(\R)$, covered form 
$-\infty$ to $+\infty$. In view of~\eqref{cond:ens}, there will be a 
eigenvalue with positive real part if and only if the curve $G(\R)$  encircles clockwise  a 
value $\bigl(\frac{2\pi n}{M}\bigl)^2$, for some $n \in \N$. But, as Penrose 
explains \cite{Pen}, it is possible only if the curve $G(\R)$ crosses the half-line~$
\bigl( \bigl[\frac{2\pi n}{M}\bigl]^2, + \infty \bigr)$ from below at some 
point $G(\xi_0)$ for some $\xi_0 \in \R$. In view of the 
definition~\eqref{extendG}, $\xi_0$ has to be a strict and local minimum of 
$\mu$, in the precise sense given in Definition~\ref{def:Pen-first}. Moreover, 
the real part of $G(\xi_0)$ should be greater than 
$\bigl(\frac{2\pi n}{M}\bigl)^2$. But since $\mu'(\xi_0)=0$, the principle 
value in~\eqref{extendG} is a true integral, and it leads after an integration 
by parts to the necessary condition~\eqref{Peneven}. 

Nevertheless, this necessary condition is not sufficient: indeed, the curve $G(\R)$ can 
cross by above the half-line of positive real numbers before  surrounding  
one of the requested values.  As mentioned in the introduction, it can be shown 
that the condition~\eqref{Peneven} is sufficient if the profile $\mu$ is 
symmetric with respect to the minimum (see \cite{GS}). But this symmetry condition is quite 
restrictive from the physical point of view. For instance, it is violated for what 
is usually called  a ``bump on tail'' profile:  a small bump added in the tail (for 
large velocities) of a given stable profile. 

However, in the limit of large boxes (or equivalently in the limit of small 
Debye length), we will prove in the following proposition that the Penrose 
criterion of Definition~\ref{def:Pen-first}  becomes a necessary and 
sufficient conditions for the existence of a eigenvalue with positive real part. 
\begin{prop} \label{prop:Penlim}
Assume that $\mu \in W^{3,1}$ satisfies the Penrose criterion of 
Definition~\ref{def:Pen-first}. Then there exists a $\eta>0$ such that if $\frac 
1 M < \eta$, then the linearized operator $L$ on $\T_M \times \R$ possesses an 
eigenvalue  $\lambda$ with $\Re \lambda >0$. Moreover, for such an eigenvalue, there exists an 
associated eigenfunction $h^\lambda$ of the form~\eqref{eigen-form}. In particular, it has some spatial inhomogeneity: precisely 
its associated density $\rho^\lambda$ has a non zero real part. Finally, if $\mu \in W^{s,1}$, then $h^\lambda \in W^{s-1,1}$.
\end{prop}

\begin{proof}[Proof of Proposition \ref{prop:Penlim}]
The proof relies only on elementary considerations. We shall treat only the case where the Penrose instability criterion is satisfied at  a strict minimum. The case of a flat minimum can be handled in a similar fashion.

Choose a minimum point of $\mu$, denoted by $\xi_0$,
satisfying the Penrose criterion of 
Definition~\ref{def:Pen-first}. As said before, it means that the curve $G(\R)$ 
crosses at $\xi_0$ the half-line of positive real numbers, at the point $G(\xi_0)$.  
We set $\ep := \frac12 G(\xi_0)$.
Since the extended function $G$ is uniformly continuous by 
Lemma~\ref{lem:propG}, we can choose some $\eta>0$ such that 
$|G(\xi) -G(\xi')| 
\le \ep$ when $|\xi-\xi'| \le  \eta$. 
We choose 
also two real numbers $\xi_-$ and $\xi_+$ satisfying 
$$\xi_0 -\eta \le \xi_- < \xi_0 < \xi_+ \le \xi_0 + \eta,
\quad \text{and} \quad 
\Im G( \xi_-) < 0, \qquad 
\Im G(\xi_+) >0.
$$
It is possible since $\xi_0$ is a strict local minimum of $\mu$.
Remark that thanks to the definition of $\eta$, they also satisfy
$ \Re G(\xi_\pm) \ge \frac12 G(\xi_0)$.

Next, we define $\ep' = \min \bigl(\frac12 G(\xi_0), |\Im G(\xi_-)|, |\Im 
G(\xi_+)|\bigr)$, and  associate to it some $\eta'>0$ by  uniform continuity of $G$. 
Then, we have $ \Re \, G(\xi_- + i \eta) >0$ and $\Im \, G(\xi_- + i \eta) <0$ 
and similarly $ \Re \, G(\xi_+ + i \eta) >0$ and $\Im \, G(\xi_+ + i \eta) 
>0$. Moreover we have $G([\xi_-,\xi_+]+ i  \eta)  \subset \{ \Re z >  0 \}$.
By the intermediate value theorem, it means that there exists some $\xi_1 \in [\xi_-, \xi_+]$ such 
that $G(\xi_1 + i \eta)\in (0,+\infty)$. 

But, since $G$ is holomorphic on $\Im^+$, its image $G(\Im^+)$ is  
open, and we can therefore conclude that $G(\Im^+)$ contains some interval 
$(a,b)$ with $0<a<b$.  Then, one can readily see that for $M$ large 
enough
\begin{equation*}
\Bigl\{ \; \Bigl(\frac{2\pi n}{M}\Bigl)^2,  \; n \in \N \Bigr\} \bigcap 
(a,b) \neq \emptyset,
\end{equation*}
so that condition~\eqref{cond:ens} is satisfied. The existence of a eigenvalue 
with positive part follows, and the claimed properties of the associated 
eigenfunction are a direct consequence of Proposition~\ref{prop:Deg}.
\end{proof}

\subsection{Proof of Theorem~\ref{thmGrenier-revisited}}
\label{subsec:insta}

Let $\mu(v)$ be a smooth profile satisfying the instability criterion of 
Definition~\ref{def:Pen-first}. Using Proposition~\ref{prop:Penlim}, we fix 
$M>0$ large enough such that 
the linearized operator $L$ around $\mu$ on $\T_M \times \R$ possesses a eigenvalue 
$\lambda$ with positive real part.

From now on, we consider only the sequence $\eps_k = \frac{1}{kM}$, for $k \in 
\mathbb{N}^*$,  but we shall forget the $k$ subscript for readability.

\medskip

{\sl  Step 1. Highly oscillating data.}

We consider $\eps M$-periodic (in $x$) solutions  to~\eqref{quasi}. 
Precisely, we look at solutions to the system
\begin{equation}
\label{quasi-osci}
\left\{
    \begin{array}{l}
\ds   \partial_t \tilde{f}_\epsilon + v \, \partial_x \tilde{f}_\epsilon   
-\partial_x \tilde{V}_\epsilon \, \partial_v \tilde{f}_\epsilon =0,  \\
\ds   - \ep^2 \partial_x^2 \tilde{V}_\epsilon = \int \tilde{f}_\epsilon dv -1, 
  \end{array} \right.
 \text{for  }t\geq 0, \,  x\in \mathbb{T}_{\eps M}:= \R/(\eps M\Z), \, v\in \mathbb{R}.
\end{equation}

We can canonically obtain from $\tilde{f}_\eps$ a solution $f_\eps$ to 
\eqref{quasi} by ``gluing'' together $(\eps M)^{-1}$ copies of $\tilde{f}_\eps$.

\bigskip

{\sl  Step 2. Rescaling.}

We perform the change of variables $(t,x,v) \to \left(\frac{t}\eps, 
\frac{x}\eps,v\right)$. In other words, we consider $(g_\eps,\varphi_\eps)$ such 
that:
\begin{equation}
\label{ch-scaling}
\tilde{f}_\eps(t,x,v) = g_\eps\left(\frac{t}\eps,\frac{x}\eps,v\right), \quad 
\tilde{V}_\eps(t,x) = \varphi_\eps\left(\frac{t}\eps,\frac{x}\eps\right).
\end{equation}
This leads to the study of the following system, which is now \emph{independent of $\eps$}, and posed for $t\geq 0$, $x\in \mathbb{T}_M$, $v\in \mathbb{R}$
\begin{equation}
\label{quasi-rescaled}
\left\{
    \begin{array}{l}
\ds   \partial_t g + v \, \partial_x g   -\partial_x \varphi \, \partial_v g 
=0,  \\
\ds   -\partial_x^2 \varphi = \int g dv -1,\\
  \end{array}
  \right.  \text{for }\; t\geq 0, x\in \mathbb{T}_M,  \, v\in \mathbb{R}.
\end{equation}
Remark that the standard Sobolev embedding on $\T_M$ implies a good 
control on the electric field. Precisely, for all 
$s \in \N$,  if $h$ has zero mean on $\T_M \times \R$, 
and  $\varphi$ satisfies $-\partial_x^2 \varphi = \int h dv $, then
\begin{equation}\label{estim:phi}
 \| \partial_x^{s+1} \varphi \|_\infty \le    \| \partial_x^s h \|_1.
\end{equation}
We shall use that estimate at multiple times in what follows.

Linearizing \eqref{quasi-rescaled} around $\mu$, we obtain
\begin{equation}
\label{quasi-rescaled-lin}
\left\{
    \begin{array}{l}
\ds   \partial_t h + Lh=   \partial_t h + v \, \partial_x h   -\partial_x \Psi 
\, \partial_v \mu =0,  \\
\ds   -\partial_x^2 \Psi = \int h dv,  
  \end{array}
  \right.  \text{for }\; t\geq 0, x\in \mathbb{T}_M,  \, v\in \mathbb{R},
\end{equation}
which is exactly the linearized system studied in the section \ref{sec:Penrose}.
\bigskip

{\sl  Step 3. A nonlinear instability result.}

The description of the spectrum obtained in the section \ref{sec:Penrose} 
allows to deduce the following non-linear instability theorem on the rescaled 
system.

\begin{thm}
\label{NLinsta} Assume that the profile $\mu$ satisfies the Penrose instability criterion~\eqref{def:Pen}, the technical condition~\eqref{cond:alpha} (or the condition~\eqref{cond:alpha'})  and belongs to all the $W^{s,1}$ for $s \in \N$. 
Then, there exists a sequence  $(\theta_r)_{r \in \N}$ of positive real numbers  such that for any $S \in \N$, and any $\delta>0$, there exists a solution $(g, \varphi)$ to 
\eqref{quasi-rescaled}  with positive $g$
satisfying $\| g(0)- \mu\|_{W^{S,1}(\T_M \times \R)} \leq \delta$ but such that
\begin{align}
& \theta_0  \le 
\sup_{t \in [0, t_\delta]}\left\| \int_\R g(t,x,v) \,dv - 1\right\|_{W^{-1,1}_{x}(\T_M)}
\le 
\sup_{t \in [0, t_\delta]}\left\| \int_\R g(t,x,v) \,dv - 1\right\|_{L^1_{x}(\T_M)} 
\nonumber
\\
&  \forall r \in \N^\ast, \quad \theta_r \le 
\sup_{t \in [0, t_\delta]}  \left\| \int_{\T_M} \Big(g(t,x,v) \,dx - \mu(v)\Big) \,dx\right\|_{W^{-r,1}_v},
\label{multi-insta}
\end{align}
with,  for a fixed $S$, $t_\delta = O(|\log \delta |)$ as $\delta \rightarrow 0$.
\end{thm}

\begin{rque} \label{rem:NLinsta}
 Assume that $\mu$ belongs only to the Sobolev space $W^{r,1}$, for some $r \in \N$. Then, the 
previous result is valid for $s \le r-1$ if the regularity index $r$ satisfies the condition
$$
r > 1+ \frac{\| \partial_v \mu \|_1}{\max \{ \Re \lambda, \lambda \in 
\sigma(L) \} },
$$
where $L$ is the linearized operator defined 
in~\eqref{quasi-rescaled-lin}.
\end{rque}
The proof of Theorem~\ref{NLinsta} follows a method introduced by Grenier 
in \cite{Gr} which is by now standard in instability theory for 
hydrodynamic equations.
We postpone it to the next subsection, after the conclusion of the proof of 
Theorem~\ref{thmGrenier-revisited}.

\bigskip
{\sl  Step 4. Back to the original variables.}

Let $s,N \in \N^*$. Take any $P \in \mathbb{N}$, such that $P> s+ N$. By 
Theorem \ref{NLinsta}, we find for all $\ep$ small enough a solution $(g_\ep, 
\varphi_\ep)$ to \eqref{quasi-rescaled}  satisfying $\| g_\ep(0)- 
\mu\|_{W^{s,1}} \leq \eps^P$ and~\eqref{multi-insta}. 
The associated instability time will be denoted $t_\eps = O(|\log \eps |)$, and 
the density by  $ \Lambda_\eps(t)= \int_\R g_\ep(t,x,v) \,dv$.

Next, a  consequence of the $\ep M$-periodicity of $f_\ep$ and of the change of variable~\eqref{ch-scaling} is that (at any time $t$):
\begin{equation} \label{W1p-scaling} 
\begin{aligned}
\| \rho_\ep -1 \|_{L^1} &= \frac1M \| \Lambda_\ep -1 \|_{L^1(\T_M)}, 
 \qquad \| \rho_\ep -1 \|_{W^{-1,1}} \ge  \frac{\eps}{C \, M} \| \Lambda_\ep -1 \|_{W^{-1,1}(\T_M)},  \\  
 \| f_\ep  - \mu \|_{W^{s,1}} &\le \frac{\ep^{-s}}M \| g_\ep -\mu  \|_{W^{s,1}(\T_M \times \R)}  \quad \text{for } \; s  \in \N.
\end{aligned}
\end{equation}
A constant $C>0$ appears in the second inequality of the first line because we are using non-homogeneous Sobolev norms.

Since the velocity variable is not affected by the scaling, we also have
\begin{equation} \label{W-rp-scaling} 
 \| f_\ep - \mu \|_{W^{-r,1}_{x,v}} \ge 
 \left\| \int_\T (f_\ep(t,x,v)-\mu(v))\,dx  \right\|_{W^{-r,1}_v} 
 \hspace{-10pt}
 = \frac1M  \left\| \int_{\T_M} (g_\ep(t,x,v)- \mu(v))\,dx \right\|_{W^{-r,1}_v}.
 \end{equation}
From this, we deduce that (for $\ep$ small enough)
\begin{align*}
& \| {f}_\eps(0)- \mu\|_{W^{s,1}} \leq  \frac{1}M \,
\eps^{P-s} \leq \eps^N, \\
&\frac{\theta_0}{C \,M} \le 
  \sup_{t \in [0, \eps t_\eps]}  \frac1\eps \|\rho_\ep-1 \|_{W^{-1,1}}
\le \sup_{t \in [0, \eps t_\eps]} \| \rho_\eps(t) - 1 \|_{L^1},   
 \\
& \theta_r \le \sup_{t \in [0, \eps t_\eps]} \| f_\eps(t) - \mu \|_{W^{-r,1}}.
\end{align*}
We therefore deduce the first point in $i)$ and $ii)$. 
From the bound from below on the $W^{-1,1}$ norm of $\rho_\eps-1$, using the Poisson equation in \eqref{quasi}, we deduce  the second point of $i)$. 
Finally note that $\eps \, t_\eps = O(\ep \, | \ln \ep|)$. 
We have thus completed the proof of Theorem \ref{thmGrenier-revisited}.

\begin{rque}
Remark that if $\mu$ is analytic, then the initial data $f_\ep(0)$ we build are also analytic, because of $ii)$ in Proposition \ref{prop:Deg}.
Nevertheless, it is not possible to get a similar theorem with the $W^{1,s}$-norm replaced by 
some analytic norm. Indeed, in this case, we rather expect stability, 
at least for short times: see for instance~\cite[Theorem 1.1.2]{Gr96} for a 
stability result with the ``superposition of fluids'' point of view, and also 
the work~\cite{BFJJ} about the well-posedness of the quasineutral equation.

We found it quite interesting to understand what in our proof  prevents us from keeping analytic norms from start to finish. This is precisely due to Step 4., where the rescaling in space (and especially the factor  $\ep^{-s}$ in the above bounds) prevents from getting ``uniform'' in $\eps$ analytic bounds.
\end{rque}

\subsection{Proof of Theorem \ref{NLinsta}}
\label{sec-insta}
We use Proposition \ref{prop:Penlim} and denote by $h_1$ an eigenfunction 
associated to an eigenvalue with maximal real
part of $L$, denoted by $\lambda_1$ (with 
$\Re \lambda_1 >0$), and such that $\rho_1:= \int h_1 \, dv \neq 0$. Up to a 
multiplication by a constant, we may
assume that $\|h_1\|_{L^1}$=1. Note that this implies that $0<\|\rho_1\|_{W^{-1,1}} \leq  \|\rho_1\|_1 \leq 1$.
Then, a good candidate for $g$ is the solution of the Vlasov-Poisson 
equation~\eqref{quasi-rescaled} with initial condition
$$
g(0) = \mu + \delta h_1,
$$
since according to the study of the linearized operator, loosely speaking, this 
solution will remain close (for small time) to
$$g^1_{app}(t) := \mu + \delta e^{\Re \lambda_1 t} h_1,
$$  
and thus ``escape'' from any small neighborhood of $\mu$. But the control of the errror between $g$ and its linear approximation $g^1_{app}$  on a sufficiently large time interval is not so straightforward: Grenier's method for overcoming that difficulty involves 
constructing a convenient high-order approximation of $g$.

Note  that the initial datum $g(0)$ defined above is a priori complex valued, since $h_1$ is. But we will show first that this $g(0)$ satisfies all the requested properties except it is not real, and explain in the last step of the proof, how to obtain from this $g(0)$ a non-negative initial condition with the requested properties.

\bigskip
{\sl Step 1. A formal high order approximation.}
Precisely, we look for a series of functions $g^N_{app}$
satisfying \eqref{quasi-rescaled} up to a small remainder $R^N_{app}$ 
$$
\partial_t g^N_{app} + 
v \, \partial_x g_{app}^N + 
\partial_x V^N_{app} \,\partial_v g^N_{app} = R^N_{app},
$$
where as usual $V^N_{app} := \partial_{xx}^{-1} (\int g^N_{app} \, dv - 1)$. 
The initial condition is the same as $g$ : $g^N_{app}(0) = g(0) = \mu + 
\delta h_1$.

The functions $g^N_{app}$ will be constructed as the partial sum of a 
series, whose terms will be defined by induction in order to decrease the order 
of the remainder $R^N_{app}$ at each step:
$$
g^N_{app}(t,x,v) = \mu(v) + \delta h_1(x,v) e^{ \lambda_1 t} + \sum_{i=2}^N 
\delta^i h_i (t,x,v).
$$
We will also use the notation $h_1(t,x,v) = h_1(x,v) e^{\lambda_1 t}$. Starting 
with $N=1$, we can see that $g_{app}^1 = \mu  + \delta
h_1(x,v) e^{ \lambda_1 t}$ is the solution to
$$
\partial_t g^1_{app} + 
v \, \partial_x g_{app}^1+ \partial_x V^1_{app} \,\partial_v g^1_{app} = 
\delta^2 
E_1 \partial_v h_1 = R^1_{app},
$$
with the notation $E_k := \partial_x \partial_{xx}^{-1} \bigl(\int h_k\,dv 
\bigr)$. In order to find the appropriate value of $h_2$, we can plug $g_2$ in 
the rescaled Vlasov-Poisson equation~\eqref{quasi-rescaled} and get
\begin{align*}
\partial_t g^2_{app} + 
v \, \partial_x g_{app}^2 +  \partial_x V^2_{app} \,\partial_v g^2_{app} = 
\delta^2  ( &\partial_t h_2 + L h_2 +  E_1 \partial_v h_1)\\
& + \delta^3(E_1 \partial_v h_2 + E_2 \partial_v h_1) 
+ \delta^4( E_2 \partial_v h_2).
 \end{align*}
We see that the best choice for $h_2$ and more generally for $h_k$ for $k\ge 2$ 
is to take it as the solution of 
\begin{equation} \label{def:hk}
 \partial_t h_k + L h_k + \sum_{j=1}^{k-1} E_j \partial_v  h_{k-j} =0,
\end{equation}
with $h_{k}(0)=0$ as initial datum. Of course, this can be done only if the 
$h_{j}$ are regular enough, but we will check this fact later. Then, the 
associated 
remainder term $R^N_{app}$ is given by
\begin{equation} \label{remainder}
R_{app}^N=  \sum_{N+1 \le j+j'\le 2N} \delta^{j+j'}E_j \partial_v h_{j'}.
\end{equation}
Remark also that in view of~\eqref{bound-lin} and the form of the source term 
in~\eqref{def:hk} (and up to some regularity issue), we allow $\| h_k \|_1$ to 
grow at most like $e^{k \Re \lambda_1 t }$. This implies that the remainder $\| 
R_{app}^N\|_1$ will grow at most like 
$\delta^{N+1} e^{(N+1) \Re \lambda_1 t}$, for not too large times. 

\bigskip
{\sl Step 2. A heuristic error estimate.}
Our first goal is to obtain good estimates on $\| g - g_{app}^N \|_1$. To 
this end, remark that $(g  - g_{app}^N)$ is solution to
$$
\partial_t (g  - g_{app}^N) + v \partial_x (g  - g_{app}^N) + \partial_x V 
\partial_v (g  - g_{app}^N) =  (\partial_x V_{app}^N - \partial_x V) 
\partial_v g^N_{app}  - R_{app}^N.
$$
If we multiply this equation by $ \mathrm{sign }  (g  - g_{app}^N)$ 
and integrate with respect to $x$ and $v$, we get
\begin{align}
\frac d {dt} \| g - g_{app}^N \|_1 & \le 
\| \partial_x V_{app}^N - \partial_x V \|_\infty 
\| \partial_v g^N_{app} \|_1 + \| R^N_{app}\|_1, \nonumber \\
& \le \| \partial_v g^N_{app} \|_1 \| g - g_{app}^N \|_1 + \| R^N_{app} \|_1.
\label{L1growth}
\end{align}
Assume that we are able to control $\| \partial_v (g^N_{app} - \mu)(t)  \|_1 \le 
1$
on a time interval $[0,T]$. It is reasonable to expect such a control since $\delta$ is small and we 
will show that the $h_k$ are smooth enough. Then, for $t \in [0,T]$, we obtain 
a 
bound
\begin{equation} \label{L1duha}
\| g(t) - g_{app}^N(t) \|_1  \le  \int_0^t  e^{(t-s) (\| \partial_v \mu \|_1 
+1)}   \| R^N_{app} (s) \|_1\,ds.
\end{equation}
If we take for granted the expected growth of the remainder $\| R_{app}^N\|_1 
\lesssim
\delta^{N+1} e^{(N+1) \Re \lambda_1 t}$, we get an estimate
\begin{align*}
\| g(t) - g_{app}^N(t) \|_1 & \lesssim \delta^{N+1} e^{\max((N+1) \Re 
\lambda_1, \| \partial_v \mu \|_1 +1) ) t }, \\
& \lesssim \bigl[ \delta e^{\Re \lambda_1 t }\bigr]^{N+1}
\qquad \text{if }\; (N+1) \Re \lambda_1 > \|\partial_v \mu \|_1 +1.
\end{align*}
The last bound will be very important for the following argument, since it 
allows to compare 
$\| g(t) - g_{app}^N(t) \|_1$ to $\| g_{app}^N(t) - \mu(t) \|_1 \approx 
\delta e^{\Re \lambda_1 t } $, almost independently of the time.  
So from now on, we fix (with the notation $\lfloor\cdot\rfloor$ for the integer 
part)
\begin{equation} \label{def:N}
N := \Big\lfloor  \frac{\|\partial_v \mu \|_1 +1}{\Re \lambda_1} \Big\rfloor.
\end{equation}

In fact, the factor $1$ added to the norm of $\partial_v \mu$ may be replaced by 
any 
positive real number, and that leads to the condition given 
in Remark~\ref{rem:NLinsta}.

\bigskip
{\sl Step 3. A rigorous error estimate.}
By assumption, we know that $\mu \in W^{N+1,1}$. By 
Proposition~\ref{prop:Penlim}, $h_1 
\in  W^{N,1}$, and thanks to its definition, $\| h_1(t) \|_{W^{N,1}} =  \| h_1 
\|_{W^{N,1}} e^{\Re \lambda_1 t}$.
We now show by 
recursion that for all $k \leq N$ there exists a constant $C_k>0$ such that for 
any 
time $t \ge 0$,
$$
\| h_k(t) \|_{W^{N-k+1,1}} \le C_k e^{k \Re \lambda_1 t}.
$$
By construction, this is true for $k=1$. We choose a $\Gamma \in (\Re \lambda_1, 
2\Re 
\lambda_1]$, and we assume the bound holds until rank $k$. Then for 
the rank $k+1$, by the definition~\eqref{def:hk}, the 
estimate~\eqref{bound-lin} on the semi-group generated by $L$, and with the help 
of Duhamel's formula, we get
\begin{eqnarray*}
 \Vert h_{k+1} \Vert_{W^{N-k,1}} &\leq&  \sum_{j=1}^{k} \int_0^t \Vert 
e^{(t-s) L}\bigl( E_j \partial_v h_{k+1-j} \bigr)\Vert_{W^{N-k,1}}ds \\
&\leq&   C_{\Gamma}^{N-k}  \sum_{j=1}^{k}
 \int_0^t  e^{\Gamma(t-s)} \Vert E_j \Vert_{W^{N-k,\infty}} \Vert 
\partial_v h_{k+1-j}\Vert_{W^{N-k,1}} \,ds \\
&\leq&   C_{\Gamma}^{N-k}  \sum_{j=1}^{k}
 \int_0^t  e^{\Gamma(t-s)} \Vert h_j \Vert_{W^{N-k,1}} \Vert 
h_{k+1-j}\Vert_{W^{N-k+1,1}} \,ds \\
&\leq&   C_{\Gamma}^{N-k}   \biggl( \sum_{j=1}^{k}  C_{j} C_{k+1-j} \biggr)    \int_0^t  e^{\Gamma(t-s)}  e^{(k+1)\Re \lambda_1 
s}ds \\
&\leq& C_{k+1}\,  e^{(k+1)\Re \lambda_1 t}.
\end{eqnarray*}
Remark that we have used the Sobolev embedding~\eqref{estim:phi}.
By formula~\eqref{remainder}, we obtain also, for the remainder, the bound
\begin{equation} \label{bound:remainder}
\|R_{app}^N(t)\|_{L^1} \leq C_N' \, \delta^{N+1} e^{(N+1) \Re \lambda_1
t}, \qquad
\text{as long as } \; \delta e^{\Re \lambda_1 t} <1.
\end{equation}

\bigskip
{\sl Step 4.  Instability with complex valued approximation.}
We can now estimate $\| g^N_{app} - g^1_{app} \|_1$ and the term $\| \partial_v g^N_{app} \|_1$ that appears  in~\eqref{L1growth}.  We introduce with $C_N'' =\max_{k \le N} C_k$
\begin{equation*}
\theta_{max} := \frac {\|\rho_1\|_{W^{-1,1}}} {3 \max 
(1,C_N',C_N'')}, \qquad  
t_{max} := \frac1{\Re \lambda_1} \ln \Bigl(\frac{\theta_{max}}\delta \Bigr),
\end{equation*}
so that $\delta e^{\Re \lambda_1 t} \le \theta_{max}$ if and only if $t \le t_{max}$. Then,
\begin{align}
\|  g_{app}^1(t) - g_{app}^N(t) \|_1 
& \le  \sum_{k=2}^N \delta^k \|h_k(t) \|_1 \le 
\sum_{k=2}^N  C_k \bigl[ \delta e^{\Re \lambda_1 t} \bigr]^k \nonumber \\
& \le C''_N  
\frac{ \bigl[ \delta  e^{\Re \lambda_1 t} \bigr]^2}{1 - \delta  e^{\Re \lambda_1 t}}
\le \frac \delta 2 \,  \|\rho_1\|_{W^{-1,1}}\,  e^{\Re \lambda_1 t}, \qquad \text{for } t \le t_{max}. \label{gNapp-g1app}
\end{align}
Similarly,
\begin{align*}
\| \partial_v (g^N_{app} (t)- \mu)\|_1 & \leq 
\sum_{k=1}^N \delta^k \| \partial_v h_k \|_1 
\le \sum_{k=1}^N C_k \delta^k e^{k \, \Re \lambda_1 t} \\
& \le C_N'' \frac{ \delta  e^{\Re \lambda_1 t}}{1 - \delta  e^{\Re \lambda_1 
t}} \le \frac12  
\qquad \text{as long as } t \le t_{max}.
\end{align*}
As a consequence, for $t \le t_{max}$, we can 
apply~\eqref{L1duha} and using estimate~\eqref{bound:remainder} and 
definition~\eqref{def:N}, it comes
\begin{align}
\| g(t) - g_{app}^N(t) \|_1  & \le  C_N' \int_0^t  e^{(t-s) (\| \partial_v \mu 
\|_1 +1)} \, \delta^{N+1} e^{(N+1) \Re \lambda_1 s}  \,ds, \nonumber \\
& \le C_N' \bigl[ \delta e^{\Re \lambda_1 t }\bigr]^{N+1}
= C_N' \bigl[ \theta_{max}]^N \, \delta \, e^{\Re \lambda_1 t }
\le \frac \delta 3 \,  \|\rho_1\|_{W^{-1,1}} \, e^{\Re \lambda_1 t }
. \label{g-gNapp}
\end{align}
Gathering~\eqref{gNapp-g1app} and~\eqref{g-gNapp} we obtain a good control of the error in the approximation of $g$ by $g^1_{app}$:
\begin{equation}
\| g(t) - g_{app}^1(t) \|_1 \le \frac{5 \delta}6 \, \|\rho_1\|_{W^{-1,1}} \, e^{\Re \lambda_1 t }, \qquad \text{for } t \le t_{max}.
\end{equation}
This implies that for $t \le t_{max}$
\begin{align}
\left\| \int g(t) \, dv - 1 \right\|_{W^{-1,1}} & \ge\left\| \int g_{app}^1(t) \, dv  - 
1 \right\|_{W^{-1,1}}  -  \| g(t) - g_{app}^1(t) 
\|_1 \nonumber \\
& \ge \delta \, \|\rho_1\|_{W^{-1,1}} \, e^{\Re \lambda_1 t} 
-  \frac {5 \delta}6  \|\rho_1\|_{W^{-1,1}}\, e^{\Re \lambda_1 t}
=  \frac \delta 6  \|\rho_1\|_{W^{-1,1}}\, e^{\Re \lambda_1 t}. \label{estim-position}
\end{align}
So in particular, at $t=t_{max}$, we have $\left\| \int g(t) \, dv - \mu 
\right\|_1\ge  \frac16 
\theta_{max} \| \rho_1 \|_{W^{-1,1}} $. Thus, with the notation of Theorem~\ref{NLinsta}, 
we can choose 
$\theta_0 = \frac16 \theta_{max} \| \rho_1\|_{W^{-1,1}}$.  

\bigskip

For what concerns the initial condition, for a fixed $s$, we have
$$
\| g(0) - \mu \|_{W^{s,1}} = \delta \, \| h_1\|_{W^{1,s}} = C_s \,  \delta,
$$
since $\mu$ belongs to all the $W^{s,1}$ spaces by assumption and so does $h_1$ by 
Proposition \ref{prop:Penlim}. From this, we see that $\| g(0) - \mu 
\|_{W^{s,1}}$ can be 
made as small as we want, and that in any case the density $\int g(t) \,dv$ will move 
away from $1$ by a distance of at least $\theta_0$ in $W^{-1,1}$-norm, and this 
before a time  of order $\ln \bigl(\| g(0) - \mu \|_{W^{s,1}}^{-1} \bigr)$.

\bigskip
{\sl Step 5. The construction of a real initial condition. }
In order to construct a \emph{real} initial condition with the requested properties, we remark that $\bar h_1$ is also an eigenfunction of $L$, associated to the eigenvalue $\bar \lambda_1$, and thus we choose as new initial condition
\begin{equation} \label{def:tildeg0}
 \tilde g(0) := \mu + \frac \delta 2 \,  [ h_1 + \overline{ h_1}].
\end{equation}
Remark that $\Re h_1 \neq 0$ (recall Proposition~\ref{prop:Penlim}). 
Then, using Grenier's method with this starting point, it can be shown that the solution $\tilde g$ to~\eqref{quasi} starting form $\tilde g(0)$ has the same properties as those of $g$.
The same proof can be performed again (we shall now write it again for the sake of conciseness), by considering similar functions (we will systematically add a tilde when we will refer to them in the following). We just remark 
that we should replace $\| h_1\|_1$ by $\| \Re h_1\|_1$, 
  $\| \rho_1\|_{W^{-1,1}}$ by $\|  \Re \rho_1\|_{W^{-1,1}}$, 
and  that in~\eqref{estim-position}, the term $\int \tilde g^1_{app}(t) \,dv$ will now oscillate;
but the bound by above is still true if $t$ is a multiple of $2\pi$, and this is sufficient to conclude.

In the case were $\mu$ satisfies the $\delta$-condition of Definition~\ref{cond-pos}, then  the particular 
form of $h_1$ given in~\eqref{eigen-form} implies that $  \delta |h_1| \le \mu$ for $\delta$ small enough, and then that the $\tilde g(0)$ defined in~\eqref{def:tildeg0} is non-negative.

 But nothing ensures that the $\tilde g$ constructed above is nonnegative when 
 $\mu$ satisfies only the $\delta'$-condition of Definition~\ref{cond-pos'}. 
 This will require some truncation argument, which will be performed in the next subsection, 
 after the precise statement of the $\delta'$-condition.

 Before, let us now study the instability in $W^{-r,1}_v$ for the averages in $x$.

\bigskip
{\sl Step 6. The instability in $W^{-r,1}_v$ on the average in position.}
By Proposition~\ref{prop:Deg}, $\tilde h_1(t,x,v)$ 
and its associated density $\tilde \rho_1$
are of the form below for some $n \in \Z^\ast$, $\xi \in \C \backslash \R$ and $\kappa>0$
\begin{equation} \label{h1-form}
\tilde h_1(t,x,v) = \frac12 \Bigl(  e^{\lambda_1 t + i \frac{2\pi n}{M} x} \, \frac{\mu'(v)}{v + \xi}+ 
e^{\overline{\lambda_1} t - i \frac{2\pi n}{M} x} \, \frac{\mu'(v)}{v + \overline \xi} \Bigr),
\qquad 
\tilde \rho_1(t,x) = \kappa\,  e^{\Re \lambda_1 t} \cos\left( \Im \lambda_1 t + \frac{2\pi n}{M} x\right). 
\end{equation}

The fact that $\kappa$ is a positive real number is a consequence of the dispersion relation~\eqref{dispersion}.
Remark that $x$-average of $\tilde h_1$ vanishes. Thus, in order to see an instability on the $x$-average we should also study $\tilde h_2$.  But $\tilde h_2$  is solution to
$$
\pa_t \tilde h_2 + L \tilde h_2 +  \tilde E_1\, \pa_v \tilde h_1 =0.
$$
Integrating with respect to $x$, and using that $\int L \tilde h_2 \,dx = 0$, we get
$$
\pa_t \biggl( \int \tilde h_2(t,x,v) \,dx \biggr) = - \int  \tilde E_1(t,x) \, \pa_v \tilde h_1(t,x,v) \,dx.
$$
Using~\eqref{h1-form} and $\pa_x \tilde E_1 = \tilde \rho_1$, we get
 $ \tilde E_1 = \frac{\kappa M}{2\pi n} e^{\Re \lambda_1 t} \sin\left(  \Im \lambda_1 t + \frac{2\pi n}{M} x \right)$, and this allows to calculate the r.h.s. in the last equation. After a short calculation
 and a integration in time, we get
\begin{equation} \label{h2-form}
\int_{\T_M} \tilde h_2(t,x,v) \,dx = 
- \frac{\kappa M}{4\pi n \, \Re \lambda_1} 
\bigl[   e^{2 \Re \lambda_1 t} -1 \bigr]
\ell'(v),
\end{equation}
where $\ell$ is a smooth function defined by
$$
\ell(v) := \Im  \Bigl[ \frac{\mu'(v)}{v + \xi}\Bigr] = 
\frac{\Im \xi \, \mu'(v)}{(v + \Re \xi)^2 + (\Im \xi)^2}.
$$
In particular, we have for any $r \in \N$
\begin{align*} 
\left\|  \int \tilde h_2(t,x,v) \,dx  \right\|_{W^{-r,1}_v} &  =
\sup_{\| \varphi \|_{W^{r,\infty}} \le 1} \int \tilde h_2(t,x,v) \varphi(v) \,dx \,dv \\
& \ge
\frac1{\|\ell  \|_{W^{r+1,\infty}}}  \int \tilde h_2(t,x,v)  \ell'(v) \,dx \,dv \\
& = 
\frac {\kappa M \| \ell'\|_2^2} {4\pi n \, \Re \lambda_1 \,\|\ell  \|_{W^{r+1,\infty}}} 
\bigl[   e^{2 \Re \lambda_1 t} -1 \bigr] =: c_r'  \bigl[   e^{2 \Re \lambda_1 t} -1 \bigr].
\end{align*}
In particular, remark that since $c_r'$ is a well-defined constant since 
$\ell$ is as smooth as $\mu'$ and also
$\| \ell'\|_2$ is finite since $\mu'' \in L^1 \cap L^\infty$. 
Therefore, the previous bound by below leads to
$$
\left\|  \int \Big(\tilde g^2_{app}(t,x,v)  - \mu(v)\Big) \,dx \right\|_{W^{-r,1}_v} 
\ge c_r' \delta^2 \, \bigl[   e^{2 \Re \lambda_1 t} -1 \bigr].
$$
Starting from this inequality, the strategy of the Step 4 can be performed again, and we can obtain (up to some redefinition of $\theta_{max}$ and $t_{max}$) the conclusion claimed in Theorem~\ref{NLinsta}. In fact all the remainder terms $ \tilde g - \tilde g^N_{app}$
and $\tilde g^N_{app} - \tilde g^2_{app}$ are controlled without integration in $x$, in a stronger topology (namely $L^1$) and at a smaller order (at most $[\delta e^{\Re \lambda_1t}]^3$).

  This conclude the proof in the case where $\mu$ satisfies the $\delta$-condition of Definition~\ref{cond-pos}.


\subsection{The alternative $\delta'$-condition. }
\label{subsec:delta'}

\begin{deft}
\label{cond-pos'}
For any non-negative $C^1$ profile $\mu(v)$, and any $\delta>0$, we define 
\begin{align} \label{def:delta}
V_\delta := \biggl\{v \in \R, \, \text{ s.t. } \frac{|\mu'(v)|}{1+ |v|} > \frac1 \delta \mu(v) \biggr\}
\quad  \subset \R, \\
W_\delta :=  \{w \in \R, \;\text{s.t. } d(w,V_\delta) \le \sqrt \delta \},
\end{align} 
where $d$ stand for the usual distance from a point to a set.
We say that $\mu$ satisfies the $\delta'$-condition 
if for any $n \in  \N$, 
\begin{equation}\label{cond:alpha'}
\liminf_{\delta \rightarrow 0}  \frac{1}{\delta^n} \, \int_{ W_\delta}  | \mu'(v)| \,dv= 0.
\end{equation}
\end{deft}

With that new condition, the conclusion of Theorem~\ref{thmGrenier-revisited} still holds. We provide in the step~$7$ belows the truncation argument  (see also \cite{GS98} for a similar construction).

\bigskip \noindent
{\sl Step 7 of the proof of Theorem~\ref{thmGrenier-revisited}. The construction of a non-negative initial condition. }

 Our goal is now to show how to construct a relevant non-negative initial condition.
Recall that from Proposition~\ref{prop:Deg}, the eigenfunction satisfies
for some $m \in \N$, $\xi \in \C \backslash \R$ and thus for some $C_1 \geq 1$  
$$
| h_1(x,v)| = \left| e^{i m x}\frac{\mu'(v)}{v+ \xi} \right| \le C_1 \frac{|\mu'(v)|}{1 +|v|}.
$$
Remark that the real initial condition $\tilde g(0)$ defined in a previous step may take negative value at any point $v$ where
$$
\delta |h_1|(v) > \mu(v).
$$ 
In order to ``remove'' such problematic points, we introduce a smooth cut-off function $k : \R \rightarrow [0,1]$ such that $k(x) = 0$ for $x \le 0$, $k(x)=1$ when $x \ge 1$, and define
$$
V_\delta' := \biggl\{v \in \R, \, \text{ s.t. } |h_1|(v) > \frac1 \delta \mu(v) \biggr\},
$$
and its $\sqrt \delta$-neighborhood $W_\delta'$. Remark that  $V_\delta'$ is related to $V_\delta$ defined in~\ref{def:delta}: precisely we have $V_\delta \subset V_{C_1 \delta}$ and 
$W_\delta \subset W_{C_1 \delta}$ since $C_1>1$,
so that the property~\eqref{cond:alpha'} is still true with $W_\delta$ replaced by $W_\delta'$.

We also define $G_\delta' := \bigl\{ w, \; \text{ s.t. } d(w,V_\delta') \ge \frac12 \sqrt \delta  \bigr\}$, and denote by $\chi_\delta$ its characteristic function.
Then  we choose $\eta$ a smooth function with total mass one and support in $[-1,1]$, and define for any $\delta >0$,  $\eta_\delta := 2 \delta^{-\frac12} \eta \bigl( \frac12 \sqrt \delta \, \cdot \bigr)$, which has still total mass one and a support in 
$ \bigl[ - \frac12 \sqrt \delta, \frac12 \sqrt \delta \bigr]$. Then we define
$$
h^\delta_1(x,v) := h_1(x,v) \, [\chi_\delta \ast \eta_\delta](v).
$$
Then $h_1^\delta$ satisfies the following properties:
\begin{enumerate}[a)]
\item $h^\delta_1 = h_1$ on $\R \backslash W_\delta'$;
\item $h^\delta_1 = 0$ on  $V_\delta'$, so that $|h^\delta_1| \le \delta \mu$;
\item for any $s \in \N$, 
\begin{equation} \label{Ws1-bound}
\| h_1^\delta \|_{W^{s,1}_{x,v}} \le C \delta^{-\frac s 2}
\| h_1\|_{W^{s,1}_{x,v}}.
\end{equation}
\end{enumerate}
From now on, we fix $n=\max(N+1,S)$. In view of~\eqref{cond:alpha'} satisfied by $\mu$, there exists a sequence of positive number $(\delta_k)_{k \in \N}$ converging to $0$ such that
$$
\lim_{k \rightarrow + \infty} \frac1{\delta_k^n}
\int_{ v \in W_{\delta_k}'}  | \mu'(v)| \,dv = 0.
$$
From now on, we assume that $\delta$ take only the values $\delta_k$ of that sequence, but do not write the indice $k$ for readability. We have using point $a)$ above and \eqref{cond:alpha'}
\begin{equation} \label{L1-bound}
\| h_1 - h^\delta_1 \|_{L^1} \le  C_1 \int_{ v \in W_\delta'} |\mu'(v)| \,dv  = o(\delta^n ),
\end{equation}
But now by interpolation (see for instance \cite{BL}), using~\eqref{Ws1-bound}
and~\eqref{L1-bound}, we get for some $C>0$,
\begin{align*}
\| h_1 - h^\delta_1 \|_{W^{n,1}} & \le 
C \sqrt{
\| h_1 - h^\delta_1 \|_{W^{2n,1}}
\| h_1 - h^\delta_1 \|_{L^1}  } \\
& = \sqrt{o( \delta^{-n} \delta^n )} = o(1).
\end{align*}
Hence, since $S \le n$, $ \| h^\delta_1 \|_{W^{S,1}}$ is bounded independently
of $\delta$.
Thus, if we define, similarly as in~\eqref{def:tildeg0} an initial condition
$$
\tilde g^\delta(0) := \mu + \frac \delta 2 ( h_1^\delta + \overline{ h_1^\delta}  ),
$$
then we have ensured that is non-negative, and that $\| \tilde g^\delta(0) - \mu \|_{W^{S,1}} \le C \delta$. 
We denote $\tilde g^\delta$ the solution to~\eqref{quasi} with initial condition
$\tilde g^\delta(0)$.
To compare $\tilde g^\delta$ to the previous approximation $\tilde g^N_{app}$, we can still apply~\eqref{L1duha} if we add a term for the difference at initial time that does not vanishes anymore:
\begin{align*}
\| \tilde g^\delta(t) - \tilde g_{app}^N(t) \|_1  & \le  \int_0^t  e^{(t-s) (\| \partial_v \mu \|_1  +1)}   \| \tilde R^N_{app} (s) \|_1\,ds +
e^{t (\| \partial_v \mu \|_1  +1)}  \|  \tilde g^\delta(0) - \tilde g(0) \|_1,
\end{align*}
and the previous analysis can still be done since by~\eqref{L1-bound}, as $N+1 \le n$, we have
$$\|  \tilde g^\delta(0) - \tilde g(0) \|_{L^1} = o(\delta^{N+1})  .$$
This concludes the proof.


\subsection{Proof of Proposition~\ref{suff-delta}.}

First remark that the $\delta'$-condition is weaker than the $\delta$-condition, since the later implies that $V_\delta = W_\delta = \emptyset$, for $\delta$ small enough, so that $\int_{W_\delta} |\mu'(v)| \,dv = 0$ and condition~\eqref{cond:alpha'} clearly holds.

\bigskip \noindent $\bullet$
{\sl  Point $i.$ implies the $\delta'$-condition.} 

First, the positivity of $\mu$ implies that for $R>0$ large enough, $V_\delta \cap [-R,R] = \emptyset $, and also $W_\delta \cap [-R,R] = \emptyset $. Then, the upper bound in~\eqref{cond:2.1i} implies that for some constant $c>0$
$$
V_ \delta \subset W_ \delta \subset  \Bigl\{  v, \; \text{s.t. }  |v| \ge  v_\delta \Bigr\},
\qquad \text{with} \qquad 
v_\delta :=  c  \, \delta^{-\frac1{\alpha-1}}.
$$
Next, remark that the lower bound in~\eqref{cond:2.1i} and the smoothness of $\mu$ forbids $\mu'$ to change its sign for $|v|$ large enough, so that, $\mu'$ is necessarily negative for large $v$.
Then, a straightforward integration of the lower bound in~\eqref{cond:2.1i}, leads for $|v|$ large enough to the inequality
$$
\mu(v) \le C e^{- \frac1{C_\alpha} |v|^{\alpha+1}}.
$$
But now
\begin{align*}
\int_{W_\delta} |\mu'(v)| \,dv 
& \le  \int_{-\infty}^{-v_\delta}   |\mu'(v)| \,dv + \int^{+\infty}_{v_\delta}   |\mu'(v)| \,dv 
 = |\mu(v_\delta)| + |\mu(-v_\delta)| \\
&  \le 2C e^{- \frac1{C_\alpha} |v_\delta|^{\alpha+1}}
 \le  2C e^{- c'  \delta^{-\gamma}} 
\quad  \text{with} \quad \gamma := \frac{\alpha+1}{\alpha-1}
\end{align*}
and the quantity in the last r.h.s. is a $o(\delta^n)$ for any $n \in \N$. 

\bigskip \noindent $\bullet$
{\sl  Point $ii.$ implies the $\delta'$-condition.} 

Using the positivity of $\mu$ on the interior of the $(a_i,b_i)$, the upper bound in~\eqref{cond:2.1ii} and arguing similarly to the previous step, we obtain that for some constants $c_i, c_i'>0$
$$
V_\delta \subset \bigcup_{i} \; (a_i, a_i^\delta) \; \cup  \; \bigcup_{i} \;  (b_i^\delta,b_i^\delta),
\quad
\text{where} \quad  
a_i^\delta := a_i + c_i \delta^{\frac1\beta},\quad
 b_i^\delta := b_i   - c_i' \delta^{\frac1\beta}.
$$
It then implies with $\beta':= \max(\beta,2)$ that for some different constants $c_i, c_i'>0$,
$$
W_\delta \subset \bigcup_{i} \; (a_i, \tilde a_i^\delta) \; \cup  \; \bigcup_{i} \;  (\tilde b_i^\delta,b_i^\delta),
\quad
\text{where} \quad  
\tilde a_i^\delta := a_i + c_i \delta^{\frac1{\beta'}},\quad
\tilde  b_i^\delta := b_i - c_i' \delta^{\frac1{\beta'}}.
$$
Moreover, the lower bound in~\eqref{cond:2.1ii}  implies that $\mu'$ does not change its sign closely above $a_i$ (and also closely below $b_i$). Then,
\begin{align*}
\int_{W_\delta} |\mu'(v)| \,dv  
& \le 
\sum_i \int_{a_i}^{\tilde a_i^\delta} |\mu'(v)|\,dv +
\sum_i \int^{b_i}_{\tilde b_i^\delta} |\mu'(v)|\,dv \\
& \le  \sum_i \mu\bigl(\tilde a_i^\delta\bigr) +
\sum_i \mu\bigl(\tilde b_i^\delta\bigr).
\end{align*}
But since $\mu$ is smooth and vanishes at any order at $a_i$ and $b_i$, we have that 
$\mu\bigl(\tilde a_i^\delta\bigr) = o(\delta^n)$ for any $n \in \N$ (we have as well a similar behavior for the $\tilde b_i^\delta$). This implies condition~\eqref{cond:alpha'}.

\bigskip \noindent $\bullet$
{\sl  The case where $\mu$ has zero inside its support.} 

We keep the notation of the previous step, but assume now that $a_1$ is a zero of finite order $m$ of $\mu$. 
 Then the ratio $\frac{\mu'}{(1+|v|)\mu(v)}$ behaves like $\frac c{(v-a_1)}$, and from the definition of $V_\delta$ we see that for some constant $c_1>0$
 $$
 (a_1, a_1^\delta) \subset V_\delta \subset W_\delta,
 \quad 
 \text{with } 
 a_1^\delta:= a_1 + c_1 \delta^{-1}.
 $$
Then, we have 
$$
\int_{W_\delta} |\mu'(v)| \,dv  \ge 
\int_{a_1}^{a^\delta_1} |\mu'(v)| \,dv  = \mu(a_1^\delta) \sim c \delta^m,
$$
for some $c >0$, so that 
 the condition~\eqref{cond:alpha'} is not satisfied for $n > m$. 

\Black

%
%
\section{Stable case: proof of Theorems \ref{thm} and \ref{thm-IP}}
\label{sec:conv}

\subsection{Some properties of the Casimir functional $H_Q$.}
\label{subsec:HQ}

We gather in the following Proposition some useful properties of the Casimir functional.
\begin{prop}\label{prop:HQ}
Let $\mu$ be a $S$-stable profile (See definition~\ref{def:S-stable}) 
associated by~\eqref{condphi}  to a profile $\varphi$. Let $H_Q$ be an associated 
Casimir functional (See Definition~\ref{def:Casimir}) defined thanks to an 
admissible function $Q$. Then :
\begin{itemize}
  \item[i)]  In general, the quantity $H_Q(f)$ is well defined in $\R^+ 
\cup \{ 
+\infty \}$ as the integral of a non-negative measurable function. But if in 
addition, $f$ satisfies
\begin{equation} \label{energyplus}
\int f (1+ v^{2 + \eta}) \,dxdv < + \infty, \quad 
\text{for some } \eta>0,
\end{equation}
then the integrals $\int Q(\mu)$, $\int Q'(\mu) f$ are finite and $\int Q(f)$ 
is bounded from below.

The same also holds if the profile $\mu$ satisfies 
\begin{equation} \label{energyplus-mu}
\int \mu (1+ v^{2 + \eta'}) dv < + \infty, \quad 
\text{for some } \eta'>0,
\end{equation}
and $f \in L^1$ and has finite kinetic energy.  

  \item[ii)] $H_Q$ is convex and non-negative, lower semi-continuous on the 
space of functions $f \in L^1$ with finite kinetic energy. 
  It vanishes only for $f = \mu$.
If moreover $Q$ is uniformly convex:  $Q'' \ge \alpha$, for some $\alpha 
 >0$, then $H_Q$ control the $L^2$ norm 
 $$
 \| f - \mu \|_2^2 \le \frac 1 \alpha H_Q(f).
 $$
 \item[iii)] When $\mu$ is a Maxwellian, $\mu(v) = \frac1{\sqrt{2\pi T}} e^{- 
\frac{|v-u|^2}{2T}}$ for some $T>0$ and $u \in \R$, then the choice $Q(s)= T 
\,s \ln s $ is admissible (up to a constant) and the associated Casimir 
functional $H_Q$ is just the usual relative entropy:
$$
H_Q(f) := H(f \vert \mu) = \int_\R  \ln \frac f \mu  \,f \, dx.
$$
Moreover, the Csisz\'ar-Kullback-Pinsker inequality implies that, for all $f \in L^1_{x,v}$, with $f\geq 0, \int f =1$, we have 
$$
\|f- \mu\|_{L^1_{x,v}}^2 \leq \frac12 H_Q(f).
$$
 \item[iv)] If a sequence $\mu_n$ of $S$-stable profile converges weakly 
towards a Dirac mass $\delta_{\bar v}$, for some $\bar v$ in $\R$, in such a way that
\begin{equation} \label{cond:conv}
\int_\R | v - \bar v |^2 \mu_n(v) \,dv \xrightarrow[n \rightarrow +\infty]{} 0,
\end{equation}
then 
for any compatible sequence $H_n$ of functionals, we have for any bounded $f \ge 
0$ satisfying~\eqref{energyplus}
$$
H_n(f) \xrightarrow[n \rightarrow +\infty]{\mbox{}} \frac{1}{2} \int_{\R\times \T} f(x,v) |v - \bar 
v|^2 \,dv dx.
$$
\end{itemize}
\end{prop}

\begin{rque}
The last point is interesting because it tells us that the natural extension of 
the Casimir functional in the case where the profile is a Dirac mass, is 
$\frac{1}{2}\int |v - \bar v|^2 f \, dv dx$. This is exactly the quantity that 
is introduced when a modulated energy method is used in the zero temperature 
limit \cite{Br00,Mas,HK}. 
\end{rque}

\begin{proof}[Proof of Proposition \ref{prop:HQ}] \mbox{}

{\sl Proof of the point $i)$. }
We shall first prove that condition~\eqref{condphi} and the monotonicity 
of $\varphi$ imply that $ \varphi(u) |u|^{\frac32} \le C$ for some $C \in 
\R^+$. In fact, since $\varphi$ is increasing, we have the following inequality
$$
  \sum_{n =0}^\infty 2^{\frac{3n}2} \varphi(-2^{n+1})  = \sum_{n =0}^\infty (2^{n+1} - 2^n) 2^{\frac n2} \varphi(-2^{n+1}) \le \int_{-\infty}^0  \sqrt{-u} \varphi(u) \,du < +\infty.
$$
It implies the boundedness of the sequence $2^{\frac{3n}2} \varphi(-2^{n+1})$, 
from which we conclude using monotonicity, that  there exists a constant $C$ 
such that
\begin{equation}
\label{condphi-consequence}
\forall u \in \R^-,\quad  0 < \varphi(u) \le  C (1-  u)^{-\frac32} .
\end{equation}
Denoting $a = \varphi(0)$ and using the requirements on $Q$ (See 
Definition~\ref{def:Casimir}), this implies that 
\begin{equation} \label{ineq:phiQ}
\forall z \in (0,a], \qquad 
 - C^{\frac23} z^{-\frac23} \le \varphi^{-1} (z) 
<0, \quad \text{and} \quad
  - \frac13 C^{\frac23} z^{\frac13} \le  Q(z) \le 0.
\end{equation}

\bigskip
Next, the term $\int Q'(\mu) \mu\,dxdv$ which appears in the definition of 
$\LL_\ep$ may be rewritten using that
\begin{equation} \label{eq:Qmu}
Q' [\mu(v)] = \varphi^{-1} \circ \varphi \Bigl( - \frac{\vert v - \bar v 
\vert^2}2\Bigl) = - \frac{\vert v - \bar v \vert^2}2.
\end{equation}
Thus, it comes 
$$
\int Q'(\mu) \mu \,dx\,dv = - \frac12 \int  \vert v - \bar v \vert^2 \varphi 
\Bigl( - \frac{\vert v - \bar v \vert^2}2\Bigl) \,dv,
$$
which is finite under condition~\eqref{condphi}. Similarly, we  see that the 
term $\int Q'(\mu) f\,dxdv$ is finite when $f \in L^1$ has finite kinetic 
energy.

\bigskip
Next, the term $\int Q(\mu) \,dxdv$ is also finite. To see this, use that $Q'
= \varphi^{-1}$ on the range of $\varphi$, the assumption $Q(0)=0$ to 
write for all $v$
$$
Q(\mu(v))=  H(u) := \int_0^{\varphi(u)} \varphi^{-1}(r) \,dr 
, \quad \text{with} \; u := - \frac{\vert v -\bar v\vert^2}2.
$$ 
Here we can use the following relation, which is clear from inspection of the 
graph of $\varphi$ (or a formal differentiation)
$$
H(u) = \int_0^{\varphi(u)} \varphi^{-1}(r) \,dr  =  u  
\varphi(u)  - \int_{-\infty}^u \varphi(s) \,ds.
$$
It leads to
\begin{align}
\int_{\T \times \R} Q(\mu(v)) \,dxdv & = \int_{-\infty}^0 H(u) 
\frac{du}{\sqrt{-u}} 
\nonumber \\
& =   - \int_{-\infty}^0 \sqrt{- u}  \, \varphi(u) \,du
- \int_{-\infty}^0 \left( \int_{-\infty}^u \varphi(s) \,ds\right) \frac{du}{\sqrt{-u}}
\nonumber \\
& =  - \int_{-\infty}^0 \sqrt{- u} \, \varphi(u) \,du
- \int_{-\infty}^0 \left( \int_s^0 \frac{du}{\sqrt{-u}} \right) \varphi(s) \,ds 
\nonumber \\
& = -3 \int_{-\infty}^0 \sqrt{- u}  \,\varphi(u) \,du,
\label{egal:Qmu}
\end{align}
 which is finite by assumption~\eqref{condphi}.
 
 \bigskip
Finally, using the bound by below~\eqref{ineq:phiQ} for $Q$ on $[0,a]$ and 
the simpler bound $Q(z) \ge - b z  - c$ for $z \ge a$ (such 
nonnegative constants $b$ and $c$ exist since $Q$ is convex), we get
\begin{align}
\int_{\T \times \R} Q(f) \,dx\,dv  & = \int_{\{ f \le a \}} Q(f) \,dx\,dv +
\int_{\{ f \ge a \}} Q(f) \,dx\,dv 
\nonumber \\
 & \ge -C \int f^{\frac13} \,dx\,dv   -  \int_{\{ f \ge a \}} (b f +c) 
\,dx\,dv 
\nonumber \\
& \ge -C \int \bigl[f(1+ v^{2+\eta})\bigr]^{\frac13} \frac{dx\,dv}{(1+v^{2+ 
\eta})^{\frac13}} - b \|f\|_1 -c \int_{\{ f \ge a \}} 
\,dx\,dv 
\nonumber \\
& \ge  -C \biggl[ \int f(1+ v^{2+\eta})\,dx\,dv \biggr]^{\frac13} 
\biggl[ \int (1+ v^{2+\eta})^{-\frac12 }\,dx\,dv \biggr]^{\frac23} - \Bigl( b + 
\frac c a  \Bigr) \| f\|_1 
\nonumber \\
&\ge -  C_\eta  \biggl[ \int f(1+ v^{2+\eta})\,dx\,dv \biggr]^{\frac13} - 
\Bigl( b 
+ \frac c a  \Bigr) \| f\|_1,
\label{bound:Qf}
\end{align}
where we have used on the fourth line the H\"older inequality. 

The case where 
 $\mu$ satisfies \eqref{energyplus-mu} allows to improve the 
bound from below~\eqref{ineq:phiQ} by
$$
\forall z \in (0,a],  \qquad
  - C' z^{\frac{1+\eta}{3+\eta}} \le  Q(z) \le 0.
$$
Having seen this, the same calculations can be done using only that $f$ has finite kinetic energy.

 All in all, we see that all the integrals composing $H_Q$ are well 
defined: the first one belongs to $\R 
\cup \{+\infty\}$ and the three other ones are finite.

 \bigskip
{ \sl Proof of Points $ii)$ and $iii)$.}
The convexity and nonnegativity of $H_Q$ are clear. The convexity of $Q$ implies 
that the first term in the definition~\eqref{HQintro} of $H_Q$ is l.s.c.. Two 
others are constant, and the last one may be rewritten $\int |v - \bar v|^2 
f(x,v) \,dx\,dv$ which is l.s.c on the space of functions $f \in L^1$ with 
finite kinetic energy.
The fact that $H_Q$ vanishes only at $\mu$ is a consequence of the uniform convexity of $Q$ on the support of $\varphi^{-1}$ (that is also the range of $\mu$).

The point $iii)$ is a simple consequence of a short calculation that we skip. The $L^1$ control is the classical Csisz\'ar-Kullback-Pinsker inequality.

\bigskip
{ \sl Proof of Point $iv)$.}
Under the assumption~\eqref{cond:conv}, the short argument at the beginning of 
the proof of Point $i)$ implies that the bounds~\eqref{ineq:phiQ} hold for each 
$n$, with a constant  $C_n$ that goes to zero as $n \rightarrow + 
\infty$. Remark also that $a_n = \varphi_n(0) \rightarrow + \infty$.

If $\| f\|_\infty < + \infty$, we will have for $n$ large enough $ f(x,v) \le  
a_n$ 
for all $(x,v)$, and this implies that $\int Q_n(f) \le 0$. It is also not 
difficult to see that the bound by below obtained in~\eqref{bound:Qf} goes to 
zero has $n$ goes to infinity.
Thus 
$$
\lim_{n \rightarrow + \infty} \int_{\T \times \R} Q_n(f(x,v)) \,dxdv  = 0, 
\qquad 
\lim_{n \rightarrow + \infty} \int_{\T \times \R} Q_n(\mu_n(x,v)) \,dxdv  = 0,
$$
thanks to~\eqref{egal:Qmu} and the assumption~\eqref{cond:conv}. The term $\int Q_n'(\mu_n)$  vanishes also in the limit because it is exactly the term that appears 
in~\eqref{cond:conv}. The last remaining term is constant and is equal to $\frac{1}{2} \int |v - \bar v|^2 f \, dv dx$, and this concludes the proof.
\end{proof}


\subsection{The well prepared case.}
In this paragraph we prove Theorem \ref{thm}. Recall that 
\begin{equation}
\label{def-L2}
\mathcal{L}_\epsilon(t)= \frac {\ep^2} 2  \int \vert \partial_x V_\epsilon(t) \vert^2 dx + \int \left[Q(f_\epsilon(t)) -  Q(\mu) - Q'(\mu)(f_\epsilon(t)-\mu) \right] dvdx.
\end{equation}

Since our solutions $f_\ep$ are strong ones,
the term $\int Q(f_\ep(t))\,dxdv$ is exactly independent of the time:
$$
\int Q(f_\ep(t))\,dxdv = \int Q(f_{\ep,0}) \,dxdv. 
$$
We introduce the current $j_\ep$, defined as follows:
$$
j_\ep(t,x) := \int f_\ep(t,x,v) v \,dv.
$$
Since the two other constants term  $\int Q(\mu) \,dxdv$ and $\int Q'(\mu) \mu \,dxdv$ are finite, it remains to 
understand how behave
\begin{align*}
\frac \ep 2  \int \vert \partial_x V_\epsilon(t) \vert^2 dx - \int Q'(\mu)f_\epsilon(t) dvdx 
&=
\frac \ep 2  \int \vert \partial_x V_\epsilon(t) \vert^2 dx + \frac12 \int \vert v - \bar v \vert^2 f_\epsilon(t) 
dvdx.  \\
& = 
\EE_\ep[f_\ep] - \bar v  \int  j_{\ep} \,dx + \frac{\vert \bar v \vert^2}2 \int 
\rho_\ep  \,dx,
\end{align*}
where we have used \eqref{eq:Qmu}.
 Since the total mass $\int \rho_\ep \, dx$, and 
the total momentum $\int j_\ep \, dx$, and the total energy $\EE_\ep[f_\ep]$ are preserved by strong solutions of the Vlasov-Poisson 
equation~\eqref{quasi}, we  finally conclude that $\LL_\ep$ is constant.

\subsection{Plasma oscillations.}
\label{sec:heuristics}

The above analysis is only useful in the well-prepared case, that
is when the potential energy vanishes in the limit: $ \| \partial_x V_{0,\ep}
\|_2 = o\bigl(\ep^{-1}\bigr)$. In general, as already evoked in the
introduction, there are time oscillations of the electric field,
called plasma oscillations, that we have to take into account.
In this paragraph, we give a description of these, with the aim to explain the form of the filtered functionals of Theorem \ref{thm-IP}.

For any $\ep>0$, we consider $f_\ep$ a solution to \eqref{quasi}. 
Then the density $\rho_\ep$ and the current $j_\ep$ satisfy the  system 
of equations
\begin{equation}  \label{FirstMoments}
\begin{cases}
\ds \partial_t \rho_\ep + \partial_x j_\ep = 0, \\
\ds \partial_t j_\ep + \partial_x \Bigl( \int f_\ep v^2 \,dv \Bigr) + \partial_x 
V_\ep \rho_\ep = 0.
\end{cases} \end{equation}
Fast oscillations are hidden in that system. Since we work in dimension one, we can always 
write
\begin{equation} \label{divgrad}
j_\ep(t,x) = \bar j_\ep (t)  + \partial_x J_\ep(t,x),
\end{equation}
which is the analogue of the decomposition of a vector-field in potential part 
and divergence free part in higher dimension.  The so-called ``potential of the 
current'' $J_\ep$ is defined up to a constant, 
that we may choose later.
Using the Poisson equation of~\eqref{quasi} in the two lines 
of~\eqref{FirstMoments}, we get
$$ 
\begin{cases}
\ds - \ep^2  \partial_{txx} V_\ep + \partial_{xx} J_\ep = 0, \\
\ds \partial_t \bar j_\ep + \partial_{tx} J_\ep + \partial_x \Bigl( \int f_\ep 
v^2 \,dv \Bigr) + \partial_x V_\ep   - \frac{\ep^2} 2 \partial_x \vert 
\partial_x V_\ep \vert^2 = 0.
\end{cases}
$$
In the second line, the gradient part in $x$ and the constant part can be solved separately thanks to the 
periodicity, and we first get that $\bar j_\ep(t) = \bar j_{0,\ep}$ for all 
times $t$.  Then the 
equation on $\partial_x J_\ep$ and $\partial_x V_\ep$ may be rewritten
\begin{equation} \label{Oscill}
\begin{cases}
\ds   \partial_t  ( \ep \,\partial_ x V_\ep)  &\ds =    \frac{\partial_x J_\ep} 
\ep,  \\
\ds  \partial_t ( \partial_x J_\ep) & \ds =  -  \frac{\ep\, \partial_x  V_\ep} 
\ep   + \partial_x  \Bigl( \frac12  \vert \ep \, \partial_x   V_\ep \vert^2 -   
\int f_\ep v^2 \,dv  \Bigr),
\end{cases}
\end{equation}
which may also rewritten directly on $V_\ep$ and $J_\ep$, which are defined only 
up to some constant. Thanks to that, we can define  a complex quantity
$$
\OO_\ep(t,x) := J_\ep(t,x) + i \, \ep \, V_\ep(t,x),
$$
which satisfies
$$
\partial_t \OO_\ep = \frac i  \ep \OO_\ep +  \vert \partial_x  (\ima \OO_\ep) 
\vert^2 -   \int f_\ep v^2 \,dv. 
$$
Using Duhamel formula, we obtain
\begin{equation} \label{duhamel}
e^{- i \frac t \ep} \OO_\ep(t) =  \OO_{0,\ep} + \int_0^t  e^{ - i \frac s  \ep}
\left(  \vert \partial_x  (\ima \OO_\ep) \vert^2 -   \int f_\ep v^2 \,dv 
\right)\,ds.
\end{equation}
This means that there are large oscillations of period $\ep$ in $\partial_x V_\ep$ and 
$\partial_x J_\ep$, respectively of amplitude or order $\ep^{- 1}$ and $1$. For 
this reason, even if $f_{0,\ep}$ converges to a stable equilibrium $\mu(v)$, we can not expect  $f_\ep$ to converge to $\mu$, without filtrating these 
oscillations. What we should expect is something like $f_\ep(t,x,v- \partial_x 
J_\ep) \approx f(v)$. In particular, $\int f_\ep v \,dv$  and $\int f_\ep(v) v^2\,dv$ should not converge to $\int f v
\,dv$ and $\int f(v) v^2\,dv$, 
but we expect something like
\begin{align*}
\int f_\ep(t,x,v) \,dv  & = \rho_\ep \approx \int \mu(v) \,dv  =1,\\
\int f_\ep(t,x,v) v \,dv  & = j_\ep(t,x) \approx \bar v +  \partial_x J_\ep(t,x),
\\
\int f_\ep(t,x, v) v^2\,dv & \approx  \int \mu(v) \vert v + \partial_x J_\ep(t,x) 
\vert^2 \,dv 
= 2T + \vert  \bar v + \partial_x J_\ep(t,x)  \vert^2,
\end{align*}
with the notation $\bar v := \int \mu(v)v\,dv$ and  $T :=  \frac12\int \mu(v) |v - 
\bar v |^2 \,dv$.  Using this in equation~\eqref{duhamel}, we get, if we forget 
the constants which are not important at the level of potentials, 
\begin{equation} \label{duhamel2}
e^{- i \frac t  \ep} \OO_\ep(t) =  \OO_{0,\ep} + \int_0^t  e^{ - i \frac s \ep}
\left(  \vert \partial_x  (\ima \OO_\ep) \vert^2 -  \vert \partial_x  (\re 
\OO_\ep) + \bar v \vert^2  \right)\,ds.
\end{equation}
The term $\OO_\ep$ displays fast oscillations in time, but we can try to rewrite 
everything in terms of
$$
\UU_\ep (t,x) :=  e^{- i \frac t \ep} \OO_\ep(t,x),
$$
which has a bounded derivative in time, because both the kinetic and potential 
energy are bounded.
We need to write $\vert \ima \partial_x \OO_\ep \vert^2$ and $\vert \re 
\partial_x \OO_\ep\vert^2$ in terms of $\partial_x \UU_\ep$. It comes
\begin{align*}
\vert \ima \partial_x \OO_\ep \vert^2 &=   \frac14 \vert  \partial_x \OO_\ep - 
\overline{\partial_x \OO_\ep}  \vert^2
 =  \frac14 \left\vert  \partial_x [ e^{- i \frac t \ep}\OO_\ep]  - e^{- 2 i 
\frac t \ep} \overline{\partial_x [e^{- i \frac t \ep} \OO_\ep]}  \right\vert^2 
\\
 &= \frac14 \left\vert  \partial_x  \UU_\ep  - e^{- 2 i \frac t \ep} 
\overline{\partial_x \UU_\ep}  \right\vert^2, \\ 
e^{- i \frac s  \ep}  \vert \ima \partial_x \OO_\ep(s) \vert^2 
& =
\frac{e^{- i \frac s  \ep}}2 \vert  \partial_x  \UU_\ep|^2 -
\frac1 4 \partial_x  \UU_\ep \overline{\partial_x  \UU_\ep} \left(e^{ i \frac s  
 \ep} + 
e^{- 3i \frac s  \ep} \right) .
\end{align*}
So since $\UU_\ep$ contains no fast oscillations in time, $e^{- i \frac s  \ep}  
\vert \ima \partial_x \OO_\ep(s) \vert^2$ is a sum of terms with fast 
oscillations in time. So we can expect its contribution in~\eqref{duhamel2} to 
be small  and therefore we neglect it. The same holds for $e^{- i \frac s  \ep}  \vert \re 
\partial_x \OO_\ep(s) \vert^2$. The only non constant term that will contribute 
is in fact the one coming from
$e^{- i \frac s  \ep}  \bar v  \re \partial_x \OO_\ep(s) $, since
\begin{align*}
 \re \partial_x \OO_\ep(s)   &=  \frac12 [ \partial_x \OO_\ep(s)  + 
\overline{\partial_x \OO_\ep(s)}]
 = \frac12 \left[  e^{i \frac s  \ep}\partial_x \UU_\ep(s)  + e^{- i \frac s  
\ep} \overline{\partial_x \UU_\ep(s)} \right], \\
2 \bar v e^{- i \frac s  \ep}  \re \partial_x \OO_\ep(s)  &= 
\bar v   \, \partial_x \UU_\ep(s)  +  \bar v e^{- 2 i \frac s  \ep} 
\overline{\partial_x \UU_\ep(s)}.\end{align*}
Only the first term in the r.h.s. of the last equation will have a significant 
contribution, and if all the approximations are justified, we will end up with a function
$\UU_\ep$ satisfying approximately the equation
$$
\UU_\ep(t) = \UU_{0,\ep} - \bar v  \int_0^t \partial_x \UU_\ep(s) \,ds, 
$$
or equivalently the linear transport equation
$$
\partial_t \UU_\ep + \bar v \, \partial_x \UU_\ep = 0,
$$
with the initial condition $\UU_{0,\ep} = J_{0,\ep} + i  \, \ep V_{0,\ep}$. 

If $f_{0,\ep}$ converges towards $\mu(v)$, then the ``potential'' part 
$J_{0,\ep}$ of the current converges towards $0$. If moreover the 
potential $\ep \, \pa_x V_{0,\ep}$ has a limit denoted 
$\pa_x V_0$ as $\ep$ goes to zero, then it is natural to define $\UU$ as the solution of the 
simple linear transport equation
\begin{equation} \label{oscilimit}
\begin{cases}
\partial_t \UU + \bar v \, \partial_x \UU = 0, \\
\UU(0)  =  i V_0
\end{cases}.
\end{equation}
The solution of~\eqref{oscilimit} is simply given by $\UU(t,x) := i \,V_0(x-\bar
v t)$.
Therefore, we can use the approximation 
\begin{equation} \label{filtre}
\OO_\ep(t) \approx  i\, e^{ i \frac t  \ep} V_0(x-\bar vt)
= V_0(x-\bar vt) \left[ - \sin \frac t  \ep  + i \, \cos \frac t 
\ep \right],
\end{equation}
in order to filtrate the plasma oscillations for small $\ep$, hence the 
expressions of the functionals in Theorem \ref{thm-IP}.

\subsection{The general case.}
\label{sec:IP}
We now prove Theorem \ref{thm-IP}. By Galilean invariance of the Vlasov equation \eqref{quasi}, we can restrict ourselves to the case $\bar v= 0$. 
Indeed, if $\bar  v \neq 0$, we can rewrite the problem in the variable $(x',v')$ defined by
$$
v = \bar v  + v', \qquad  x = x' + \bar v t.
$$
This will simplify the calculations.

Next with the help of some simple changes of variables, remark that we may 
rewrite
\begin{multline} \label{newcas}
H_Q\left[ f_\ep\left(t,x,v - \partial_x V_0 \sin \frac t \ep \right) \right] = 
\int Q(f_\ep) \,dx\,dv +  \\
\frac12 \int \Bigl\vert v  +  \partial_x V_0\sin \frac t \ep \Bigr\vert^2 
 f_\epsilon(t,x,v)    \,dx\,dv
- \int Q( \mu)\,dv - \int \frac{v^2}2 \mu(v) \,dv 
\end{multline}
Since, we are dealing with strong solution, the first term in the r.h.s. 
of~\eqref{newcas} is constant. In view of Proposition~\ref{prop:HQ}, the 
last two are also finite constants.

We may also develop the kinetic energy term, which yields
\begin{multline} \label{LO:Ekin}
\frac12 
\int \Bigl\vert v  +  \partial_x V_0\sin \frac t \ep \Bigr\vert^2 
 f_\epsilon(t,x,v)    \,dxdv  
 = \: \frac12 \int v^2 f_\ep(t) \,dx\,dv \\
+ \frac12 \sin^2 \frac t \ep \int \vert\partial_x V_0 \vert^2 \rho_\ep(t,x)\,dx  
+ \sin \frac t\ep \int \partial_x V_0  \, j_\ep(t,x) \,dx.
\end{multline}

Finally, the term with the electric field in $\LL_\ep^O$ leads to 
\begin{multline}
 \frac 1 2 \int \Bigl|  \ep \partial_x V_\epsilon - \partial_x V_0\cos \frac 
t\ep \Bigr|^2 dx 
 = \frac12 \int \vert \ep \partial_x V_\ep \vert^2 \,dx \\
+ \frac12 \cos^2 \frac t\ep \int \vert \partial_x V_0(x) \vert^2 \,dx
- \cos \frac t\ep \int \partial_x V_0 \,[\ep  \partial_x V_\ep]
 \,dx. \label{LO:Epot}
\end{multline}

Summing up the first terms in the r.h.s in~\eqref{LO:Ekin} and~\eqref{LO:Epot}, 
we
get the total energy $\EE_\ep[f_\ep(t)]$, which is preserved by the dynamics. Moreover, up to a constant, we can replace the
$\cos^2$ appearing in~\eqref{LO:Epot} by a $-\sin^2$. 
Finally, summing up~\eqref{LO:Ekin} and~\eqref{LO:Epot}, we get that
$$
\LL_\ep^O(t) = \KK_\ep^O(t) + \text{a constant term}
$$
where
\begin{equation} \label{defKep}
\KK_\ep^O(t) := \frac12  \sin^2 \frac t\ep \int \vert
\partial_x V_0 \vert^2 ( \rho_\ep -1)\,dx +  \int  
\partial_x V_0  \Bigl[   j_\ep \sin \frac t\ep -   \ep  \partial_x
V_\ep \cos \frac t\ep \Bigr] \,dx. 
\end{equation}
Using the Poisson equation in~\eqref{quasi}, the decomposition of 
$j_\ep$ in its constant and potential part $j_\ep = \bar j_\ep + \partial_x
J_\ep$, and introducing 
\begin{equation}
\label{K}
\partial_x R_\ep :=  \partial_x J_\ep  \sin \frac t \ep  - 
 \ep  \partial_x V_\ep \cos \frac t \ep,
 \end{equation}
we can express $\KK_\ep^O$ in terms of $\partial_x R_\ep$ only:
\begin{equation} \label{Kep2}
\KK_\ep^O(t) = - \frac{1} 2  \sin^2 \frac t \ep  \int \vert
\partial_x V_0 \vert^2  \, \ep^2  \partial_{xx} V_\ep \,dx +  \int  
\partial_x V_0   \, \partial_x R_\ep \,dx . 
\end{equation}
After an integration by parts, we can  rewrite it as
\begin{align*} 
\KK_\ep^O(t) \: &=  \:    \ep\,  \sin^2 \frac t \ep   \int 
\partial_{xx} V_0  \partial_x V_0   
  \,  \ep  \partial_x V_\ep \,dx   +  \int  
\partial_x V_0  \,   \partial_x R_\ep  \,dx  \\
&=:  \hspace{50pt} \II_\ep^{O,1}(t) \hspace{59pt}  +  \hspace{28pt} 
\II_\ep^{O,2}(t).
\end{align*}
 It can be shown the first term $\II_\ep^{O,1}$ is of order 
$\ep$, but we shall not use this fact, because we need this term in 
order to compensate for some bad terms coming form the time derivative of 
$\II_\ep^{O,2}(t)$. 
The interest of $\partial_x R_\ep$ is that thanks to~\eqref{Oscill}, we have 
\begin{equation}
\label{derivK}
\partial_t [\partial_x R_\ep ] = \partial_x \Bigl( \frac12 \vert
  \ep  \partial_x V_\ep \vert^2 - \int f_\ep v^2 \,dv \Bigr) \sin \frac
t \ep ,
\end{equation}
so that the possible oscillations with amplitude $\ep^{-1}$ in $\partial_t
\KK_\ep^O$ will vanish. Using once again \eqref{Oscill}, it comes
\begin{align}
 \frac d{dt} \II_\ep^{O,1}(t) & =  
2 \sin \frac t \ep   \cos \frac t \ep   \int \pa_{xx} V_0 \pa_x V_0 \, \eps 
\pa_x V_\eps \, dx  
+  \sin^2 \frac t \ep    \int \pa_{xx} V_0 \pa_x V_0    \pa_x J_\eps \, dx 
 , \label{deriv:I1}\\
 \frac d{dt} \II_\ep^{O,2}(t)
&= \sin \frac t \ep  \int  \partial_x V_0    \partial_x \Bigl( \frac12 \vert
  \ep  \partial_x V_\ep \vert^2 - \int f_\ep v^2 \,dv \Bigr)  \, dx \nonumber\\
&=  - \frac{1}{2}  \sin \frac{t} \ep  \int \pa_{xx} V_0 \,|\ep \pa_x V_\ep|^2  
\, dx
+ \sin \frac{t} \ep  \int f_\ep v^2 \pa_{xx} V_0  \, dv dx . \label{deriv:I2}
\end{align}
We shall now write the two decompositions
\begin{align}
- \frac{1}{2} \sin \frac{t} \ep  \int \pa_{xx} V_0  |\ep \pa_x V_\ep|^2  \, dx
& = -\frac{1}{2} \sin \frac{t} \ep  \int \pa_{xx} V_0 \left|\ep \pa_x V_\ep- 
\pa_x V_0 \cos \frac{t} \ep \right|^2 \, dx \nonumber\\
& \hspace{-70pt} - \sin \frac{t} \ep  \cos \frac{t} \ep \int \pa_{xx} V_0 \pa_x 
V_0 \, \ep \pa_x V_\ep    \, dx 
+\frac{1}{2} \cos^2 \frac{t} \ep  \sin \frac{t} \ep \int \pa_{xx} V_0  |\pa_x 
V_0|^2   \, dx, \label{decomp1} \\
\sin \frac{t} \ep  \int f_\ep v^2 \pa_{xx} V_0  \, dv dx 
& = \sin \frac{t} \ep  \int f_\ep \left|v  + \pa_x V_0 \sin \frac{t} \ep  
\right|^2 \pa_{xx} V_0 \, dv dx  \nonumber \\
& \hspace{-70pt} - \sin^3 \frac{t} \ep \int  \rho_\ep |\pa_x V_0|^2  \pa_{xx} 
V_0   \,  dx - 2 \sin^2 \frac{t} \ep  \int \pa_{xx} V_0 \pa_x V_0  \, j_\eps 
\,  dx. \label{decomp2}
\end{align}
Note that in the first decomposition, the last term of the r.h.s. is actually 
equal to $0$ by integration by parts. 

The interest of these two decompositions is that it introduces terms which are
very similar to those of $\LL^0_\ep(t)$ : up to the multiplicative factor $\pa_{xx} V_0$, 
the first term in the r.h.s. of~\eqref{decomp1} is the relative potential 
energy term 
of $\LL^0_\ep(t)$, and the first term in the r.h.s. 
of~\eqref{decomp2} appears in the relative entropy
part of $\LL^0_\ep(t)$ (See~\eqref{newcas}). The idea is then to let 
appear the missing terms.
To this end, remark that since $\mu$ does not depend on $x$, by integration 
by parts, we have
\begin{equation*}
\int Q(\mu) \,  \pa_{xx} V_0  \, dv dx = 0, \quad
\int Q'(\mu) \mu \,  \pa_{xx} V_0 \, dv dx = 0.
\end{equation*}
Thus, using~\eqref{deriv:I1}--\eqref{decomp2}, we finally 
end up with
\begin{multline}
\label{calcul-LL}
  \frac d{dt} \LL_\ep^{O}(t) =
  -  \sin\frac{t}{\eps} \int \Bigg\{\int2\big[Q\left( \tilde f_\ep \right)- 
Q(\mu)
- Q'(\mu)(\tilde f_\ep -\mu)\big] \, dv  + \frac{1}{2}  \Bigl[  \ep 
\partial_x V_\epsilon -
\partial_x V_0
\cos \frac t \ep \Bigr]^2 \Bigg\} \\
\times \pa_{xx} V_0 \, dx +  \sin \frac t \ep \, B_\eps^1(t) + \sin \frac t \ep
B_\eps^2(t) + r_\eps(t),
 \end{multline} 
where $\tilde f_\ep(t,x,v) := f_\ep\left(t,x,v - \partial_x V_0 \sin \frac t \ep 
\right)$ and with:
\begin{align*}
 B_\eps^1(t)\: &:=  \:    \int \pa_{xx} V_0 \, \pa_x V_0   \Bigl( \cos \frac t 
\ep 
\, \ep \pa_x  V_\eps  - \sin \frac t \ep \, \pa_x J_\eps \Bigr) \, dx  
 = \int  \pa_{xx} V_0 \, \pa_x V_0  \,  \pa_x R_\eps \, dx, \\
B_\eps^2(t)\: &:= \: 2  \int Q(f_\ep) \,  \pa_{xx} V_0   \, dv dx,
\\
r_\ep(t)\: &:= \: - \sin^3 \frac{t} \ep \int  \rho_\ep |\pa_x V_0|^2  \pa_{xx}
V_0   \,  dx. 
\end{align*}
In what follows, we will show that $r_\ep$ is of order $\ep$, and that 
$B^1_\ep$ and $B^2_\ep$ are bounded with bounded time derivatives, so that the 
contributions of $\sin \frac t \ep \, B_\eps^1$ and  $\sin \frac t \ep  
B_\eps^2$ will be of order $\ep$ after integration in time.

Integrating~\eqref{calcul-LL} in time and using the convexity of $Q$, we obtain 
the bound
\begin{equation}
\label{eq:gronwall}
\LL_\ep^{O}(t) \leq  \LL_\ep^{O}(0) 
+ 2  \, \| \pa_{xx} V_0 \|_{L^\infty} \int_0^t \LL_\ep^{O}(s) _,ds  
+ \int_0^t \Bigl( \sin \frac s \ep \, B_\eps^1(s) 
+ \sin \frac s \ep \, B_\eps^2(s)+ r_\eps(s)  \Bigr)\,ds.
\end{equation}

\bigskip
\noindent $\bullet$ \emph{Treatment of the time integral with $B_\eps^i$ for
$i=1,2$.}
Since $B^1_\ep$ and $B^2_\ep$ are not small, but have bounded derivatives, their
respective contribution in~\eqref{eq:gronwall} may be controlled  with the help of
an integration by parts. Precisely, we have
\begin{align*}
\int_0^t  \sin \frac s \ep \, 
B_\eps^i(s) \,ds 
& =
\int_0^t \left( \ep  \cos \frac s \ep  \right)' B_\eps^i(s)
\,ds  \\
& = -  \ep \int_0^t 
\cos \frac s \ep   {B_\eps^i}'(s)  \,ds +  \left[ 
\cos \frac s \ep    B_\eps^i(s) \right]_0^t,
\end{align*}
and this leads to
\begin{equation}
 \left| \int_0^t  \sin \frac s
\ep \, B_\eps^i(s) \,ds  \right|  \le
\ep \,  \| {B^i_\ep}'\|_\infty \, t  + 2 \ep \,  \| B_\eps^i \|_\infty
\label{contrib:Ri}
\end{equation}
So, it only remains to get uniform (in time) bounds on $B^i_\eps$ and its
derivative for $i=1,2$.

\bigskip
\noindent $\bullet$ \emph{Uniform bounds on $B_\eps^1$ and ${B_\ep^1}'$.} 

According to the definition of $B^1_\eps$ and~\eqref{K},
we have the following bound
\begin{align*}
 | B_\eps^1(t) | &  \le \| \pa_{xx} V_0 \|_\infty \, 
\| \pa_x V_0 \|_\infty \, \| \pa_x R_\ep \|_{L^1} \\
& \le \| \pa_{xx} V_0 \|_\infty^2 
\bigl( \| \ep \pa_x V_\ep \|_{L^1} + \| \pa_x J_\ep \|_{L^1} \bigr) \\
& \le  \| \pa_{xx} V_0 \|_\infty^2 
\bigl( \| \ep \pa_x V_\ep \|_{L^2} + 2 \, \| j_\ep \|_{L^1} \bigr) \\
& \le  \| \pa_{xx} V_0 \|_\infty^2 
\bigl( 1+ \| \ep \pa_x V_\ep \|_{L^2}^2 + \| f_\ep \|_{L^1} + \| f_\ep
v^2 \|_{L^1} \bigr) \\
& \le 2 \, \| \pa_{xx} V_0 \|_\infty^2 
\bigl( 1 + \EE_{\ep,0} \bigr),
\end{align*}
where we used the notation $\EE_{\ep,0}:= \EE_\ep[f_{0,\ep}]$, the fact that $\partial_x J_\ep = j_\ep - \bar j_\ep$ 
($\bar j_\ep$ being the average of $j_\ep$), and a simple interpolation.
Remember also that $\| f_\ep \|_{L^1}=1$ by assumption on the initial datum.

Now, according to~\eqref{derivK}, we have
\begin{align*}
 | \partial_t B_\eps^1(t) |  &:=
\Bigl| \int  \pa_{xx} V_0 \, \pa_x 
V_0  \, \pa_t \pa_x R_\ep \,dx \Bigr| \\
 & =  \Bigl| \sin \frac
t \ep  \int  \pa_x\left(\pa_{xx} V_0 \pa_x V_0\right)  \Bigl( \frac12
\vert  \ep  \partial_x V_\ep \vert^2 - \int f_\ep v^2 \,dv \Bigr)   \, dx 
\Bigr|  \\
& \le 4 \, \| \partial_{xxx} V_0 \|_{L^\infty} \, \|
\partial_{xx} V_0 \|_{L^\infty} \, \EE_{\ep,0}.
\end{align*}
Plugging it into~\eqref{contrib:Ri}, we get the bound
\begin{equation} \label{contrib:R1}
 \left| \int_0^t \sin \frac s
\ep \, R_\eps^1(s) \,ds  \right| 
 \le 
4 \ep \,  \| \pa_{xx} V_0 \|_\infty \bigl( 1 + \EE_{\ep,0} \bigr)  \Bigl( 
\| \partial_{xxx} V_0
\|_{L^\infty} \, t  +  \| \pa_{xx} V_0 \|_\infty   \Bigr).
\end{equation}

\bigskip

\noindent $\bullet$ \emph{Uniform bounds on $B_\eps^2$ and ${B_\ep^2}'$.} 
Since our solutions are strong, we simply bound using the
notations introduced in~\eqref{eq-crucial}
\begin{align*}
|B_\ep^2(t)| & \le  2 \, \| \pa_{xx} V_0 \|_\infty \int |Q(f_\ep)(t)| \,dx\,dv
\\
& = 2  \, \| \pa_{xx} V_0 \|_\infty \int |Q(f_{\ep,0})(t)| \,dx\,dv \le
2  \, \| \pa_{xx} V_0 \|_\infty \QQ_{\ep,0}.
\end{align*}

To bound the time derivative of $B^2_\ep$, we rely on the fact that $Q(f_\ep)$
is also a strong solution of the Vlasov equation~\eqref{quasi}: 
$$
\pa_t Q(f_\ep) + v \, \pa_x \bigl[Q(f_\ep)\bigr] - \pa_x V_\ep \pa_v
\bigl[Q(f_\ep)\bigr] =0.
$$
We thus get
\begin{align*}
\int \pa_t \bigl[Q(f_\ep)\bigr] \,  \pa_{xx} V_0 \, dv \,dx  
&=  \int \pa_{xxx} V_0 \, v \, Q(f_\ep) \, dv dx   -  \int
  Q(f_\ep)  \, \pa_v \bigl[ \pa_x V_\ep\,   \pa_{xx} V_0  \bigr] \, dv \,dx\\
&=  \int \pa_{xxx} V_0 \, v \, \sqrt{f_\ep}\, \frac{Q(f_\ep)} {\sqrt{f_\ep}} \,
dv dx.
\end{align*}
Note that all the underlying calculations are well justified since
$Q(f_\ep)$ is assumed to be in $L^1_{x,v}$ and $\pa_{xx} V_0$ and $\pa V_\ep$
 are bounded (non uniformly in $\ep$). Then
\begin{align*}
\bigl| \pa_t B^2_\ep \bigr|& = 
\left|\int \pa_t Q(f_\ep) \,  \pa_{xx} V_0  \, dv dx  \right| \\
& \leq \frac 1 2 \| \pa_{xxx} V_0 \|_\infty \left(\int  v^2 {f_\ep} \, dv dx +
\int \frac{Q^2(f_\ep)}{f_\ep} \, dv dx\right) \\
& \le  \| \pa_{xxx} V_0 \|_\infty \Bigl( \EE_{\ep,0} + \frac12 
\QQ_{\ep,0} \Bigr).
\end{align*}
Plugging all into~\eqref{contrib:Ri}, we get the bound
\begin{equation} \label{contrib:R2}
 \left| \int_0^t  \sin \frac s \ep \, R_\eps^1(s) \,ds  \right| 
 \le 
\ep \Bigl[ \| \pa_{xxx} V_0 \|_\infty 
\bigl( \EE_{\ep,0} + \QQ_{\ep,0} \bigr) \, t 
+ 4  \| \pa_{xx} V_0 \|_\infty \QQ_{\ep,0} \Bigr].
\end{equation}

\bigskip
\noindent $\bullet$\emph{Treatment of $r_\eps$.} The last term of the
remainder is the easiest to analyze. We use the Poisson equation in
\eqref{quasi} to write
\begin{align*}
r_\ep(t) &= \: - \sin^3 \frac{t}{\ep} \int |\pa_x V_0|^2  \pa_{xx} V_0   \,  dx
+ \sin^3 \frac{t}{\ep} \int \eps^2 \pa_{xx}^2 V_\ep |\pa_x V_0|^2  \pa_{xx} V_0
 \,  dx \\
&=  \:- \ep \sin^3 \frac{t}{\ep} \int \eps \pa_{x} V_\ep \,   \pa_x
\bigl(|\pa_x V_0|^2  \pa_{xx} V_0 \bigr)   \,  dx
\end{align*}
By Cauchy-Schwarz inequality and  by conservation of the energy, we deduce that
$$
|r_\ep(t)| \leq  3 \, \ep  \,
\| \pa_{xx} V_0 \|_\infty^2 \,
\| \pa_{xxx} V_0 \|_\infty
\bigl( 1 + \EE_{\ep,0} \bigr),
$$
and a time integration leads to
\begin{equation} \label{contrib:R3}
\left| \int_0^t   r_\eps(s) \,ds  \right| 
 \le 
3 \, \ep  \,
\| \pa_{xx} V_0 \|_\infty^2 \,
\| \pa_{xxx} V_0 \|_\infty 
\bigl( 1 + \EE_{\ep,0} \bigr) \, t
\end{equation}

\bigskip
\noindent $\bullet$ \emph{Conclusion.}
For simplicity, we will denote $\kappa:= 2\,\| \pa_{xx} V_0 \|_{L^\infty}$.
Using now~\eqref{eq:gronwall} and gathering the
contributions~\eqref{contrib:R1}--\eqref{contrib:R3} together, we have proved
that
\begin{equation}
\LL_\ep^{O}(t) \leq \LL_\ep^{O}(0) 
+ \kappa \int_0^t \LL_\ep^{O}(s) _,ds  
+ \ep  (a+ bt),
\end{equation}
 
where
$$
 a   := C \,  \kappa \bigl[ \kappa  (1 + \EE_{\ep,0}) + \bar{\mathcal
Q}\bigr]
\quad \text{and} \quad 
 b   :=  C \, \| \pa_{xxx} V_0 \|_\infty \bigl[ (1+\kappa^2) (1 +
\EE_{\ep,0}) + \QQ_{\ep,0}
\bigr],
$$
for some numerical constants $C$. An application of Gronwall lemma leads to the
inequality
$$
\LL_\ep^{O}(t) \leq e^{\kappa t} \Bigl[ 
\LL_\ep^{O}(0) + \ep \,  a + \ep \,  \frac b \kappa
\Bigr],
$$
and this concludes the proof of Theorem~\ref{thm-IP}, with precisely
\begin{equation} \label{def:K}
K : = C \bigl( 1+  \| \pa_{xx} V_0 \|_\infty^2 \bigr) 
\Bigl( 1 + \frac{\| \pa_{xxx} V_0 \|_\infty}{\| \pa_{xx} V_0 \|_\infty} \Bigr).
\end{equation}

\bigskip
\noindent $\bullet$ \emph{The interpolation argument for
Corollary~\ref{rem:lowreg}.}
If $V_0$ has no bounded third derivative, we introduce some smoothing of $V_0$
$$
V_{0,\eta} = V_0 \ast k_\eta, \qquad
k_\eta := \frac1 \eta k \Bigl( \frac \cdot \eta \Bigr),
$$
where $k$ is some nonnegative smooth function with compact support and $\int k
=1$. The mollified potential $V_{0,\eta}$ satisfies the following properties:
\begin{equation*}
\| \pa_{xx} V_{0,\eta} \|_\infty \le  \| \pa_{xx} V_0 \|_\infty, 
\quad
\| \pa_{xxx} V_{0,\eta} \|_\infty \le \frac C \eta \, \| \pa_{xx} V_0 \|_\infty,
\quad 
\bigl\| \pa_x (V_{0,\eta}  - V_0) \bigr\|_2 \le C \,\eta \, \| \pa_{xx} V_0
\|_\infty
\end{equation*}
for some constants $C$ depending only on $k$.
It implies that the constant $K_\eta$ given by~\eqref{def:K}
applied to $V_{0,\eta}$ may be written $K_\eta = K' \bigl( 1 + \frac1 \eta
\bigr)$, for some constant $K'$ depending only on $\| \pa_{xx} V_0 \|_\infty$.
Therefore if we apply the Theorem~\ref{thm-IP} for the potential $V_{0,\eta}$,
we get the bound
$$
\LL_{\ep,\eta}^{O}(t) \leq e^{2 \| \pa_{xx} V_0 
\|_{L^\infty} t}  \Bigl[ \LL_{\ep,\eta}^{O}(0)+  K'  \ep \Bigl( 1 +
\frac 1 \eta \Bigr) \bigl( 1 +
\EE_{\ep,0} + \QQ_{\ep,0} \bigr) \Bigr].
$$
But we also have
\begin{align*}
\LL_{\ep,\eta}^{O}(t)  & = H_Q(\tilde f_\ep) + 
\frac12  \ \Bigl\|   \Bigl(\ep \, \pa_x V_\ep - \pa_x V_0 \cos
\frac t \ep \Bigr) + \cos \frac t \ep \Bigl( \pa_x V_0  
- \pa_x V_{0,\eta} \Bigr)
\Bigr\|_2^2,  \\
& \le 
H_Q(\tilde f_\ep) + 
 \ \Bigl\| \ep \, \pa_x V_\ep - \pa_x V_0 \cos
\frac t \ep \Bigr\|_2^2  +  \Bigl\| \pa_x V_0  
- \pa_x V_{0,\eta} \Bigr\|_2^2, \\
& \le 2 \,  \LL_\ep^{O}(t) + C \, \eta^2,
\end{align*}
thanks to the bound satisfied by $V_{0,\eta}$. Similarly, we can also prove that
$ \ds
\LL_\ep^{O}(t)  \le 2  \,\LL_{\ep,\eta}^{O}(t) + C \, \eta^2.$
All in all, we get that
\begin{align*}
\LL_\ep^{O}(t) & \le 2 \, \LL_{\ep,\eta}^{O}(t) + C \, \eta^2  \\
& \le 2 \, e^{2 \| \pa_{xx} V_0 
\|_{L^\infty} t}  \Bigl[ \LL_{\ep,\eta}^{O}(0)+  K'  \ep \Bigl( 1 +
\frac 1 \eta \Bigr) \bigl( 1 +
\EE_{\ep,0} + \QQ_{\ep,0} \bigr) \Bigr]  + C \, \eta^2 \\
&  \le 4\,  e^{2 \| \pa_{xx} V_0 
\|_{L^\infty} t}  \Bigl[ \LL_\ep^{O}(0)+  K'  \ep \Bigl( 1 +
\frac 1 \eta \Bigr) \bigl( 1 +
\EE_{\ep,0} + \QQ_{\ep,0} \bigr)  + C \, \eta^2 \Bigr],
\end{align*}
where we recall that the value of $C$ may change form line to line. 
The claimed result follows from the choice 
$ \eta = \ep^{\frac13} \bigl( 1 +
\EE_{\ep,0} + \QQ_{\ep,0} \bigr)^{\frac13}$.
  
\subsection{Some analogies.}

We conclude this section with a short digression about the analogies between the
quasineutral limit and two classical singular limits in hydrodynamics.

The first one is the so-called hydrostatic approximation of the Euler equation.
This limit turns out to be false in general due to the existence
of instabilities for the unscaled system (see for instance \cite{Br99}). As for
the quasineutral limit, it is important to consider data satisfying a
stability condition, namely the Rayleigh condition (a kind of monotonicity
condition; we refer to \cite{Br99} for details). There are similarities between
our approach and the proof of derivation that Brenier gave in \cite{Br03}. His proof
relies also on some modulated energy method. As for the quasineutral limit, the
first proof of this result was due to Grenier \cite{Gr99-2}; the techniques he
used are related to those suggested in \cite{Gr99-1} (see also Masmoudi-Wong
\cite{MW}).

There are also analogies with the derivation of the Prandtl equation in the
inviscid limit of the Navier-Stokes equations and with its ill-posedness properties; we refer to
\cite{OS,SamCaf,Gr,GVD,GVN,GN,MW2,GVM}. It would be very interesting to further investigate
these.

\section{Locally symmetric solutions to~\eqref{kineuler-intro} 
are homogeneous.}
\label{sec-EQ}

In this section, we prove the following 

\begin{prop}
\label{prop-EQ}
 Let $f$ be a weak solution to the quasineutral equation~\eqref{kineuler-intro}
satisfying the following hypothesess. The electric field $E = - \partial_x V$ belongs to
${L^\infty_t L^1_x}$ and there exists a $C^1$
function $ \bar v(t,x)$  such that  
\begin{itemize}
 \item[i)] for all $t,x$, $v \mapsto f(t,x,v)$ is increasing for $v < \bar 
v(t,x)$
and decreasing for $v > \bar v(t,x)$, and so has a maximum at $v=\bar 
v(t,x)$,
  \item[ii)] for all $t,x,v$, $f(t,x,2 \bar v(t,x) - v) = f(t,x,v)$.
\end{itemize}
Then there exist a constant (in time and
position) $\bar v$ and a profile $\varphi : \R^- \rightarrow \R^+$,
 nondecreasing and satisfying $\int_{\R^+} \varphi(u) \frac{du}{\sqrt{u}}= 1$,
such 
that for all $t\geq 0, x \in \T, v \in \R$,
 \begin{equation}
 \label{goal-EQ}
f(t,x,v)= \varphi\left(-\frac{|v - \bar v|^2}{2}\right).
\end{equation}
\end{prop}

Similar ``rigidity'' properties (in a different context) were also studied by 
Ben Abdallah and Dolbeault \cite[Section
2.3]{BAD03}. 

\begin{proof}[Proof of Proposition~\ref{prop-EQ}]

During this proof we shall use the following notation : when $f$ (or one of its
derivatives) stands without reference to the variables, this
means $f= f(t,x,v)$. Otherwise the variables are explicitly written. By
instance, we will write the point $ii)$ as $f(t,x,2 \bar v(t,x) - v) = f$. For 
the function $\bar v$ there is no possible ambiguity since the variables will 
always  be $(t,x)$.

\bigskip

{\sl  Step 1. Some remarkable identities.}\\
We start form the equality $f(t,x,2 \bar v - v) = f$, and differentiate
it in $v$,  $t$ and $x$ in order to get the following identities:
\begin{equation*}
\left\{
\begin{aligned}
 &\partial_v f(t,x,2 \bar v - v)  =  - \partial_v f, \\
 &\partial_t f(t,x,2 \bar v - v)  =  \pa_t f  - 2\, \pa_v f(t,x,2 \bar v 
- v) \, \partial_t \bar v =  \partial_t f  + 2\, \pa_v f \, \pa_t
\bar v, \\
 &\partial_x f(t,x,2 \bar v - v)  =  \partial_x f  - 2\,  \pa_v f(t,x,2 \bar v 
- v) \, \partial_x \bar v =  \partial_x f  + 2 \, \partial_v f \, \pa_x
\bar v.
\end{aligned}
\right.
\end{equation*}
Using this in the equation~\eqref{kineuler-intro} written at the point 
$(t,x,2\bar v -v)$, that is
$$ \partial_t f (t,x,2\bar v -v) + (2 \bar v -v) \, \partial_x f(t,x,2\bar v 
-v) +
E \, \partial_v f(t,x,2\bar v -v) = 0,$$
we get
$$
\partial_t f + (2\bar v -v)  \partial_x f +
\left[ 2 \, \partial_t \bar v + 2(2 \bar v -v)  \partial_x
\bar v - E \right]\partial_v f = 0 \,.
$$
Subtracting this equation to the original Vlasov equation \eqref{kineuler-intro} 
 (and
dividing by $2$) we get
$$
(v - \bar v) \partial_x f + \left[E - \partial_t \bar v - (2  \bar v - v)
\partial_x \bar v  \right]\partial_v f = 0 \,.
$$
Since the time does not appear in the equation, we may work for a fixed $t$. 

\bigskip
{\sl  Step 2. An ODE system at frozen $t$.}\\
The previous equation means that at the frozen time $t$, the distribution $f(t,
\cdot,
\cdot)$ is constant along the trajectories of the system of ODEs
$$ \begin{cases}
\dds X = \Xi - \bar v(t,X), \\
\dds \Xi = E(t,X) - \partial_t \bar v(t,X) - (2\bar v(t,X) - \Xi) \partial_x 
\bar
v(t,X).
\end{cases}$$ 
Using the new variables $X,W:= \Xi - \bar v(t,X)$, we get the
system
$$ \begin{cases}
\dds X = W, \\
\dds W = E(t,X) - \partial_t \bar v(t,X) - \bar v(t,X) \partial_x \bar v(t,X)
=: \tilde E(t,X).
\end{cases} $$ 
We remark that $\bar v$ appears with the directional derivative $D_t =
\partial_t + \bar v \partial_x$.

We now prove that $\tilde E$ is the gradient (with respect to $x$) of some 
potential $\tilde V$. Since we work in $1D$, it is sufficient to prove that the 
average of $\tilde E$ on $\T$ is equal to $0$.
Since $E=-\partial_x V$, we already know that $\int_{\T} E \, dx =0$. We also 
have
$$ 
\int_{\T} \bar v(t,x) \partial_x \bar v(t,x) \,dx  = \int_{\T} 
\partial_x \left( \frac{\bar v(t,x)^2}2 \right) \,dx = 0  \,.
$$
For the $\partial_t \bar v $ term, remark that from the  symmetry assumption
$ii)$ and $\rho = 1$, we have  
$$\int_{\T} \bar v (t,x) \,dx = \int_{\T} v f \,dxdv =
P,$$
where $P$ is the total momentum which is preserved by the equation 
\eqref{kineuler-intro}. Therefore, the integral $\int_{\T} \partial_t \bar
v(t,x) \,dx = \frac{dP}{dt}$ also vanishes. 

We remark that according to~\cite{Hau2D} and since $ \tilde 
E= - \pa_x \tilde V \in L^1$, the measure preserving flow
associated to that system of ODEs is uniquely defined,  and allows to 
construct the solutions of the associated transport equation. 

Consequently, in the $(X,W)$ coordinates, the trajectories are the curves 
$\tilde V(X) +
\frac{W^2}2 = Cst$. If $\bar v$ is $C^1$ and {$E \in L^1_x$}, then the 
potential $\tilde{V}$
is continuous. Since it is defined up to a constant, we assume that its minimum
value is $0$ and denote by $x_0$ a point where $\tilde V (x_0) = 0$. Using the  
assumptions $i)$ and $ii)$, we can define a function $g: \R^+ \to \R^+$, which 
is nonincreasing
on $\R^+$ and such that for all $w \in \R$,
$$
g\left(\frac{w^2}2\right) = f(t,x_0,\bar v(t,x_0) + w).
$$
Let $(x,w) \in \T \times \R$. Since the integral curve of $(x,w)$ crosses
the line $\{x = x_0 \}$, precisely at the point $(x_0,w_0)$, where 
$\frac{w_0^2}{2}= \tilde{V}(x) +\frac{w^2}{2}$,   we get that 
\begin{equation} \label{eq:cons0}
f(t,x,\bar v + w) = g\left(\frac{w^2}2 + \tilde V(x)
\right).
\end{equation}

{\sl  Step 3. Using the $\rho=1$ constraint.}

Now, by \eqref{kineuler-intro},  since for all $x$, we have $\rho(t,x)=1$, we 
should have
$$\int  f(t,x,v)  \,dv = 2
\int_0^{+\infty} g\left(\frac{w^2}2 + \tilde V(x) \right) \,dw =1.
$$
This implies that $\tilde V(x)=0$ for all $x$. Indeed,  assume by contradiction that there exist $x$ 
 such that $ \tilde V(x) > \tilde V(x_0)=0$. Then, the following holds
$$
\int_0^{+\infty} \left[g\left(\frac{w^2}2 \right)  
-g\left(\frac{w^2}2 + \tilde V(x) \right)  \right] \,dw = 0.
$$
Since $g(\cdot)$ is nonincreasing, we deduce that for all $w \in \R^+$,
$$
g\left(\frac{w^2}2 \right)  -g\left(\frac{w^2}2 + \tilde 
V(x) \right)=0
$$
But since $g$ is integrable with the weight $\frac{1}{\sqrt{u}}$, $g$ cannot 
be constant and there is $z \in \R^+$ such that 
$$
g\left(\frac{z^2}2 + \tilde V(x) \right)  <  g\left(\frac{z^2}2 \right).
$$
This is a contradiction; we deduce that $\tilde V =0$. This implies that 
$\tilde E = 0$ and thus we get the remarkable identity for the 
electric field
\begin{equation}
\label{expressionE}
E(t,x) = \partial_t \bar v(t,x) + \bar v(t,x) \partial_x \bar v(t,x).
\end{equation}
The ODE for $(X,W)$ is now trivial, and we get a distribution
$f(t,\cdot,\cdot)$ depending only on $t$ and $|w|$, i.e.
$$ f(t,x,\bar v(t,x) + w) = g\left(t,\frac{w^2}{2}\right)$$
or equivalently
\begin{equation}
\label{expressionf}
f(t,x,v) = g\left(t, \frac{|v - \bar v(t,x)|^2}{2}\right) \,.
\end{equation}

\bigskip
{\sl  Step 4. Consequences in the original Vlasov equation 
\eqref{kineuler-intro}.}\\
Inserting \eqref{expressionf} in \eqref{kineuler-intro}, it comes
$$
 \partial_t g\left(t, \frac{|v - \bar v(t,x)|^2}{2}\right)  + 
 \bigl[ E - \partial_t 
\bar v  - v \, \partial_x \bar v \,\bigr]  
\bigl(v - \bar v(t,x) \bigr) \partial_u 
g\left(t, \frac{|v - \bar v(t,x)|^2}{2}\right)  = 0.
$$
where $\partial_u g$ denotes the derivative with respect to the second variable of $g$. Using \eqref{expressionE}, we deduce
$$
 \partial_t g\left(t, \frac{|v - \bar v(t,x)|^2}{2}\right)  -  |\overline v -v 
|^2 \partial_x \bar v \,
 \partial_u g\left(t, \frac{|v - \bar v(t,x)|^2}{2}\right)  =  0,
 $$
 from which we can finally write
\begin{equation}
\label{eq:g}
\partial_t g -  2 \,u \,\partial_x \bar v \,
 \partial_u g  =  0.
\end{equation}
We can now integrate this equation against the measure $\frac{du}{\sqrt{u}}$,
which yields (after an integration by parts)
$$
\frac{d}{dt} \int g \frac{du}{\sqrt{u}}+  \partial_x \bar v \int  g \frac{du}{\sqrt{u}}=0.
$$
Recall that by \eqref{expressionf}, since $\rho=1$, we have $\int g \frac{du}{\sqrt{u}}=1$. We therefore deduce that for all $t,x$,
$$\partial_x \bar v = 0.$$
The equation \eqref{eq:g} also becomes $\partial_t g=0$ so that
we can rewrite $f$ as
$$f(t,x,v) = g\left( \frac{|v - \bar v(t)|^2}{2}\right).$$ 
But in this case $\bar v (t)  = \int f \,dxdv =P$ which is preserved by the
equation and therefore is a constant. We finally get
$$ f(t,x,v) =g\left(\frac{|v - \bar v|^2}{2}\right)$$
and set $\varphi(\cdot) = g(-\cdot)$, which concludes the proof of  \eqref{goal-EQ}.
\end{proof}

\section{Construction of BGK waves: Proof of Theorem \ref{thm:1}} \label{sec:proof}
\label{sec-sta1}

In this section, we give a proof of Theorem \ref{thm:1}. We will restrict ourselves for simplicity to the case where $f^-_0 =0$, but the general case may be handled exactly in the same way.
In our model, the Hamiltonian (or energy)
is given by 
\[
E(x,v) = \frac{v^2}2 + V(x),
\]
and a solution of \eqref{eq:vla1D} is constant on the trajectories of the associated Hamiltonian
system, which are the connected components of $\{ (x,v)  |
E(x,v) = h\}$, except in the case where the potential remains constant on
a whole interval. 

We consider potentials $V$ reaching their maximum at $x=0$ and $x=1$, i.e.
\begin{equation} \label{condV}
V(0) =V(1)=0, \qquad \sup_{x \in [0,1]} V(x) \le  0.
\end{equation}
There exists at least one solution satisfying the above condition : the
homogeneous equilibrium 
\begin{equation} \label{homeq}
f(x,v)  = 
\begin{cases}
f_0^+(v)  & \text{if }  v \ge 0, \\
0 & \text{else},
\end{cases}
\end{equation}
 together with the constant potential $V=0$. 
Under the condition \eqref{condV}, if $\frac{v^2}2 + V(x) \ge 0$ and $v >0$,
then the level line passing through $(x,v)$ crosses the incoming boundary $\{ 0 \}
\times \R^+$ at $(0, \sqrt{v^2 + 2 V(x) })$. Therefore, the value of $f$ at
that point is given by
$$
f(x,v) = f_0^+ \bigl( \sqrt{v^2 + 2V(x) } \bigr), \qquad
\text{if } \quad
v  \ge \sqrt{ -2 V(x) }.
$$
If $v <0$ and $\frac{v^2}2 + V(x) \ge 0$, the level line crosses the incoming
boundary $\{ 1 \} \times \R^-$, and this leads to
$$
f(x,v) = 0, \qquad \text{if } \quad v  \le - \sqrt{-2 V(x) }.
$$
In between, for $|v| \le \sqrt{-2 V(x)}$, the particles are ``trapped'' in the
sense that they do not have
a sufficient energy to reach one of the boundary, and as a consequence, their
density is not fixed by the boundary condition. In that region, we will assume
that the density of $f$ is constant on the level lines of $E$, even if they are
not connected. But this is not a restriction since
 there is   only one density profile for the trapped particles
that leads to a solution of~\eqref{eq:vla1D} when
$V(x)$ is known, as we shall see later. In order to be consistent with
the previous discussion, we will use the notation
$$
f(x,v) = f_T \bigl(\sqrt{- v^2 - 2V(x) } \bigr),
\qquad
\text{if } \quad
|v|  < \sqrt{-2 V(x) }.
$$
The unknown function $f_T$ is defined on the interval $[0,\sqrt{-2 V_{min}}]$,
and the subscript $T$ stands for ``trapped''.
With these notation, the neutrality condition $ 1 = \int f \,dv $ now reads
$$
1  =  2 \int_0^{\sqrt{-2 V(x)}}  f_T  \bigl(\sqrt{ - v^2 - 2V(x) }
\bigr) \,dv   + \int_{\sqrt{-2 V(x)}}^{+\infty} 
f_0^+ \bigl( \sqrt{v^2 + 2V(x) } \bigr) \,dv.
$$
After a change of variable, it may be rewritten
$$
1 =  2 \int_0^{\sqrt{-2 V(x)}}  f_T (u) \frac{u\,du}{\sqrt{ - u^2 - 2V(x) }}
+ \int_0^{+\infty} f_0^+ ( u ) \frac{u\,du}{\sqrt{u^2 - 2V(x) }}.
$$
In order to get a solution, we only need to ensure that this condition is satisfied
for any $x \in [0,1]$. This is precisely stated in the following lemma.
\begin{lem}
\label{lem:1}
Assume that $f_0^+ \in  L^1$ satifies $\int_0^{+\infty} f_0^+(v) \,dv =1$, and that there exists a  
non-negative mesurable function $f_T : (0,+\infty) \rightarrow
\R^+$ such that
\begin{equation} \label{condfT}
\forall\,   r >0, \quad 
 2 \int_0^r  f_T (u) \frac{u\,du}{\sqrt{ r^2 - u^2 }} = 1 
 - \int_0^{+\infty} f_0^+ ( u ) \frac{u\,du}{\sqrt{ r^2 +  u^2  }}.
\end{equation}
Then, for any continuous potential $V : [0,1] \rightarrow \R^-$ satisfying $V(0)=V(1)=0$, the function $f \in L^1$ defined by
\begin{equation}\label{def:solgene}
 f(x,v) =
 \begin{cases}
  f_0^+ \bigl( \sqrt{v^2 + 2V(x) } \bigr) &
  \text{if} \quad v  \ge \sqrt{ -2 V(x) }, \\
  0 & \text{if} \quad v  \le - \sqrt{-2 V(x) }, \\
  f_T \bigl(- \sqrt{- v^2 - 2V(x) } \bigr) &
\text{if} \quad |v|  < \sqrt{-2 V(x) },
 \end{cases}
\end{equation}
is a solution of~\eqref{eq:vla1D}. 
\end{lem}

\begin{proof}[Proof of Lemma \ref{lem:1}]
 The proof is straightforward. The condition~\eqref{condfT} ensures that
$\rho(x)=1$, and in particular it implies that~\eqref{def:solgene} defines a function $f  \in L^1$.  The fact that the function $f$  solves~\eqref{eq:vla1D}
in the sense of distributions follows since $f$ is a function of the
energy $\frac{v^2}2 + V(x)$. It can be checked using some smooth test function
$\varphi$, and the change of variables $(x,v) \mapsto \bigl( x, v^2 + 2 V(x) 
\bigr)$ in several regions.
\end{proof}

So in order to construct solutions to~\eqref{eq:vla1D}, it suffices to
find a function $f_T$ satisfying~\eqref{condfT}. This is
done in the following proposition.
\begin{prop} \label{prop:exisfT}
Assume that $f_0^+ \in L^1$ with $\int f^+_0 =1$. Then the 
function $f_T$, defined on $(0,+\infty)$ as follows
\begin{equation} \label{def:fT}
f_T(u) := \frac1 \pi \int_0^\infty f_0^+(v) \frac{u\,v\,dv}{(u^2 + 
v^2)^{\frac32}},
\end{equation}
satisfies~\eqref{condfT} for all $r >0$.  And for any $\bar r >0$,
it is the unique function that satisfies~\eqref{condfT} for all $r \in [0, 
\bar r]$.
Moreover,  if in addition $f_0^+$ is 
continuous at $0$, then $\lim_{u \rightarrow 0} f_T(u) = f_0^+(0)$.
\end{prop}

The theorem~\ref{thm:1} is then a consequence of Lemma~\ref{lem:1} and
Proposition~\ref{prop:exisfT}.

\begin{proof}[Proof of Proposition~\ref{prop:exisfT}] \mbox{}

{\sl Step 1. $f_T$ is a solution of~\eqref{condfT}.}
First remark that the function $f_T$ defined 
by~\eqref{def:fT} may be rewritten
\begin{align*}
f_T(u) =  \frac1 {2\pi}\int_0^\infty f_0^+( u v) \frac{dv}{(1 + 
v^2)^{\frac32}}. 
\end{align*}
This allows to prove that $\lim_{u \rightarrow 0} f_T(u) = f_0^+(0)$ when $f_0^+$ is continuous at $0$.
Next, denote by $g$ the function defined on $\R^+$ by the r.h.s. 
of~\eqref{condfT}:
\begin{equation} \label{def:g}
g(r) := 1  - \int_0^{+\infty} f_0^+ ( u ) \frac{u\,du}{\sqrt{ r^2 +  u^2  }}.
\end{equation}
Its derivative is given for $r>0$ by
\begin{equation*}
g'(r) = r \int_0^\infty \frac{f_0^+(u) \,u\,du}{(r^2 + u^2)^{\frac32}}
= \int_0^\infty \frac{f_0^+(r u) \,u\,du}{(1 + u^2)^{\frac32}}.
\end{equation*}
Thus $g'$ is positive and thus in $L^1$, and  with the second expression, we see 
that $\lim_{r \rightarrow 0} g'(r) = f_0^+(0)$ if $f_0^+$ is continuous at $0$. 
The condition~\eqref{condfT} maybe rewritten
\begin{equation*} 
\forall\,   r \in [0, \bar r ], \quad 
\int_0^r  f_T (u) \frac{u\,du}{\sqrt{ r^2 - u^2 }} = \frac12 \,  g(r).
\end{equation*}
Next using that for all $a  > b \ge0$,
\begin{equation} \label{simpleeq}
\int_a^b \frac{u \,du}{\sqrt{(b^2-u^2)(u^2-a^2)}} = \frac\pi 2 ,
\end{equation}
we obtain that for all $r>0$
\begin{align*}
\int_0^r  \biggl( \frac1 \pi \int_0^u \frac{g'(s)\,ds}{\sqrt{u^2-s^2}}  \biggr) 
\frac{u\,du}{\sqrt{ r^2 - u^2 }} = \frac12 \int_0^r g'(s) \,ds = \frac12 \, g(r).
\end{align*}
This means that 
\begin{align*}
f_T(u) & :=  \frac1 \pi \int_0^u \frac{g'(s)\,ds}{\sqrt{u^2-s^2}} 
= 
\frac1\pi \int_0^u \int_0^\infty \frac{f_0^+(v) \, s\,v \,ds 
\,dv}{\sqrt{(u^2-s^2)(s^2 + v^2)^3}} \\
& = 
\frac1{2\pi} \int_0^1 \int_0^\infty  \frac{f_0^+( u v)  \,v  \,ds 
\,dv}{\sqrt{(1-s)(s + v ^2)^3}},
\end{align*}
is a solution to~\eqref{condfT}. The conclusion follows from the equality
$$
\int_0^1 \frac{\,ds}{\sqrt{(1-s)(s+x)^3}} = \biggl[ 
- \frac{2 \sqrt{1-s}}{(1+x)\sqrt{x+s}} 
\biggr]_0^1
= \frac 2 {(1+x)^{\frac 32}}.
$$

\bigskip
{\sl Step 2. Uniqueness.}
Assume that $g_T : [0, \bar r] \rightarrow \R$ is such that
$$
\forall r \in [0, \bar r] , \qquad
\int_0^r g_T(u) \frac{ u\,du}{\sqrt{r^2-u^2}} =0.
$$
Let $r_1 \in [0,r]$.
Then, multiplying the previous equation by $\frac r{\sqrt{r_1^2-r^2}}$, and 
integrating on the interval $[0,r_1]$ , and using~\eqref{simpleeq}, we get
$$
\frac \pi 2 \int_0^{r_1} g_T(u) u\,du =0.
$$
Since it holds for any $r_1 \in [0, \bar r]$, it implies that $g_T = 0$ on that 
interval.

\end{proof}

%
%
\section{The case of the Vlasov-Poisson equation for ions}
\label{sec:mod}
In the last part of this paper, we focus on the quasineutral limit of the Vlasov-Poisson equation for ions:
\begin{equation}
\label{quasi2}
\left\{
    \begin{array}{ll}
 \ds  \partial_t f_\epsilon + v\partial_x f_\epsilon - \partial_x V_\ep \partial_v f_\epsilon =0,  \\
\ds   \alpha V_\epsilon -\epsilon^2 \,\partial_x^2 V_\epsilon = \rho_\epsilon 
-1,
    \end{array}
  \right.
\end{equation}
where $\alpha >0$. We add an initial condition $f_{0,
\epsilon} \in  L^1$ such that $f_{0,\ep}\geq 0$,  $\int 
f_{0,\ep} dvdx =1$.  As already said in the
introduction, this allows to describe the dynamics of ions in a plasma, in a
background of  ``adiabatic'' electrons, i.e. electrons which  
instantaneously  reach a thermodynamic equilibrium.

\begin{rque}
In \eqref{quasi2}, there is a parameter $\alpha >0$, which comes
from the fact that this model is only a linearization of the ``physical''
equations, in which the density of electrons follows a Maxwell-Boltzmann law and
the Poisson equation thus reads:
$$
-\epsilon^2 \, \partial_x^2 V_\epsilon = \rho_\epsilon - e^{-\alpha V_\epsilon}.
$$
The linearization consists then in writing $e^{-\alpha V_\epsilon} \approx 1 - \alpha V_\eps$, which yields \eqref{quasi2}.
\end{rque}

The scaled physical energy of this system reads:
\begin{equation}
\label{energy21}
\mathcal{E}_\epsilon(t)= \frac{1}{2} \int f_\epsilon \vert v \vert^2 dv dx  +  
\frac{\alpha}{2}  \int V_\epsilon^2 dx + \frac{\epsilon^2}{2} \int \vert 
\partial_x V_\epsilon \vert^2 dx.
\end{equation}

Assume now to simplify that $\alpha=1$. We can proceed as in the introduction
and formally obtain in the limit $\eps\to 0$ the Vlasov-Dirac-Benney equation
\begin{equation}
\label{kinSW}
\left\{
    \begin{array}{ll}
  \partial_t f + v\partial_x f - \partial_x V \partial_v f =0,  \\
V=  \rho-1.\\
    \end{array}
  \right.
\end{equation}
 We observe that the energy associated to this system reads :
\begin{equation}
\label{energy22}
\mathcal{E}(t)= \frac{1}{2} \int f \vert v \vert^2 dv dx + \frac{1}{2}  \int \rho^2 dx.
\end{equation} 
\begin{rque} Note that this can be seen a kinetic version of the shallow water
system (or isentropic gas dynamics with $\gamma=2$). Indeed, for monokinetic
profiles, that is 
$$
f(t,x,v)=\rho(t,x) \mathbbm{\delta}_{v=u(t,x)},
$$
we get the one-dimensional shallow water system:
\begin{equation}
\label{SW}
\left\{
    \begin{array}{ll}
  \partial_t \rho + \partial_x (\rho u)=0,  \\
\partial_t u + u\partial_x u + \partial_x \rho=0 .\\
    \end{array}
  \right.
\end{equation}
As a matter of fact, the derivation of \eqref{SW} from \eqref{quasi2} for
monokinetic data was performed in \cite{HK}.
\end{rque}

We now explain how to adapt the results which have been proved in this
paper.

\subsection{The Penrose criterion.}

We shall start by explaining what is the right Penrose criterion in the context
of the Vlasov equation
\begin{equation}
\label{eq:VP1D-mod}
\left\{
    \begin{array}{ll}
 \ds  \partial_t f + v\partial_x f - \partial_x V \partial_v f =0,  \\
\ds   \alpha V -\partial_x^2 V = \rho -1.
    \end{array}
  \right.
\end{equation}
We define the ``$\alpha$-Penrose instability criterion'' as follows.
\begin{deft} 
\label{def:Penmod}
We say that an homogeneous profile $\mu(v)$, such that $\int \mu \, dv=1$, satisfies the $\alpha$-Penrose instability 
criterion if there exists a local minimum point $\bar v$ of $\mu$ such that the following inequality holds 
\begin{equation} \label{def:Pen-first-mod}
\int_{\R} \frac{ \mu(v) - \mu(\bar v)} {(v-\bar v)^2} \, dv >   \alpha.
\end{equation}
If the local minimum is flat, i.e. is reached on an interval $[\bar v_1, \bar v_2]$, then~\eqref{def:Pen} has to be satisfied for all $\bar v \in [\bar v_1, \bar v_2]$.  
\end{deft}
\begin{rque}
Note that if $\alpha=0$, we recover the same instability conditions of the introduction.
\end{rque}

Exactly as for the case $\alpha=0$, we may obtain the exact analogue of Proposition \ref{prop:Penlim}, which is a key point in Theorem \ref{thmGrenier-revisited-mod} below.

\subsection{Unstable Case.}
The instability result we are able to prove is the same as for $\alpha=0$. 
\begin{thm} 
\label{thmGrenier-revisited-mod}
Let $\mu(v)$ be a smooth positive profile satisfying the Penrose instability criterion 
of Definition~\ref{def:Penmod} and the $\delta$-condition of Definition~\eqref{cond-pos}.  For any $N>0$ and $s>0$, there exists a 
sequence of non-negative initial data $(f_{0,\eps})$ such that 
$$
\| f_{\eps,0}- \mu\|_{W^{s,1}_{x,v}} \leq \eps^N,
$$
and denoting by $(f_\eps)$ the sequence of solutions
to \eqref{quasi2} with initial data $(f_{0,\eps})$, the following holds:.

\begin{enumerate}[i)]
\item {\bf $L^1$ instability for the macroscopic observables:} the density $\rho_\ep := \int f_\eps \, dv$, and the electric field $E_\ep = - \pa_x V_\ep$. For all $\alpha\in [0,1)$, we have
\begin{equation}
\label{insta:macromod}
\liminf_{\eps \rightarrow 0} \sup_{t \in [0,\eps^\alpha]} \left\| \rho_\eps(t) - 1 \right\|_{L^1_{x}} >0,
\qquad
\liminf_{\eps \rightarrow 0} \sup_{t \in [0,\eps^\alpha]} {\eps}\left\| E_\eps \right\|_{L^1} >0.
\end{equation}

\item {\bf Full instability for the distribution function:} for any $r \in \Z$, we have
\begin{equation}
\label{insta:fullmod}
\liminf_{\eps \rightarrow 0} \sup_{t \in [0,\eps^\alpha]} \left\| f_\eps(t) - \mu \right\|_{W^{r,1}_{x,v}} >0
\end{equation}
\end{enumerate}
\end{thm}

The same proof as that of Theorem \ref{thmGrenier-revisited} holds, \emph{mutatis
mutandis}: we only have to  switch the Poisson equations.

\subsection{Stable Case.}

The following holds for any value of $\alpha$, and therefore we consider here for simplicity that $\alpha=1$.
We restrict ourselves only on the well-prepared case.
Our stability theorem for \eqref{quasi2} goes as follows:

\begin{thm}
\label{thm-mod}
Let $\mu$ be a $S$-stable stationary solution to \eqref{kinSW} of the 
form given in~\eqref{condphi}. Assume that there exists $\eta>0$, such that  $\mu$ satisfies 
\begin{equation}
\label{assumption-mu-mod}
\int \mu(v) (1+ v^{2 + \eta})\,dv<+\infty.
\end{equation}
For all $\epsilon >0$, let $(f_\epsilon,V_\epsilon)$ be the global weak solution in the sense of Arsenev
to \eqref{quasi2}, with initial datum $f_{0,\epsilon}$ 
and define the 
``modulated energy''
\begin{equation} \label{def:Lep-mod}
\LL_\ep[f_\ep] := 
H_Q(f_\ep)  + \frac {\ep^2} 2 \int ( \partial_x V_\epsilon )^2 dx + \frac {1} 2 \int V_\epsilon^2 dx  .
\end{equation}
Then, $\LL_\ep$ is a Lyapunov functional in the sense that
$$
\forall t \in \R^+, \qquad \LL_\ep[f_\ep(t)] = \LL_\ep[f_{0,\ep}].
$$
\end{thm}
We thus see that the only thing to do is to adapt the definition of the modulated energy \eqref{def:Lep-mod} according to the energy \eqref{energy21} of \eqref{quasi2}. Then, the proof is exactly the same as that of Theorem~\ref{thm}, and therefore we omit it.

\bigskip

The adaptation the stability result for ill-prepared initial 
data of Theorem~\ref{thm-IP} requires more than a simple rephrasing, and this does not fall within the scope of this paper.

\subsection{Construction of the associated BGK waves.}
We consider now the boundary value problem
\begin{equation} \label{eq:vla1D2}
\begin{cases}
\ds v  \, \partial_x f - \frac1 \alpha \partial_x \rho\,   \partial_v f = 0, \\
\ds \rho = \int f(x,v) \,dv, 
\end{cases} 
\end{equation}
on the space $\Omega = [0,1]\times \R$. The incoming boundary conditions are
given by~\eqref{eq:BC}.
This corresponds to the stationary equations associated to the
 Vlasov-Dirac-Benney equation~\eqref{quasi2}, with boundary
conditions. An adaptation of the proof of Theorem~\ref{thm:1}  leads to the
following theorem.
\begin{thm} \label{thm:2}
Assume that $f_0^\pm : \R^+ \rightarrow \R^+$  is  are nonnegative and measurable functions such that
$\int_0^{+\infty} \bigl( f_0^+(v) + f_0^-(-v) \bigr) \,dv =1$. Define $f_T$
on $(0,+\infty)$ by
\begin{equation}  \label{def:fTdouble2}
f_T(u) :=  - \frac {\alpha \,u }{2\pi }  + \frac1 \pi \int_0^\infty \bigl( f_0^+(v) + f^-_0(-v)\bigr) \frac{u\,v\,dv}{(u^2 + 
v^2)^{\frac32}},
\end{equation}
and denote $\bar u := \inf \{u>0, \; s.t.  \; f_T(u) <0  \}$, which is in any case bounded by $\sqrt{\frac 2 \alpha}$. Then,  for any continuous potential  with values in  $ Â¬â  \bigl[ - \frac{\bar u^2}2, 0 \bigr]$ satisfying $V(0)=V(1)=0$,
the function $f$ defined by~\eqref{def:fstat}
together with $V$ gives a solution of~\eqref{eq:vla1D2} in the sense of
distributions.
 Moreover, any solution with $V$ nonpositive and vanishing at the boundary 
is of the above form. 
\end{thm}

The bound on $\bar u$ comes from a straightforward a priori bound of the right 
hand side of~\eqref{def:fTdouble2}: use the elementary inequality 
$(u^2+v^2)^{-\frac 32} \le u^{-2} v^{-1}$ and $\int_0^{+\infty} \bigl( f_0^+(v) 
+ f_0^-(-v) \bigr) \,dv =1$.

When $f_0^+(\cdot) + f_0^-(-\cdot)$ is continuous at $0$, we have also that $\lim_{u \rightarrow 0 } f_T(u) = 
f_0^+(0)+ f_0^-(0)$, so that it is clear that $\bar u$ is strictly positive if  $f_0^+(0) + f_0^-(0)$ is.
If $f_T(0)=0$, then  $\bar u >0$  when $f_T'(u) >0$, that is
$$
\int_0^\infty \bigl( f_0^+(v) + f^-_0(-v)\bigr) \frac{dv}{  
v^2} > \frac \alpha 2.
$$

The proof of~\ref{thm:1} can be adapted without difficulty to this case, with a potential given by  $V
= \frac{\rho -1}\alpha$. For instance, the definition of $g$ in~\eqref{def:g} should be replaced by
\begin{equation*} 
g(r) := 1  - \alpha \frac{r^2}2 + \int_0^{+\infty} f_0^+ ( u ) \frac{u\,du}{\sqrt{ r^2 +  u^2  }},
\end{equation*}
and the additionnal term $- \alpha \frac{r^2}2$ leads after some straightforward computations to the additional term $- \frac {\alpha \,u }{2\pi } $ in~\eqref{def:fTdouble2}.


\begin{thebibliography}{10}
\bibliographystyle{plain}

\bibitem{Ar65}
V.~I. Arnold.
\newblock On conditions for non-linear stability of plane stationary
  curvilinear flows of an ideal fluid.
\newblock {\em Dokl. Akad. Nauk SSSR}, 162:975--978, 1965.

\bibitem{Ar66}
V.~I. Arnold.
\newblock An a priori estimate in the theory of hydrodynamic stability.
\newblock {\em Izv. Vys\v s. U\v cebn. Zaved. Matematika}, 1966(5 (54)):3--5,
  1966.

\bibitem{Ar}
A.~A. Arsenev.
\newblock Existence in the large of a weak solution of {V}lasov's system of
  equations.
\newblock {\em \v Z. Vy\v cisl. Mat. i Mat. Fiz.}, 15:136--147, 276, 1975.

\bibitem{Bardos}
C.~Bardos.
\newblock {About a Variant of the $1d$ Vlasov equation, dubbed
  ``Vlasov-Dirac-Benney'' Equation}.
\newblock {\em S{\'e}minaire Laurent Schwartz - EDP et applications}, 15:21 p.,
  2012-2013.

\bibitem{BB}
C.~Bardos and N.~Besse.
\newblock {The Cauchy problem for the Vlasov-Dirac-Benney equation and related
  issued in fluid mechanics and semi-classical limits}.
\newblock {\em Kinet. Relat. Models}, 6(4):893--917, 2013.

\bibitem{BarNou}
C.~Bardos and A.~Nouri.
\newblock {A Vlasov equation with Dirac potential used in fusion plasmas}.
\newblock {\em J. Math. Phys.}, 53(11):115621--115621, 2012.

\bibitem{BR}
J.~Batt and G.~Rein.
\newblock A rigorous stability result for the {V}lasov-{P}oisson system in
  three dimensions.
\newblock {\em Ann. Mat. Pura Appl. (4)}, 164:133--154, 1993.

\bibitem{BMM}
J.~Bedrossian, N.~Masmoudi, and C.~Mouhot.
\newblock {Landau damping: paraproducts and Gevrey regularity}.
\newblock {\em arXiv preprint arXiv:1311.2870}, 2013.

\bibitem{BAD03}
N.~Ben~Abdallah and J.~Dolbeault.
\newblock Relative entropies for kinetic equations in bounded domains
  (irreversibility, stationary solutions, uniqueness).
\newblock {\em Arch. Ration. Mech. Anal.}, 168(4):253--298, 2003.

\bibitem{BL}
J.~Bergh and J.~L{\"o}fstr{\"o}m.
\newblock {\em Interpolation spaces. {A}n introduction}.
\newblock Springer-Verlag, Berlin, 1976.
\newblock Grundlehren der Mathematischen Wissenschaften, No. 223.

\bibitem{BGK}
I.~B. Bernstein, J.~M. Greene, and M.~D. Kruskal.
\newblock Exact non-linear plasma oscillations.
\newblock {\em Phys. Rev. (2)}, 108:546--550, 1957.

\bibitem{Besse}
N.~{Besse}.
\newblock {On the waterbag continuum.}
\newblock {\em {Arch. Ration. Mech. Anal.}}, 199(2):453--491, 2011.

\bibitem{BFJJ}
M.~{Bossy}, J.~{Fontbona}, P.-E. {Jabin}, and J.-F. {Jabir}.
\newblock {Local existence of analytical solutions to an incompressible
  Lagrangian stochastic model in a periodic domain}.
\newblock {\em Comm. Partial Differential Equations},
  http://doi.org/10.1080/03605302.2013.786727, To appear.

\bibitem{Br89}
Y.~{Brenier}.
\newblock {A Vlasov-Poisson type formulation of the Euler equations for perfect
  incompressible fluids}.
\newblock {\em {Rapport de recherche INRIA}}, 1989.

\bibitem{Br99}
Y.~Brenier.
\newblock Homogeneous hydrostatic flows with convex velocity profiles.
\newblock {\em Nonlinearity}, 12(3):495--512, 1999.

\bibitem{Br00}
Y.~Brenier.
\newblock Convergence of the {V}lasov-{P}oisson system to the incompressible
  {E}uler equations.
\newblock {\em Comm. Partial Differential Equations}, 25(3-4):737--754, 2000.

\bibitem{Br03}
Y.~Brenier.
\newblock Remarks on the derivation of the hydrostatic {E}uler equations.
\newblock {\em Bull. Sci. Math.}, 127(7):585--595, 2003.

\bibitem{BG}
Y.~Brenier and E.~Grenier.
\newblock Limite singuli\`ere du syst\`eme de {V}lasov-{P}oisson dans le
  r\'egime de quasi neutralit\'e: le cas ind\'ependant du temps.
\newblock {\em C. R. Acad. Sci. Paris S\'er. I Math.}, 318(2):121--124, 1994.

\bibitem{CCD}
M.~J. C{\'a}ceres, J.~A. Carrillo, and J.~Dolbeault.
\newblock Nonlinear stability in {$L^p$} for a confined system of charged
  particles.
\newblock {\em SIAM J. Math. Anal.}, 34(2):478--494 (electronic), 2002.

\bibitem{Deg}
P.~Degond.
\newblock Spectral theory of the linearized {V}lasov-{P}oisson equation.
\newblock {\em Trans. Amer. Math. Soc.}, 294(2):435--453, 1986.

\bibitem{Degond&all}
P.~Degond, F.~Deluzet, L.~Navoret, A.-B. Sun, and M.-H. Vignal.
\newblock Asymptotic-preserving particle-in-cell method for the
  {V}lasov-{P}oisson system near quasineutrality.
\newblock {\em J. Comput. Phys.}, 229(16):5630--5652, 2010.

\bibitem{DM87}
R.~J. DiPerna and A.~J. Majda.
\newblock Oscillations and concentrations in weak solutions of the
  incompressible fluid equations.
\newblock {\em Comm. Math. Phys.}, 108(4):667--689, 1987.

\bibitem{Gal}
I.~Gallagher.
\newblock R\'esultats r\'ecents sur la limite incompressible.
\newblock {\em Ast\'erisque}, 299:Exp. No. 926, vii, 29--57, 2005.
\newblock S{\'e}minaire Bourbaki. Vol. 2003/2004.

\bibitem{GVD}
D.~G{\'e}rard-Varet and E.~Dormy.
\newblock On the ill-posedness of the {P}randtl equation.
\newblock {\em J. Amer. Math. Soc.}, 23(2):591--609, 2010.

\bibitem{GVM}
D.~Gerard-Varet and N.~Masmoudi.
\newblock {Well-posedness for the Prandtl system without analyticity or
  monotonicity}.
\newblock {\em arXiv preprint arXiv:1305.0221}, 2013.

\bibitem{GVN}
D.~G{\'e}rard-Varet and T.~Nguyen.
\newblock Remarks on the ill-posedness of the {P}randtl equation.
\newblock {\em Asymptot. Anal.}, 77(1-2):71--88, 2012.

\bibitem{Gr95}
E.~Grenier.
\newblock Defect measures of the {V}lasov-{P}oisson system in the quasineutral
  regime.
\newblock {\em Comm. Partial Differential Equations}, 20(7-8):1189--1215, 1995.

\bibitem{Gr96}
E.~Grenier.
\newblock Oscillations in quasineutral plasmas.
\newblock {\em Comm. Partial Differential Equations}, 21(3-4):363--394, 1996.

\bibitem{Gr99-1}
E.~Grenier.
\newblock Limite quasineutre en dimension 1.
\newblock In {\em Journ\'ees ``\'{E}quations aux {D}\'eriv\'ees {P}artielles''
  ({S}aint-{J}ean-de-{M}onts, 1999)}, pages Exp.\ No.\ II, 8. Univ. Nantes,
  Nantes, 1999.

\bibitem{Gr99-2}
E.~Grenier.
\newblock On the derivation of homogeneous hydrostatic equations.
\newblock {\em M2AN Math. Model. Numer. Anal.}, 33(5):965--970, 1999.

\bibitem{Gr}
E.~Grenier.
\newblock On the nonlinear instability of {E}uler and {P}randtl equations.
\newblock {\em Comm. Pure Appl. Math.}, 53(9):1067--1091, 2000.

\bibitem{GN}
Y.~Guo and T.~Nguyen.
\newblock A note on {P}randtl boundary layers.
\newblock {\em Comm. Pure Appl. Math.}, 64(10):1416--1438, 2011.

\bibitem{GS}
Y.~Guo and W.~A. Strauss.
\newblock Nonlinear instability of double-humped equilibria.
\newblock {\em Ann. Inst. H. Poincar\'e Anal. Non Lin\'eaire}, 12(3):339--352,
  1995.

\bibitem{GS98}
Y.~Guo and W.~A. Strauss.
\newblock Unstable {BGK} solitary waves and collisionless shocks.
\newblock {\em Comm. Math. Phys.}, 195(2):267--293, 1998.

\bibitem{HK}
D.~Han-Kwan.
\newblock Quasineutral limit of the {V}lasov-{P}oisson system with massless
  electrons.
\newblock {\em Comm. Partial Differential Equations}, 36(8):1385--1425, 2011.

\bibitem{Hau2D}
M.~Hauray.
\newblock On two-dimensional {H}amiltonian transport equations with {$L_{\rm
  loc}^p$} coefficients.
\newblock {\em Ann. Inst. H. Poincar\'e Anal. Non Lin\'eaire}, 20(4):625--644,
  2003.

\bibitem{HauX}
M.~Hauray.
\newblock {Mean field limit for the one dimensional Vlasov-Poisson equation}.
\newblock {\em S{\'e}minaire Laurent Schwartz - EDP et applications}, 21:16 p.,
  2012-2013.

\bibitem{HMRW}
D.~D. {Holm}, J.~E. {Marsden}, T.~{Ratiu}, and A.~{Weinstein}.
\newblock {Nonlinear stability of fluid and plasma equilibria}.
\newblock {\em Phys. Rep.}, 123:1--2, July 1985.

\bibitem{JabNou}
P.-E. Jabin and A.~Nouri.
\newblock Analytic solutions to a strongly nonlinear {V}lasov equation.
\newblock {\em C. R. Math. Acad. Sci. Paris}, 349(9-10):541--546, 2011.

\bibitem{LZ1}
Z.~Lin and C.~Zeng.
\newblock Small {BGK} waves and nonlinear {L}andau damping.
\newblock {\em Comm. Math. Phys.}, 306(2):291--331, 2011.

\bibitem{LZ2}
Z.~Lin and C.~Zeng.
\newblock {Small BGK waves and nonlinear Landau damping (higher dimensions)}.
\newblock {\em Arxiv preprint arXiv:1106.4368}, 2011.

\bibitem{Lio96}
P.-L. Lions.
\newblock {\em Mathematical topics in fluid mechanics. {V}ol. 1}, volume~3 of
  {\em Oxford Lecture Series in Mathematics and its Applications}.
\newblock The Clarendon Press Oxford University Press, New York, 1996.
\newblock Incompressible models, Oxford Science Publications.

\bibitem{MP86}
C.~Marchioro and M.~Pulvirenti.
\newblock A note on the nonlinear stability of a spatially symmetric
  {V}lasov-{P}oisson flow.
\newblock {\em Math. Methods Appl. Sci.}, 8(2):284--288, 1986.

\bibitem{Mas}
N.~Masmoudi.
\newblock From {V}lasov-{P}oisson system to the incompressible {E}uler system.
\newblock {\em Comm. Partial Differential Equations}, 26(9-10):1913--1928,
  2001.

\bibitem{MW2}
N.~Masmoudi and T.~K. Wong.
\newblock {Local-in-Time Existence and Uniqueness of Solutions to the Prandtl
  Equations by Energy Methods}.
\newblock {\em arXiv preprint arXiv:1206.3629}, 2012.

\bibitem{MW}
N.~Masmoudi and T.~K. Wong.
\newblock On the {$H^s$} theory of hydrostatic {E}uler equations.
\newblock {\em Arch. Ration. Mech. Anal.}, 204(1):231--271, 2012.

\bibitem{MV}
C.~Mouhot and C.~Villani.
\newblock On {L}andau damping.
\newblock {\em Acta Math.}, 207(1):29--201, 2011.

\bibitem{OS}
O.~A. Oleinik and V.~N. Samokhin.
\newblock {\em Mathematical models in boundary layer theory}, volume~15 of {\em
  Applied Mathematics and Mathematical Computation}.
\newblock Chapman \& Hall/CRC, Boca Raton, FL, 1999.

\bibitem{Pen}
O.~{Penrose}.
\newblock {Electrostatic instability of a uniform non-Maxwellian plasma}.
\newblock {\em Phys. Fluids}, 3:258--265, 1960.

\bibitem{Rei}
G.~Rein.
\newblock Non-linear stability for the {V}lasov-{P}oisson system---the
  energy-{C}asimir method.
\newblock {\em Math. Methods Appl. Sci.}, 17(14):1129--1140, 1994.

\bibitem{SamCaf}
M.~Sammartino and R.~E. Caflisch.
\newblock Zero viscosity limit for analytic solutions, of the {N}avier-{S}tokes
  equation on a half-space. {I}. {E}xistence for {E}uler and {P}randtl
  equations.
\newblock {\em Comm. Math. Phys.}, 192(2):433--461, 1998.

\end{thebibliography}

\end{document}